\documentclass{amsart}

\usepackage{amsmath}
\usepackage{amssymb}
\usepackage{amsthm}
\usepackage{amscd}

\newtheorem{thm}{Theorem}[section]
\newtheorem{prop}[thm]{Proposition}
\newtheorem{lem}[thm]{Lemma}
\newtheorem{cor}[thm]{Corollary}

\theoremstyle{definition}
\newtheorem{dfn}[thm]{Definition}

\theoremstyle{remark}
\newtheorem{rem}[thm]{Remark}
\newtheorem*{acknowledgment}{Acknowledgements}

\newcommand{\C}{\mathbb{C}}
\newcommand{\R}{\mathbb{R}}
\newcommand{\Z}{\mathbb{Z}}

\newcommand{\K}{\mathcal{K}}

\newcommand{\X}{\mathcal{X}}
\newcommand{\Y}{\mathcal{Y}}

\newcommand{\Twist}{\mathfrak{Twist}}

\newcommand{\id}{\mathrm{id}}
\newcommand{\pt}{\mathrm{pt}}

\newcommand{\Fred}{\mathrm{Fred}}
\newcommand{\Gr}{\mathrm{Gr}}
\newcommand{\Vect}{\mathrm{Vect}}
\newcommand{\Pin}{\mathrm{Pin}}

\title
{Freed-Moore $K$-theory}

\author[K. Gomi]{Kiyonori Gomi}

\address{
Department of Mathematics, 
Tokyo Institute of Technology,
2-12-1 Ookayama, Meguro-ku, Tokyo, 152-8551, 
Japan.}

\email{kgomi@math.titech.ac.jp}



\date{}

\begin{document}

\begin{abstract}
The twisted equivariant $K$-theory given by Freed and Moore is a $K$-theory which unifies twisted equivariant complex $K$-theory, Atiyah's `Real' $K$-theory, and their variants. In a general setting, we formulate this $K$-theory by using Fredholm operators, and establish basic properties such as the Bott periodicity and the Thom isomorphism. We also provide formulations of the $K$-theory based on Karoubi's gradations in both infinite and finite dimensions, clarifying their relationship with the Fredholm formulation.
\end{abstract}

\maketitle

\tableofcontents


\section{Introduction}
\label{sec:introduction}


\subsection{Freed-Moore $K$-theory}

The conventional complex $K$-theory $K(X)$ of a topological space $X$, introduced by Atiyah and Hirzebruch \cite{A1}, can be constructed from complex vector bundles on $X$. Since its introduction, it admits various generalizations such as:
\begin{itemize}
\item
Equivariant $K$-theory \cite{Se2}. For this theory to be defined, we consider a space $X$ with an action of a compact Lie group $G$. Then the equivariant $K$-theory $K_G(X)$ can be constructed from $G$-equivariant complex vector bundles on $X$, namely, vector bundles that admit $G$-actions covering the $G$-action on the base space $X$ which induce complex linear transformations on fibers.

\item
Atiyah's `Real' $K$-theory \cite{A2}. For this to be defined, we consider a space $X$ with an action of the cyclic group $\Z_2$ of order $2$ (i.e.\ an involution). The `Real' $K$-theory $KR(X)$ can be constructed from `Real' vector bundles on $X$, namely, complex vector bundles that admit $\Z_2$-actions covering the $\Z_2$-action on the base which induce complex anti-linear transformations on fibers. If one takes up the trivial $\Z_2$-action on $X$, then the $KR$-theory recovers $KO$-theory $KO(X)$, which can be constructed from real vector bundles in the usual sense. A variant of `Real' $K$-theory is Dupont's `Symplectic' (or `Quaternionic') $K$-theory \cite{Dup}, which recovers the $K$-theory of quaternionic vector bundles if the $\Z_2$-action on $X$ is trivial.

\item
Twisted $K$-theory \cite{D-K,R} and its equivariant version \cite{FHT1}. For this to be defined, we need additional data called `twists' $\tau$ and $c$ on $X$ which respectively define cohomology classes in the Borel equivariant cohomology $H^3_G(X; \Z)$ and $H^1_G(X; \Z_2)$. Though is not the case in general \cite{A-Se}, the twisted $K$-theory $K^{(\tau, c)}_G(X)$ can be constructed from $(\tau, c)$-twisted vector bundles. If $G = \Z_2$, then the twisted equivariant $K$-theory recovers a variant of $K$-theory $K_\pm(X)$ introduced in \cite{A-H,W}, by taking $[\tau] \in H^3_{\Z_2}(X; \Z)$ to be trivial and $[c] \in H^1_{\Z_2}(X; \Z_2)$ to be the class given by the identity homomorphism $c : \Z_2 \to \Z_2$.
\end{itemize}

The twisted equivariant $K$-theory of Freed and Moore \cite{F-M} unifies these generalizations. Though is not the most general setting, let us consider a space $X$ with an action of a compact Lie group $G$, a homomorphism $\phi : G \to \Z_2$ and a twist represented by a group $2$-cocycle $\tau \in Z^2_{\mathrm{group}}(G; C(X, U(1))_\phi)$ with local coefficients in the group $C(X, U(1))$ of $U(1)$-valued continuous functions on $X$ regarded as a $G$-module by the $G$-action on $X$ and $\phi$. Then the Freed-Moore $K$-theory\footnote{The terminology is due to the recognition of \cite{F-M} in the community of condensed matter physics. The ideas of the $K$-theory were ``largely developed in collaboration with Hopkins and Teleman'', according to Freed.} ${}^\phi K_G^{\tau}(X)$ is defined by using finite rank twisted equivariant complex vector bundles \cite{F-M}. The key datum is the homomorphism $\phi$ that morally indicates which element of $G$ acts on the fibers of complex vector bundles complex anti-linearly. Thus, if $\phi$ is trivial, then ${}^\phi K_G^{\tau}(X)$ recovers the $G$-equivariant twisted $K$-theory $K_G^{\tau}(X)$. Atiyah's `Real' $K$-theory can be recovered by taking the cyclic group $G = \Z_2$, the identity homomorphism $\phi : G \to \Z_2$ and the trivial twist $\tau$. If we turn on a non-trivial twist $\tau_\phi$ associated to $\phi$, then the Freed-Moore $K$-theory recovers Dupont's `Symplectic' $K$-theory as a twisted $KR$-theory. 

\medskip

The introduction of the Freed-Moore $K$-theory is motivated by recent applications of $K$-theories to the classification of certain quantum systems such as \textit{topological insulators}. A remarkable discovery of Kitaev \cite{Kit} is that the Bott periodicities of $K$-theories explain the so-called periodic table of topological insulators. Some classes of topological insulators involve a symmetry called the time-reversal symmetry, and this serves as the source of the appearance of $KR$-theory. From the viewpoint of condensed matter physics, it is natural to incorporate other symmetries which stem from the symmetries of crystals. This leads one to consider generalizations of $KR$-theory replacing the $\Z_2$-action by an action of a larger group $G$. Some nature of the action of a symmetry on quantum systems naturally produces twisting. Then, as an application of the $K$-theoretic classification scheme of topological insulators, a calculation of equivariant twisted $K$-theory results in a `new' $\Z_2$-phase of topological crystalline insulators (see \cite{SSG1} for example). Therefore one can anticipate calculations of Freed-Moore $K$-theory lead to further discovery of interesting topological insulators, and a calculation results in a novel $\Z_4$-phase \cite{SSG2}.


\subsection{The purposes of this paper}

This paper has two purposes.

\subsubsection{Fredholm formulation}

The conventional $K$-theory can be constructed from vector bundles. However, an analogous construction of twisted $K$-theory based on twisted vector bundles \cite{BCMMS,Kar2} of finite rank fails generally. Instead, an infinite-dimensional formulation is required \cite{A-Se,FHT1,R,TXL}. Thanks to the work of Atiyah and Singer \cite{A-Si}, the infinite-dimensional Fredholm formulation \cite{A-Se,FHT1,R} is useful to define $K$-theory with degree $K^n(X)$ and to prove the Bott periodicity. The $K$-theory of Freed-Moore in \cite{F-M} is formulated by using finite rank (twisted) vector bundles. Its Fredholm formulation is sketched, but seems not fully developed in the literature. One purpose of this paper is therefore to give the Fredholm formulation to lay the foundation of this $K$-theory. 

\medskip

We carry out this formulation under a general setting: Let $\X$ be a local quotient groupoid \cite{FHT1}. Then there is a category $\Phi(\X)$ whose objects are classified by the cohomology $H^1(\X; \Z_2)$. A typical object in $\Phi(\X)$ is a map of groupoids $\X \to \pt//\Z_2$, where $\pt//\Z_2$ is the quotient groupoid associated to the trivial action of $\Z_2$ on a point. Under a choice of $\phi \in \Phi(\X)$, we can introduce a notion of $\phi$-twists. This is a generalization of the notion of twists \cite{FHT1} based on twisted extensions \cite{F-M}. The $\phi$-twists form a category ${}^\phi\Twist(\X)$, and its objects are classified by $H^3(\X; \Z_\phi) \times H^1(\X; \Z_2)$. Then, in a way parallel to the formulation of twisted equivariant complex $K$-theory in \cite{FHT1}, we use skew-adjoint Fredholm families on a twisted $\Z_2$-graded Hilbert bundle to formulate the $K$-theory ${}^\phi K^{(\tau, c) + n}(\X)$, where $(\tau, c)$ represents the data of a $\phi$-twist and $n \in \Z$ is the grading. We can then prove that the $K$-theory enjoys the Bott periodicity
$$
{}^\phi K^{(\tau, c) + n}(\X) \cong {}^\phi K^{(\tau, c) + n + 8}(\X).
$$
A consequence of the periodicity is that ${}^\phi K^{(\tau, c) + n}(\X)$ satisfies the axioms of generalized cohomology theory, formulated suitably in the context of groupoids (Theorem \ref{thm:axiom}). Another consequence is the existence of particular twists $c_{\phi}$ and $\tau_{\phi}$ which have the effects of the degree shift (Theorem \ref{thm:degree_shift})
\begin{align*}
{}^\phi K^{(\tau, c) + c_\phi + n}(\X) 
&\cong 
{}^\phi K^{(\tau, c) + n + 2}(\X), \\
{}^\phi K^{(\tau, c) + \tau_\phi + n}(\X) 
&\cong 
{}^\phi K^{(\tau, c) + n + 4}(\X), \\
{}^\phi K^{(\tau, c) + (\tau_\phi, c_\phi) + n}(\X) 
&\cong 
{}^\phi K^{(\tau, c) + n + 6}(\X).
\end{align*}
The effect of the degree shift by $\tau_\phi$ generalizes the fact \cite{Dup} that `Symplectic' $K$-theory is isomorphic to $KR$-theory with its degree shifted by $4$. 

\medskip

As is mentioned, the Freed-Moore $K$-theory recovers various $K$-theories under specializations. Because of the generality of our formulation, we can introduce a twisted $KR$-theory, which would reproduce the twisted $KR$-theory in \cite{Mou1,Mou2,Mou3}. We anticipate that the Freed-Moore $K$-theory would also reproduces the twisted equivariant $KR$-theory in \cite{Fok1,Fok2}. The generality of our formulation further yields twisted $K$-theories beyond \cite{F-M}: A simple example is a twisted $K$-theory ${}^\phi K^n(X)$ of a space $X$ whose twisting datum $\phi \in \Phi(X)$ is classified by $H^1(X; \Z_2)$. This is different from the twisted $K$-theory $K^{c + n}(X)$ whose twisting datum $c$ is also classified by $H^1(X; \Z_2)$.

\medskip

The proof of the Bott periodicity is based on the idea in \cite{FHT1}: By nature of local quotient groupoid, we reduce the problem to the case of the quotient groupoid $\pt//G$, where $G$ is a compact Lie group. Then, based on the so-called ``Mackey decomposition'', we further reduce the problem to the case that $G$ is trivial. At this point, the periodicity essentially follows from \cite{A-Si}, which is the reason that we use skew-adjoint Fredholm operators to formulate ${}^\phi K^{(\tau, c) + n}(\X)$. It should be noticed that we topologize the space of Fredholm operators by using the compact open topology in the sense of \cite{A-Se}, as opposed to the operator norm topology as in \cite{A-Si}. Accordingly, some analytical details about the space of Fredholm operators are also supplied in this paper.

Also, based on the idea in \cite{FHT1}, the Thom isomorphism theorem for real vector bundles can be shown in the context of the Freed-Moore $K$-theory (see \S\S\ref{subsec:thom_isomorphism}). As in the other cases \cite{D-K,FHT1}, the isomorphism involves a twist associated to the real vector bundle. Geometrically, this twist is explained as the obstruction to the orientability and to the existence of \textit{$\phi$-twisted $\Pin^c$-structure} introduced in Definition \ref{dfn:twisted_pinc}.

\subsubsection{Karoubi formulation}

In view of the classifications of gapped quantum systems like topological insulators, Karoubi's formulation of $K$-theory by using the notion of triples \cite{Kar} is very useful, as is seen in \cite{Kit}. Concretely, in this formulation of the standard complex $K$-theory $K(X)$ of a space $X$, its representative is a triple $(E, \eta_0, \eta_1)$ consisting of a finite rank Hermitian vector bundle $E$ on $X$ and two \textit{gradations} (or \textit{$\Z_2$-gradings}), namely, self-adjoint involutions $\eta_0$ and $\eta_1$ acting on $E$. These self-adjoint involutions define subbundles $\mathrm{Ker}(1 - \eta_i) \subset E$, and the pair of these vector bundles is nothing but a representative of the standard formulation of $K(X)$. In the context of the classification of gapped quantum systems,  the Hamiltonians of such systems lead to self-adjoint involutions $\eta$ (see for instance \cite{SSG3}). Hence the $K$-theory in Karoubi's formulation naturally works as a framework to measure the relative topological phases of two gapped quantum systems.

One can generalize Karoubi's triples to formulate Freed-Moore $K$-theory. However, its relationship with the finite rank formulation as in \cite{F-M} and the Fredholm formulation seems to be not fully studied in the literature. It should be noticed also that the relationship between Karoubi's formulation and the standard formulation of $K$-theory cannot be generalized in the presence of a certain twist. The other purpose of this paper is thus to clarify the relationship among the formulations.

\medskip

For this purpose, the key is an infinite-dimensional version of Karoubi's formulation above: Based on the infinite-dimensional Grassmannian in \cite{P-S,Qui}, we introduce a group ${}^\phi \K^{(\tau, c) + n}(\X)$ under the same setting as in the Fredholm formulation of the Freed-Moore $K$-theory. We then prove (Theorem \ref{thm:Fredholm_vs_Karoubi}) that there is a natural isomorphism of groups
$$
\vartheta: \ {}^\phi K^{(\tau, c) + n}(\X) 
\overset{\cong}{\longrightarrow} {}^\phi \K^{(\acute{\tau}, c) + n}(\X).
$$
Here one should notice the change of the $\phi$-twists $\tau \mapsto \acute{\tau}$. It will be shown in \S\S\ref{subsec:Fredholm_vs_Karoubi} that $\tau \cong \acute{\tau}$ if $\phi$ or $c$ are trivial. Hence the essential effect of the twist change is observed only when non-trivial $\phi$ and $c$ are present. The appearance of the twist change is due to the use of skew-adjoint operators in the Fredholm formulation. Using self-adjoint operators instead, one can avoid the twist change (Remark \ref{rem:twist_change}).

\medskip

To see the relationship between the infinite-dimensional and finite-dimensional Karoubi formulations, we suppose that the groupoid $\X$ is the quotient groupoid $X//G$ associated to an action of a finite group $G$ on a compact Hausdorff space $X$, $\phi \in \Phi(X//G)$ is associated to a homomorphism $\phi : G \to \Z_2$, and the $\phi$-twist is realized as a twisted extension of $X//G$. In this setting, we define a group ${}^\phi \K^{(\tau, c) + n}_G(X)_{\mathrm{fin}}$ by using Karoubi triples of finite rank, and show (Theorem \ref{thm:finite_and_infinite_dimensional_Karoubi_formulations}) that there is an isomorphism
$$
\jmath : \ {}^\phi \K^{(\tau, c) + n}_G(X)_{\mathrm{fin}}
\overset{\cong}{\longrightarrow} {}^\phi \K^{(\tau, c) + n}_G(X)
= {}^\phi \K^{(\tau, c) + n}(X//G).
$$

To summarize, we denote by ${}^\phi K^{(\tau, c) + n}_G(X)_{\mathrm{fin}}$ the Freed-Moore $K$-theory formulated by finite rank bundles as in \cite{F-M}, and put ${}^\phi K^{(\tau, c) + n}_G(X) = {}^\phi K^{(\tau, c) + n}(X//G)$. Then we have a diagram
$$
\begin{CD}
{}^\phi K^{(\tau, c) + n}_G(X)_{\mathrm{fin}}
@>{\imath}>>
{}^\phi K^{(\tau, c) + n}_G(X) \\
@. @V{\cong}V{\vartheta}V \\
{}^\phi \K^{(\acute{\tau}, c) + n}_G(X)_{\mathrm{fin}}
@>{\jmath}>{\cong}>
{}^\phi \K^{(\acute{\tau}, c) + n}_G(X),
\end{CD}
$$
in which $\imath$ is a homomorphism, and $\vartheta$ and $\jmath$ are isomorphisms. It is stated in \cite{F-M} that $\imath$ is bijective if $n = 0$ (Remark 7.37), but its proof (Appendix E) seems to work only when the twist $c$ is trivial (see \S\S\ref{subsec:finite_rank}). In this case, we reprove the bijectivity of $\imath$ by constructing the inverse of $\jmath^{-1} \circ \vartheta \circ \imath : {}^\phi K^{(\tau, c) + n}_G(X)_{\mathrm{fin}} \to {}^\phi \K^{(\acute{\tau}, c) + n}_G(X)_{\mathrm{fin}}$. Actually, the inverse is induced from the construction $(E, \eta_0, \eta_1) \mapsto (\mathrm{Ker}(1 - \eta_0), \mathrm{Ker}(1 - \eta_1))$ as mentioned above (see \S\S\ref{subsec:finite_dimensional_formulations}).

\medskip

Besides the formulations above are $C^*$-algebraic formulations. Such a formulation of the Freed-Moore $K$-theory can be found for example in \cite{Kel,Kub,Th}. Notice that, in \cite{Th}, Karoubi's triple formulation is also presented in a context of a $C^*$-algebra. These formulations should produce the same $K$-theory as formulated in this paper.


\subsection{Outline of the paper}

In \S\ref{sec:twisted_bundle}, we introduce notions of twists and twisted vector bundles needed for the Freed-Moore $K$-theory. We start with a brief review of groupoids and their cohomology. We then recall the notion of twisted extension in \cite{F-M}, and use it to define $\phi$-twists and twisted vector bundles along the idea of \cite{FHT1}. We also introduce the notion of locally universal bundles following \cite{FHT1}. 

In \S\ref{sec:Fredholm_formulation}, we formulate the Freed-Moore $K$-theory ${}^\phi K^{(\tau, c) + n}(\X)$ by using Fredholm operators. As an intermediate step, we introduce a $K$-theory ${}^\phi K^{(\tau, c) + (p, q)}(\X)$ with bigrading as in \cite{Kar}, by using the Clifford algebra. We then prove the Bott periodicity. As explained, the proof consists of reductions to easier cases following \cite{FHT1} and the periodicity on the point \cite{A-Si}. The reduction argument based on the Mackey decomposition and the periodicity on the point are separated to Appendix. Then, we derive the relation between twists and degree shifts from the Bott periodicity. After that, we review how the Freed-Moore $K$-theory reproduces known $K$-theories. We also treat the finite rank realizability here, introducing ${}^\phi K^{(\tau, c) + 0}_G(X)_{\mathrm{fin}}$. At the end of this section, a notion of $\phi$-twisted $\Pin^c$-structures and the Thom isomorphism in the Freed-Moore $K$-theory are given.

\S\ref{sec:Karoubi_formulation} is devoted to Karoubi's formulations. We first introduce ${}^\phi\K^{(\tau, c) + (p, q)}(\X)$ in the infinite-dimensional Karoubi formulation, and relate it with the Fredholm formulation ${}^\phi K^{(\tau, c) + (p, q)}(\X)$. We then relate the infinite-dimensional Karoubi formulation ${}^\phi\K^{(\tau, c) + (p, q)}_G(X)$ with its finite-dimensional counterpart ${}^\phi\K^{(\tau, c) + (p, q)}_G(X)_{\mathrm{fin}}$. Finally, two finite-dimensional formulations ${}^\phi K^{(\tau, c) + 0}_G(X)_{\mathrm{fin}}$ and ${}^\phi \K^{(\tau, c) + 0}_G(X)_{\mathrm{fin}}$ are compared. 

In Appendix \ref{sec:classification_of_twists}, we summarize the classification of twists in some simple cases needed. In Appendix \ref{sec:mackey_decomposition}, we provide the Mackey decomposition needed for our reduction argument, and supply some technical details of the Bott periodicity on a point. Finally, in Appendix \ref{sec:quotient_monoid}, the quotient monoid is reviewed, which is used to give ${}^\phi K^{(\tau, c) + 0}_G(X)_{\mathrm{fin}}$ and ${}^\phi\K^{(\tau, c) + (p, q)}_G(X)_{\mathrm{fin}}$.

\medskip

As a convention, a space is always assumed to be locally contractible, paracompact and completely regular, as in \cite{FHT1}. Vector bundles are always $\Z_2$-graded, and infinite-dimensional cases are allowed. In the infinite-dimensional case, the fibers are assumed to be separable Hilbert spaces, and operators are assumed to be bounded (continuous).

\medskip

\begin{acknowledgment}
I would like to thank I.~Sasaki for discussion about some analytic aspects in this work. I would also like to thank J.~Rosenberg and anonymous reviewers for useful and helpful comments which improved this paper significantly. The author's research is supported by JSPS KAKENHI Grant Number JP15K04871.
\end{acknowledgment}


\section{Twisted vector bundle on groupoid}
\label{sec:twisted_bundle}

In this section, we prepare for the setting for the formulation of the Freed-Moore $K$-theory. We start with a brief review of local quotient groupoids \cite{FHT1} and their cohomology groups. We then recall the notion of twisted extension \cite{F-M}, and introduce $\phi$-twists and twisted vector bundles.

\subsection{Groupoid}

A \textit{groupoid} $\X$ in this paper means a small category in which all the arrows (morphisms) are invertible, and the set of objects $\X_0$ as well as that of invertible morphisms (isomorphisms) $\X_1$ are topological spaces subject to our convention. We also assume the continuity of the maps $\X_1 \to \X_0$ that associates the source objects $s$ and the target objects $t$ to morphisms $s \overset{f}{\to} t$, the map $\X_1 \to \X_1$ of taking the inverse of arrows, and the map $\X_0 \to \X_1$ that associates the identity arrows to objects. 

We will write $\partial_0 : \X_1 \to \X_0$ and $\partial_1 : \X_1 \to \X_0$ for the associations of the source and the target of a morphism, respectively,
\begin{align*}
\partial_0(s \overset{f}{\to} t) &= s, &
\partial_1(s \overset{f}{\to} t) &= t.
\end{align*}
For $n > 1$, we denote by $\X_n$ the space of $n$ composable morphisms, and define $\partial_i : \X_n \to \X_{n-1}$, ($i = 0, \ldots, n$) by
$$
\partial_i(f_1, \ldots, f_n)
= 
\left\{
\begin{array}{ll}
(f_2, \ldots, f_n), & (i = 0) \\
(f_1, \ldots, f_if_{i+1}, \ldots, f_n), & (1 \le i \le n-1) \\
(f_1, \ldots, f_{n-1}), & (i = n)
\end{array}
\right.
$$
which satisfy
$$
\partial_i \circ \partial_j = \partial_{j-1} \circ \partial_i. 
\quad
(i < j)
$$
The spaces $\X_0, \X_1, \X_2, \ldots$ and the maps $\partial_i$ above, called the face maps, are part of the data of the simplicial space associated to the groupoid $\X$. The remaining data called the degeneracy maps will play no essential role in this paper, so we omit their definitions here.

A well-known example of a groupoid is the \textit{quotient groupoid} $\X = X//G$, which is associated to an action of a compact Lie group $G$ on a compact Hausdorff space $X$. In this groupoid, the set of objects is identified with $(X//G)_0 = X$, and that of arrows with $(X//G) = G \times X$.

A map of groupoids $\X \to \Y$ is given by a functor. Taking the topological setting into account, we assume the induced map of objects $\X_0 \to \Y_0$ and that of arrows $\X_1 \to \Y_1$ are continuous. For example, let us consider a map of groupoids $\phi : \X \to \pt//\Z_2$. Since the map of objects is trivial, this $\phi$ amounts to a continuous map $\phi : \X_1 \to \Z_2$ such that $\phi(f_1 \circ f_2) = \phi(f_1)\phi(f_2)$ for all the composable morphisms $f_1, f_2 \in \X_1$. In particular, a continuous homomorphism $\phi : G \to \Z_2$ gives a map from the quotient groupoid $X//G$ to $\pt//\Z_2$, although not all maps $X//G \to \pt//\Z_2$ come from continuous homomorphisms $G \to \Z_2$.

\medskip

As equivalences of groupoids, we consider \textit{local equivalences} \cite{FHT1}. Then a \textit{local quotient groupoid} is defined as a groupoid which is covered by full subgroupoids which are locally equivalent to the groupoids associated to actions of compact Lie groups on Hausdorff spaces (see \cite{FHT1} for details).

\subsection{Cohomology of groupoid}
\label{subsec:cohomology_of_groupoid}

For any abelian group $A$ (or more generally any ring), the cohomology $H^n(\X; A)$ of a groupoid $\X$ can be defined as the cohomology of the simplicial space associated to $\X$. A convenient way to realize $H^n(\X; A)$ is to use a \v{C}ech cohomology (cf.\ \cite{G2}). 

Any abelian group $A$ admits the automorphism $\iota: A \to A$ of taking the inverse. Then, combining $\iota$ with a map of groupoids $\phi : \X \to \pt//\Z_2$, we can define the cohomology $H^n(\X; A_\phi)$ of $\X$ with local coefficients. A definition of $H^n(\X; A_\phi)$ in terms of \v{C}ech cohomology uses the notion of a twisting function \cite{May2} of the simplicial space associated to $\X$. The twisting function in the present case is the sequence of maps $\phi_n : \X_n \to \Z_2$, ($n \ge 1$) defined by $\phi_1 = \phi$ and $\phi_n = \phi \circ \partial_2 \circ \cdots \circ \partial_n$ for $n \ge 2$, which are subject to 
\begin{align*}
\phi_n \cdot \partial_0^*\phi_{n-1} &= \partial_1^*\phi_{n-1}, &
\partial_i^*\phi_{n-1} &= \phi_n. \quad (i > 1)
\end{align*}
One can ``twist'' a differential of a double complex which computes $H^n(\X; A)$, by using the twisting function (cf.~\cite{G1}). This construction produces another double complex, and its cohomology gives $H^n(\X; A_\phi)$.

Notice that if $\phi' : \X \to \pt//\Z_2$ is another map and there is $\psi_0 : \X_0 \to \Z_2$ such that $\partial_1^*\psi_0 \cdot \phi_1 = \phi'_1 \cdot \partial_0^*\psi_0$, then $\psi_0$ defines an isomorphism $H^n(\X; A_\phi) \to H^n(\X; A_{\phi'})$. If $\psi_0' : \X_0 \to \Z_2$ is another map such that $\partial_1^*\psi'_0 \cdot \phi_1 = \phi'_1 \cdot \partial_0^*\psi'_0$, we get the same isomorphism in cohomology. In view of this fact, we regard that maps $\phi : \X \to \pt//\Z_2$ constitute objects of a category in which the set of morphisms $\mathrm{Mor}(\phi, \phi')$ consists of maps $\psi_0$ as above modulo those satisfying $\partial_0^*\psi_0 = \partial_1^*\psi_0$.

More generally, let us consider the category $\Phi(\X)$ such that its object is a pair $(F : \tilde{\X} \to \X, \phi)$ consisting of a local equivalence $F : \tilde{\X} \to \X$ and a map of groupoids $\phi : \tilde{\X} \to \pt//\Z_2$. We define the set of morphisms from $(F_1 : \tilde{\X}_1 \to \X, \phi_1)$ to $(F_2 : \tilde{\X}_2 \to \X, \phi_2)$ to be the direct limit (colimit)
$$
\varinjlim_{\Y} \mathrm{Mor}(\pi_1^*\phi_1, \pi_2^*\phi_2),
$$
where $\Y$ runs over groupoids which fill the diagram of local equivalences
$$
\begin{CD}
\Y @>{\pi_2}>> \tilde{\X}_2 \\
@V{\pi_1}VV @VV{F_2}V \\
\tilde{\X}_1 @>{F_1}>> \X.
\end{CD}
$$
By definition, we can associate an object in $\Phi(\X)$ to each map of groupoids $\phi : \X \to \pt//\Z_2$ by considering the identity local equivalence $\X \to \X$. In general, $\Phi(\X)$ contains objects which are not associated to maps of groupoids $\phi : \X \to \pt//\Z_2$ as above. However, for the quotient groupoid $\pt//G$, any object in $\Phi(\pt//G)$ is isomorphic to the object associated to a homomorphism $\phi : G \to \Z_2$.

For each object $(F : \tilde{\X} \to \X, \phi)$ in $\Phi(\X)$, we have the cohomology $H^n(\tilde{\X}; A_\phi)$. If there is a morphism between two objects in $\Phi(\X)$, then it is unique and induces a unique isomorphism in cohomology. Therefore we take the colimit to define the cohomology twisted by an isomorphism class $[F, \phi]$ of $(F, \phi) \in \Phi(\X)$ as
$$
H^n(\X; A_{[F, \phi]})
= \varinjlim_{(F, \phi) \in \Phi(\X)}
H^n(\tilde{\X}; A_\phi).
$$
By abuse of notation, we may write $\phi$ to mean an object in $\Phi(\X)$, and $H^n(\X; A_\phi)$ for the above cohomology associated to the isomorphism class of $\phi$.

It should be noticed that the objects in $\Phi(\X)$ admit the classification
$$
\pi_0(\Phi(\X)) \cong H^1(\X; \Z_2),
$$
where $\pi_0(\Phi(\X))$ denotes the set of isomorphism classes. The identification above is actually an isomorphism of groups, where the group structure on $\pi_0(\Phi(\X))$ is induced from the obvious product of morphism of groupoids $\phi: \X \to \pt//\Z_2$.

\medskip

If $\X$ is a quotient groupoid $\X = X//G$, then $H^n(\X; A)$ can be identified with the Borel equivariant cohomology $H^n_G(X; A)$, which is the cohomology of the Borel construction $EG \times_G X$ with its coefficients in $A$. By definition, the Borel construction is the quotient space $EG \times_G X = (EG \times X)/G$, where $EG$ is the total space of the universal $G$-bundle $EG \to BG$ and the action of $g \in G$ on the direct product is $(\xi, x) \mapsto (\xi g^{-1}, gx)$. For a map of groupoids $\phi : X//G \to \pt//\Z_2$, one may identify the cohomology $H^n(\X; A_\phi)$ with the Borel equivariant cohomology $H^n_G(X; A_\phi)$, where the local system on $EG \times_G X$ is the map $EG \times_G X \to B\Z_2$ associated to $\phi : X//G \to \pt//\Z_2$. In the case where $X = \pt$, the cohomology can be identified with a group cohomology. The cochain complex producing $H^n_G(\pt; A_\phi)$ is explicitly given in Appendix \ref{sec:classification_of_twists}.


\subsection{Twisted extension}

We introduce some notations following \cite{F-M}: Given a complex number $z \in \C$ and a sign $\phi \in \Z_2 = \{ \pm 1 \}$, we write
$$
{}^\phi z = \left\{
\begin{array}{ll}
z, & (\phi = 1) \\
\bar{z}. & (\phi = -1).
\end{array}
\right.
$$
Similarly, for a complex vector bundle $E \to X$ on a space $X$, we write
$$
{}^\phi E 
= \left\{
\begin{array}{ll}
E, & (\phi = 1) \\
\overline{E}, & (\phi = -1).
\end{array}
\right.
$$
where $\overline{E}$ is the complex conjugate of $E$. As a generalization, for a continuous map $\phi : X \to \Z_2$, we define a vector bundle ${}^\phi E \to X$ by
$$
{}^\phi E = E|_{\phi^{-1}(1)} \sqcup \overline{E}|_{\phi^{-1}(-1)},
$$
noting that we can express $X$ as the disjoint union $X = \phi^{-1}(1) \sqcup \phi^{-1}(-1)$.

A Hermitian vector bundle $E$ over a space $X$ is called a $\Z_2$-graded Hermitian vector bundle if $E$ admits a direct sum decomposition $E = E^0 \oplus E^1$ into Hermitian vector bundles $E^i$. We call $E^0$ the even part (or degree $0$ part), and $E^1$ the odd part (or degree $1$ part). It can happen that $E^0 = 0$ or $E^1 = 0$. Thus, a $\Z_2$-graded Hermitian line bundle amounts to a Hermitian line bundle $L \to X$ with a $\Z_2$-grading (or parity) specified: $L = L^0$ or $L = L^1$. Generalizing this, for a continuous map $c : X \to \Z_2$, we define a $c$-graded Hermitian line bundle $L \to X$ to be a Hermitian line bundle such that the restriction to $c^{-1}((-1)^i) \subset X$ has degree $i$. If $L$ is $c$-graded and $L'$ is $c'$-graded, then their tensor product $L \otimes L'$ is $cc'$-graded as a convention. To the exchange of factors, we apply the Koszul sign rule as in \cite{F-M}, so that a negative sign appears only in the exchange of odd homogeneous elements.

\begin{dfn}[\cite{F-M}] \label{dfn:twisted_extension}
Let $\X$ be a groupoid, and $\phi : \X \to \pt//\Z_2$ a map of groupoids. A \textit{$\phi$-twisted $\Z_2$-graded extension $(L, \tau, c)$ of $\X$} consists of 
\begin{itemize}
\item
a map of groupoids $c : \X \to \pt//\Z_2$, 

\item
a $c$-graded Hermitian line bundle $L \to \X_1$, and 

\item
a unitary isomorphism $\tau : \ \partial_2^* L \otimes {}^{\phi_2} \partial_0^* L \to \partial_1^*L$ on $\X_2$ which preserves the $\Z_2$-grading and makes the following diagram commutative on $\X_3$,
$$
\begin{CD}
\partial_2^*\partial_2^* L \otimes \!
{}^{\phi_3}\!\partial_0^*(\partial_2^* L \otimes \!{}^{\phi_2}\!\partial_0^* L)
@>{\id \otimes {}^{\phi_3}\!\partial_0^*\tau}>>
\partial_2^*\partial_2^*L \otimes \!
{}^{\phi_3}\!\partial_0^*\partial_1^* L \\
@| @| \\
\partial_3^*\partial_2^* L \otimes
\partial_3^*\!{}^{\phi_2}\!\partial_0^* L \otimes
\partial_1^*\!{}^{\phi_2}\!\partial_0^*L
@.
\partial_2^*\partial_2^*L \otimes
\partial_2^*\!{}^{\phi_2}\!\partial_0^*L \\
@V{\partial_3^*\tau \otimes \id}VV @VV{\partial_2^*\tau}V \\
\partial_3^*\partial_1^*L \otimes \partial_1^*{}^{\phi_2}\partial_0^* L
@.
\partial_2^*\partial_1^*L \\
@| @| \\
\partial_1^*\partial_2^* L \otimes
\partial_1^*{}^{\phi_2}\partial_0^* L
@>{\partial_1^*\tau}>>
\partial_1^*\partial_1^*L,
\end{CD}
$$
where $\phi_n = \phi \circ \partial_2 \circ \cdots \circ \partial_n$ for $n \ge 2$ as defined in \S\S\ref{subsec:cohomology_of_groupoid}. 
\end{itemize}
The trivial $\phi$-twisted $\Z_2$-graded extension $(L, \tau, c)$ consists of the trivial map $c : \X \to \pt//\Z_2$ (i.e.\ $\X_1 \to \Z_2$ is the constant map at $1 \in \Z_2)$, the product bundle $L = \X_1 \times \C$ and the trivial isomorphism $\tau$.
\end{dfn}

It is helpful to express the commutative diagram for $\tau$ in Definition \ref{dfn:twisted_extension} as follows: Let $L_f$ denote the fiber of $L \to \X_1$ at $f \in \X_1$. Then, the map $\tau$ at $(f_1, f_2) \in \X_2$, consisting of composable morphisms $f_1, f_2 \in \X_1$, amounts to $\tau_{(f_1, f_2)} : L_{f_1} \otimes {}^{\phi(f_1)} L_{f_2} \to L_{f_1f_2}$, and the diagram at $(f_1, f_2, f_3) \in \X_3$ to 
$$
\begin{CD}
L_{f_1} \otimes {}^{\phi(f_1)}(L_{f_2} \otimes {}^{\phi(f_2)} L_{f_3})
@>{\mathrm{id} \otimes \tau_{(f_2, f_3)}^{\phi(f_1)}}>> 
L_{f_1} \otimes {}^{\phi(f_1)}L_{f_2f_3} \\
@V{\tau_{(f_1, f_2) \otimes \mathrm{id}}}VV 
@VV{\tau_{(f_1, f_2f_3)}}V \\
L_{f_1f_2} \otimes {}^{\phi(f_1f_2)} L_{f_3}
@>>{\tau_{(f_1f_2, f_3)}}> L_{f_1f_2f_3}.
\end{CD}
$$
This ``fiberwise expression'' is employed in \cite{F-M}, and a similar expression is possible for twisted vector bundles to be defined in \S\S\ref{subsec:twisted_vector_bundle}

We remark that we apply a convention different from the one in \cite{F-M}. We also remark that a $\phi$-twisted ungraded extension of a groupoid $\X$ is defined by forgetting about the information on the $\Z_2$-grading specified by $c$. Every $\phi$-twisted ungraded extension of $\X$ can be thought of as a $\phi$-twisted $\Z_2$-graded extension by taking $c : \X \to \pt//\Z_2$ to be trivial.

\begin{dfn}
Let $\X$ be a groupoid, and $\phi : \X \to \pt//\Z_2$ a map of groupoids. An isomorphism $[K, \beta, b] : (L', \tau', c') \to (L, \tau, c)$ of $\phi$-twisted $\Z_2$-graded extensions of $\X$ is the equivalence class of data $(K, \beta, b)$ consisting of
\begin{itemize}
\item
a map $b : \X_0 \to \Z_2$, 

\item
a $b$-graded Hermitian line bundle $K \to \X_0$, and 

\item
a unitary isomorphism $\beta : L' \otimes {}^\phi\partial_0^*K \to \partial_1^*K \otimes L$ on $\X_1$ which preserves the $\Z_2$-grading and makes the following diagram commutative on $\X_2$,
$$
\begin{CD}
\partial^*_2L' \otimes 
{}^{\phi_2} \partial^*_0
(L' \otimes {}^\phi \partial_0^*K)
@>{\id \otimes {}^{\phi_2}\partial^*_0\beta}>>
\partial^*_2L' \otimes 
{}^{\phi_2} \partial^*_0
(\partial^*_1K \otimes  L) \\
@V{\tau \otimes 1}VV @| \\
\partial^*_1L' \otimes 
{}^{\phi_2}\partial^*_0 {}^\phi \partial^*_0K 
@.
\partial^*_2L' \otimes 
\partial^*_2{}^\phi \partial^*_0K \otimes 
{}^{\phi_2} \partial^*_0 L \\
@| @VV{\partial^*_2\beta \otimes \id}V \\
\partial^*_1L' \otimes \partial^*_1{}^\phi\partial^*_0 K 
@.
\partial^*_2\partial^*_1K \otimes 
\partial^*_2L \otimes 
{}^{\phi_2} \partial^*_0 L \\
@V{\partial^*_1\beta}VV @VV{\id \otimes \tau'}V \\
\partial^*_1\partial^*_1K \otimes \partial^*_1L
@=
\partial^*_2\partial^*_1 K \otimes \partial^*_1 L.
\end{CD}
$$ 
\end{itemize}
The data $(K, \beta, b)$ and $(K', \beta', b')$ are equivalent if we have
\begin{itemize}
\item
$b' = ba$ for a map $a : \X_0 \to \Z_2$ such that $\partial_0^*a = \partial_1^*a$, and

\item
a unitary isomorphism $\alpha : K \to K'$ on $\X_0$ which preserves the $\Z_2$-grading and makes the following diagram commutative on $\X_1$,
$$
\begin{CD}
L' \otimes {}^\phi \partial^*_0 K @>\beta>> \partial^*_1K \otimes L \\
@V{\id \otimes {}^\phi \partial^*_0\alpha}VV 
@VV{\partial^*_1 \alpha \otimes \id}V \\
L' \otimes {}^\phi \partial^*_0 K' @>{\beta'}>> \partial^*_1 K' \otimes L.
\end{CD}
$$
\end{itemize}
\end{dfn}

With the morphisms above, we get a category ${}^\phi\mathfrak{Ext}(\X)$ whose objects are $\phi$-twisted $\Z_2$-graded extensions of $\X$.

We here examine a special type of a $\phi$-twisted $\Z_2$-graded extension $(L, \tau, c)$ of $\X$ such that $L \to \X_1$ is the product bundle. In this case, the unitary isomorphism $\tau$ of Hermitian line bundles amounts to a function $\tau : \X_2 \to U(1)$ satisfying
$$
\partial_3^* \tau \cdot \partial_1^*\tau
=
{}^{\phi_3}\partial_0^*\tau \cdot
\partial_2^* \tau.
$$
Moreover, if $\X$ is the quotient groupoid $X = X//G$, then $\tau$ is a function $\tau : G \times G \times X \to U(1)$ satisfying
$$
\tau(g, h; kx) \cdot
\tau(gh, k; x)
=
{}^{\phi(g)}\tau(h, k; x) \cdot
\tau(g, hk; x).
$$
Thus, in terms of group cohomology, we have $\tau \in Z^2_{\mathrm{group}}(G; C(X, U(1))_\phi)$, namely, $\tau$ is a $2$-cocycle of $G$ with values in the group $C(X, U(1))$ of $U(1)$-valued functions regarded as a two-sided $G$-module by the homomorphism $\phi : G \to \Z_2$ and the pull-back action of $G$ (see Appendix \ref{sec:classification_of_twists} for details). Under the same assumption, the unitary isomorphism $\beta$ in an isomorphism $[K, \beta, b] : (L, \tau, c) \to (L', \tau', c')$ of $\phi$-twisted $\Z_2$-graded central extensions amounts to a $1$-cochain $\beta$ of $G$ such that
$$
\tau(g, h; x) \beta(gh; x) =
{}^{\phi(g)}\beta(h; x) \beta(g; hx) \tau'(g, h; x),
$$
modulo the coboundary of a $0$-cochain $a$ of $G$.

\subsection{Twist}
\label{subsec:twist}

Generalizing \cite{FHT1}, we define twists involving $\phi$ as follows.

\begin{dfn}
Let $\X$ be a groupoid, and $\phi = (F : \tilde{\X} \to \X, \phi)$ an object of $\Phi(\X)$.
\begin{enumerate}
\item[(a)]
A \textit{graded $\phi$-twist} (or a \textit{$\phi$-twist}, a \textit{twist} for short) on $\X$ consists of:
\begin{itemize}
\item
a local equivalence $\tilde{F} : \tilde{\tilde{\X}} \to \tilde{\X}$,

\item
$\tilde{F}^*\phi$-twisted $\Z_2$-graded extension $(L, \tau, c)$ of $\tilde{\tilde{\X}}$.
\end{itemize}
We may write $(\tau, c)$ for a $\phi$-twist $(\tilde{F} : \tilde{\tilde{\X}} \to \tilde{\X}, L, \tau, c)$.

\item[(b)]
For graded $\phi$-twists $(\tilde{F}_1 : \tilde{\tilde{\X}}_1 \to \tilde{\X}, L_1, \tau_1, c_1)$ and $(\tilde{F}_2 : \tilde{\tilde{\X}}_2 \to \tilde{\X}, L_2, \tau_2, c_2)$ on $\X$, the set of isomorphisms is defined as
$$
\varinjlim_{\tilde{\Y}}
\mathrm{Mor}_{{}^\phi\mathfrak{Ext}(\tilde{\Y})}
(\tilde{\pi}^*_1(c_1, L_1, \tau_1), 
\tilde{\pi}_2^*(c_2, L_2, \tau_2)),
$$
where $\tilde{\Y}$ runs over groupoids which fill the diagram of local equivalences
$$
\begin{CD}
\tilde{\Y} @>{\tilde{\pi}_2}>> \tilde{\tilde{\X}}_2 \\
@V{\tilde{\pi}_1}VV @VV{\tilde{F}_2}V \\
\tilde{\tilde{\X}}_1 @>{\tilde{F}_1}>> \tilde{\X}.
\end{CD}
$$
\end{enumerate}
We write ${}^\phi\Twist(\X)$ for the category of graded $\phi$-twists on $\X$.
\end{dfn}

As in the case of extensions of groupoids, ungraded twists are defined by forgetting the information on the $\Z_2$-grading $c$ in the definition above. Any ungraded twist can be thought of as a graded twist by the trivial $\Z_2$-grading. Hence the category ${}^\phi\Twist(\X)$ of graded $\phi$-twists on $\X$ contains the category ${}^\phi\Twist^+(\X)$ of ungraded $\phi$-twists as its full subcategory. By the tensor product of line bundles, these categories give rise to monoidal categories. Considering the isomorphism classes of these monoidal categories, we get the groups $\pi_0({}^\phi\Twist(\X))$ and $\pi_0({}^\phi\Twist^+(\X))$. By means of \v{C}ech cohomology groups (cf.\ \cite{G2, Kub}), we can show the following classification of twists
\begin{align*}
\pi_0({}^\phi\Twist(\X)) &\cong H^3(\X; \Z_\phi) \times H^1(\X; \Z_2), &
&(\mbox{set bijection}) \\
\pi_0({}^\phi\Twist^+(\X)) &\cong H^3(\X; \Z_\phi). &
&(\mbox{group isomorphism})
\end{align*}
The group structure on $\pi_0({}^\phi\Twist(\X))$ leads to the exact sequence of groups
$$
1 \to 
H^3(\X; \Z_\phi) \to 
H^3(\X; \Z_\phi) \times 
H^1(\X; \Z_2) \to 
H^1(\X; \Z_2) \to
1.
$$
With some calculations, we can identify the extension class of $\pi_0({}^\phi\Twist(\X))$ with the cup product $\cup : H^1(\X; \Z_2) \times H^1(\X; \Z_2) \to H^2(\X; \Z_2)$ followed by the Bockstein homomorphism $\tilde{\beta} : H^2(\X; \Z_2) \to H^3(\X; \Z_\phi)$ associated to the short exact sequence of coefficients $\Z_\phi \overset{2}{\to} \Z_\phi \to (\Z_2)_\phi = \Z_2$.

\medskip

For maps $f_i : \X' \to \X$ of groupoids ($i = 0, 1$), a homotopy is defined as a map $\tilde{f} : \X' \times [0, 1] \to \X$ such that $\tilde{f}|_{\X \times \{ i \}} = \tilde{f}_i$, where $\X' \times [0, 1]$ is the groupoid such that $(\X' \times [0, 1])_j = \X'_j \times [0, 1]$. 
\begin{lem} \label{lem:homotopy_property}
Let $\X$ be a local quotient groupoid, and $\phi : \X \to \pt//\Z_2$ a map of groupoids. Suppose that there is a homotopy $\tilde{f} : \X' \times [0, 1] \to \X$ of maps $\tilde{f}_0$ and $\tilde{f}_1$ from another groupoid $\X'$ to $\X$. For any $\phi$-twisted $\Z_2$-graded extension $(L, \tau, c)$ of $\X$, there is an isomorphism $\beta : f_0^*(L, \tau, c) \to f_1^*(L, \tau, c)$.
\end{lem}
\begin{proof}
This lemma is essentially a consequence of the homotopy invariance of the cohomology that classifies $\phi$-twisted $\Z_2$-graded extensions: We have a homotopy $\tilde{f} : \X'_1 \times [0, 1] \to \X_1$ between maps $f_i : \X'_1 \to \X_1$ on the space of morphisms. This homotopy induces an isomorphism $\beta : f_0^*L \to f_1^*L'$ of $c$-graded line bundles. Together with the product bundle $K$ and the trivial map $b$, the isomorphism $\beta$ gives rise to an isomorphism of the $\phi$-twisted extensions. 
\end{proof}

\subsection{Twisted vector bundle}
\label{subsec:twisted_vector_bundle}

As is mentioned, a $\Z_2$-graded vector bundle $E$ on a space $X$ is a vector bundle with a decomposition $E = E^0 \oplus E^1$. Such a decomposition is in one to one correspondence with an involution $\epsilon : E \to E$ covering the identity of $X$. A fiber preserving map $f : E \to E$ is said to be degree $k$ if $f \circ \epsilon = (-1)^k \epsilon \circ f$. Notice that, for $i = 1, 2$, if $f_i$ is a map of degree $\lvert f_i \rvert$ and $x_i$ is an element of degree $\lvert x_i \rvert$, then $(f_1 \otimes f_2)(x_1 \otimes x_2) = (-1)^{\lvert f_2 \rvert \lvert x_1 \rvert} f_1(x_1) \otimes f_2(x_2)$ under our sign rule.

For $p, q \ge 0$, we write $Cl_{p, q}$ for the Clifford algebra \cite{A-B-S,L-M} associated to the quadratic form $Q(x) = x_1^2 + \cdots + x_p^2 - x_{p+1}^2 - \cdots - x_{p+q}^2$on $\R^{p + q}$. Concretely, $Cl_{p, q}$ is the algebra over $\R$ generated by $e_1, \ldots, e_{p+q}$ subject to the relations
\begin{align*}
e_ie_j &= - e_je_i, \quad (i \neq j), &
e_i^2 &=
\left\{
\begin{array}{ll}
-1, & (i = 1, \ldots, p) \\
1. & (i = p+1, \ldots, p+q)
\end{array}
\right.
\end{align*}
As is known, $Cl_{p, q}$ has a natural $\Z_2$-grading. A representation of $Cl_{p, q}$ or a $Cl_{p, q}$-module on a $\Z_2$-graded Hermitian vector space $E$ will mean an algebra homomorphism $\gamma : Cl_{p, q} \to \mathrm{End}_{\C}(E)$ of degree $0$ such that $\gamma(e)$ is unitary for each vector $e \in \R^{p+q}$ of unit norm. An equivalent definition of a $Cl_{p, q}$-module is a $\Z_2$-graded Hermitian vector space equipped with odd unitary maps $\gamma_i = \gamma(e_i)$, ($i = 1, \ldots, p+q$) subject to
\begin{align*}
\gamma_i\gamma_j &= - \gamma_j\gamma_i, \quad (i \neq j), &
\gamma_i^2 &=
\left\{
\begin{array}{ll}
-1, & (i = 1, \ldots, p) \\
1. & (i = p+1, \ldots, p+q)
\end{array}
\right.
\end{align*}

Now, we introduce the notion of twisted bundles \cite{F-M} in our convention.

\begin{dfn}[\cite{F-M}] \label{dfn:twisted_bundle}
Let $\X$ be a groupoid, $\phi : \X \to \pt//\Z_2$ a map of groupoids, and $(L, \tau, c)$ a $\phi$-twisted $\Z_2$-graded extension of $\X$. For $p, q \ge 0$, a \textit{$(\phi, \tau, c)$-twisted vector bundle $E$ over $\X$ with $Cl_{p, q}$-action} (or a \textit{twisted bundle} for short) is a vector bundle $E \to \X_0$ such that its fiber is a separable Hilbert space and is equipped with the following data:
\begin{itemize}
\item
($\Z_2$-grading)
a self-adjoint involution $\epsilon : E \to E$ which specifies a $\Z_2$-grading $E = E^0 \oplus E^1$ by $E^k = \mathrm{Ker}(\epsilon - (-1)^k)$.

\item
($(\phi, \tau, c)$-twisted action)
an isometric map 
$$
\rho : L \otimes {}^\phi \partial_0^*E \to \partial_1^*E
$$ 
on $\X_1$ which preserves the $\Z_2$-grading and makes the following diagram commutative on $\X_2$,
\begin{equation*}
\begin{CD}
\partial_2^*L \otimes
{}^{\phi_2}\partial_0^*L \otimes
{}^{(\phi_2 \partial_0^*\phi)}\partial_0^*\partial_0^* E
@=
\partial_2^*L \otimes 
{}^{\phi_2}\partial_0^*(L \otimes {}^\phi \partial_0^*E) \\
@V{\tau \otimes \id}VV 
@VV{\id \otimes {}^{\phi_2}\partial_0^*\rho}V \\
\partial_1^*L \otimes {}^{\partial_1^*\phi}\partial_0^*\partial_0^* E
@. 
\partial_2^*L \otimes 
{}^{\phi_2} \partial_0^*\partial_1^* E \\
@| @| \\
\partial_1^*L \otimes
\partial_1^*{}^\phi\partial_0^*E
@.
\partial_2^*L \otimes 
\partial_2^*{}^\phi \partial_0^* E \\
@V{\partial_1^*\rho}VV @VV{\partial_2^*\rho}V \\
\partial_1^*\partial_1^* E
@=
\partial_2^*\partial_1^* E.
\end{CD}
\end{equation*}

\item
($Cl_{p, q}$-action)
Unitary maps $\gamma(e) : E \to E$ for unit norm elements $e \in Cl_{p, q}$ which make each fiber of $E$ into a representation of $Cl_{p, q}$ and the following diagram into a commutative one on $\X_1$,
$$
\begin{CD}
L \otimes {}^\phi \partial_0^* E @>{\rho}>> \partial_1^*E \\
@V{1 \otimes {}^\phi\partial_0^*\gamma(e)}VV 
@VV{\partial_1^*\gamma(e)}V \\
L \otimes {}^\phi \partial_0^* E @>{\rho}>> \partial_1^*E.
\end{CD}
$$
\end{itemize}
\end{dfn}

In the ``fiberwise expression'', $\rho$ amounts to $\rho_f : L_f \otimes {}^{\phi(f)}E_{x} \to E_y$ at a morphism $f : x \to y$ in $\X_1$, and the commutative diagram on $\X_2$ amounts to
$$
\begin{CD}
L_{f_1} \otimes {}^{\phi(f_1)}L_{f_2} \otimes {}^{\phi(f_1f_2)} E_x
@>{\id \otimes {}^{\phi(f_1)}\rho_{f_2}}>>
L_{f_1} \otimes {}^{\phi(f_1)}E_y \\
@V{\tau_{(f_1, f_2)} \otimes \id}VV 
@VV{\rho_{f_1}}V \\
L_{f_1f_2} \otimes {}^{\phi(f_1f_2)} E_x 
@>{\rho_{f_1f_2}}>> E_z
\end{CD}
$$
for composable morphisms $f_2 : x \to y$ and $f_1 : y \to z$.

In Definition \ref{dfn:twisted_bundle}, the fiber of a twisted vector bundle can be both infinite-dimensional and finite-dimensional. In the infinite-dimensional case, we assume that the structure group of $E$ is topologized by the compact open topology in the sense of Atiyah and Segal \cite{A-Se,F-M}, and the maps $\epsilon$, $\rho$ and $\gamma(e)$ are continuous with respect to the topology.

In the case that some of the data $\phi$, $c$ and $\tau$ are trivial, we often omit it from the modifier ``$(\phi, \tau, c)$-twisted''. For example, when $\phi$ is trivial, we say $(\tau, c)$-twisted bundles instead of $(\phi, \tau, c)$-twisted bundles. The same omission will be applied to the action of $Cl_{0, 0}$.

\begin{dfn}
Let $\X$ be a groupoid, $\phi : \X \to \pt//\Z_2$ a map of groupoids, and $(L, \tau, c)$ a $\phi$-twisted $\Z_2$-graded extension of $\X$. A degree $k$ map 
$$
f :\ (E, \epsilon, \rho, \gamma) \longrightarrow
(E', \epsilon', \rho', \gamma')
$$
of $(\phi, \tau, c)$-twisted vector bundles on $\X$ with $Cl_{p, q}$-action is a vector bundle map $f : E \to E'$ which covers the identity of $\X_0$ and satisfies
\begin{align*}
f \circ \epsilon &= (-1)^k \epsilon' \circ f, &
\partial_1^*f \circ \rho &= \rho' \circ (\id_L \otimes {}^\phi\partial_0^*f), &
f \circ \gamma(e) &= (-1)^k \gamma(e) \circ f,
\end{align*}
where $e \in \R^{p+q}$ is of unit norm.
\end{dfn}

As before, the continuity of $f$ is understood in the compact open topology.

\medskip

It would be helpful to describe the data of a twisted bundle explicitly under a simplifying assumption. Let us consider a quotient groupoid $\X = X//G$ and $\phi$ associated to a homomorphism $\phi : G \to \Z_2$. We further assume that a $\phi$-twisted $\Z_2$-graded extension $(L, \tau, c)$ is such that $c$ is associated to a homomorphism $c : G \to \Z_2$ and $L$ is the product bundle. Under these assumptions, a $(\phi, \tau, c)$-twisted bundle is a Hilbert space bundle $E \to X$ equipped with: 
\begin{itemize}
\item
a self-adjoint involution $\epsilon : E \to E$ defining the $\Z_2$-grading,

\item
real orthogonal maps $\rho(g) : E \to E$ which cover the actions of $g \in G$ and satisfy
\begin{align*}
\sqrt{-1} \rho(g) &= \phi(g) \rho(g)\sqrt{-1}, &
\epsilon \rho(g) &= c(g) \rho(g) \epsilon, &
\rho(g)\rho(h) &= \tau(g, h)\rho(gh).
\end{align*}

\item
unitary maps $\gamma_j : E \to E$ which cover the identity of $X$ and satisfy
\begin{align*}
\gamma_i\gamma_j &= - \gamma_j\gamma_i, \ (i \neq j), &
\gamma_i^2 &=
\left\{
\begin{array}{ll}
-1, & (i = 1, \ldots, p) \\
1, & (i = p+1, \ldots, p+q)
\end{array}
\right. \\
\gamma_i \epsilon &= - \epsilon \gamma_i, &
\gamma_i \rho(g) &= c(g) \rho(g) \gamma_i.
\end{align*}
\end{itemize}
A degree $k$ map $f$ from this twisted bundle to another $(\phi, \tau, c)$-twisted bundle $E' \to X$ with the data $\epsilon'$, $\rho'$ and $\gamma_i'$ as above is a vector bundle map $f : E \to E'$ on $X$ satisfying
\begin{align*}
f \circ \epsilon &= (-1)^k \epsilon' \circ f, &
f \circ \rho(g) &= c(g)^k \rho'(g) \circ f, &
f \circ \gamma_i &= (-1)^k \gamma'_i \circ f.
\end{align*}

\medskip

As usual, a map of vector bundles can be regarded as a section of a vector bundle. 

\begin{lem} \label{lem:hom_bundle}
Let $\X$ be a groupoid, $\phi : \X \to \pt//\Z_2$ a map of groupoids, and $(L, \tau, c)$ a $\phi$-twisted $\Z_2$-graded extension of $\X$. For $(\phi, \tau, c)$-twisted vector bundles $(E, \epsilon, \rho, \gamma)$ and $(E', \epsilon', \rho', \gamma')$ on $\X$ with $Cl_{p, q}$-action, there is a $\phi$-twisted vector bundle $\mathrm{Hom}_{Cl_{p, q}}(E, E')$ on $\X$ such that the sections of its degree $k$ part are in one to one correspondence with degree $k$ maps $f : (E, \epsilon, \rho, \gamma) \to (E', \epsilon', \rho', \gamma')$.
\end{lem}

\begin{proof}
We first consider the case without the Clifford actions ($p = q = 0$). The $\phi$-twisted vector bundle $\mathrm{Hom}(E, E')$ is constructed as follows. Its underlying vector bundle is $\mathrm{Hom}(E, E')$ on $\X_0$. This vector bundle has the $\Z_2$-grading $\varepsilon$ by the degree of maps, and its fiber is a Hilbert space since $E$ and $E'$ are. In the compact open topology, continuous sections of $\mathrm{Hom}(E, E') \to \X_0$ are in one to one correspondence with continuous maps $E \to E'$. On $\X_1$, we define a $\phi$-twisted action $\varrho$ to be the composition of the degree $0$ maps
$$
\begin{array}{c@{}c@{}c@{}c@{}c@{}}
{}^\phi \partial_0^*\mathrm{Hom}(E, E')
& = &
\mathrm{Hom}(L \otimes {}^\phi \partial_0^*E,
L \otimes {}^\phi \partial_0^*E')
& \to &
\partial_1^*\mathrm{Hom}(E, E'). \\
f & \mapsto & \id \otimes f & \mapsto &
\rho' \circ (\id \otimes f) \circ \rho^{-1}
\end{array}
$$
This map $\varrho$ satisfies the coherence condition on $\X_1$, making $(\mathrm{Hom}(E, E'), \varepsilon, \varrho)$ into a $\phi$-twisted vector bundle on $\X$. By construction, a section of the degree $k$ part $\mathrm{Hom}^k(E, E')$ is a section $s : \X_0 \to \mathrm{Hom}^k(E, E')$ such that $\varrho \circ {}^\phi \partial^*_0s = \partial_1^*s$. Such sections $s$ are clearly one to one correspondence with degree $k$ maps $E \to E'$ of $(\phi, \tau, c)$-twisted bundles. If the $Cl_{p, q}$-actions are present, then there clearly exists a subbundle $\mathrm{Hom}_{Cl_{p, q}}(E, E')$ of $\mathrm{Hom}(E, E')$ respecting the Clifford actions.
\end{proof}

\medskip

For a groupoid $\X$, a map of groupoids $\phi : \X \to \Z_2$, and a $\phi$-twisted $\Z_2$-graded extension $(L, \tau, c)$ of $\X$, we denote the category of $(\phi, \tau, c)$-twisted vector bundles on $\X$ with $Cl_{p, q}$-action by
$$
{}^\phi\Vect^{(\tau, c) + (p, q)}(\X).
$$
In the case that $\X$ is a quotient groupoid $X//G$, we may write
$$
{}^\phi\Vect^{(\tau, c) + (p, q)}(X//G)
= {}^\phi\Vect^{(\tau, c) + (p, q)}_G(X).
$$
The tensor product of twisted bundles induces a functor
$$
\otimes : 
{}^\phi\Vect^{(\tau, c) + (p, q)}(\X) \times
{}^\phi\Vect^{(\tau', c') + (p', q')}(\X) \to
{}^\phi\Vect^{((\tau, c) + (\tau', c'))+ (p + p', q+q')}(\X).
$$
A map of groupoids $f : \X' \to \X$ also induces by pull-back a functor
$$
f^* : \
{}^\phi\Vect^{(\tau, c) + (p, q)}(\X) \longrightarrow
{}^{f^*\phi}\Vect^{(f^*\tau, f^*c) + (p, q)}(\X'),
$$
and a representative $(K, \beta, b)$ of an isomorphism $[K, \beta, b] : (L', \tau', c') \to (L, \tau, c)$ of $\phi$-twisted $\Z_2$-graded extensions induces
$$
(K, \beta, b)^* : 
{}^\phi\Vect^{(\tau, c) + (p, q)}(\X) \longrightarrow
{}^\phi\Vect^{(\tau', c') + (p, q)}(\X)
$$
by the assignment of twisted bundles $E \mapsto K \otimes E$. We remark that, in general, an automorphism of $(L, \tau, c)$ acts non-trivially on ${}^\phi\Vect^{(\tau, c) + (p, q)}(\X)$.

\subsection{Locally universal bundle}

We introduce here an extension of the notion of locally universal twisted Hilbert bundles given in \cite{FHT1}.

\begin{dfn}
Let $\X$ be a groupoid, $\phi : \X \to \pt//\Z_2$ a map of groupoids, and $(L, \tau, c)$ a $\phi$-twisted $\Z_2$-graded extension of $\X$. A $(\phi, \tau, c)$-twisted vector bundle $E$ on $\X$ with $Cl_{p, q}$-action is called \textit{locally universal} if there is an isometric embedding $E' \to E|_{\X'}$ for any open full subgroupoid $\X' \subset \X$ and any $(\phi|_{\X'}, \tau|_{\X'}, c|_{\X'})$-twisted vector bundle $E'$ on $\X'$ with $Cl_{p, q}$-action.
\end{dfn}

Extending argument in \cite{FHT1}, one can show that an embedding $E' \to E|_{\X'}$ as above is unique up to homotopy, and hence $E$ is unique up to unitary isomorphisms. Also, if $E$ is locally universal, then so is $E \oplus E$.

\begin{lem} \label{lem:locally_unversal_bundle}
Let $\X$ be a local quotient groupoid, $\phi : \X \to \pt//\Z_2$ a map of groupoids, and $(L, \tau, c)$ a $\phi$-twisted $\Z_2$-graded extension of $\X$. There is a $(\phi, \tau, c)$-twisted locally universal twisted vector bundle $E$ on $\X$ with $Cl_{p, q}$-action.
\end{lem}

\begin{proof}
The idea of the proof is basically the same as that given in \cite{FHT1}. 

First of all, it is enough to consider the case where $(L, \tau)$ is trivial. The key to this reduction is the groupoid $\mathcal{L}$ such that its space of objects is $\mathcal{L}_0 = \X_0$ and its space of isomorphisms $\mathcal{L}_1 = S(L)$ is the unit sphere bundle of $L \to \X_1$. The pull-back under the projection $\pi : S(L) \to \X_1$ induces a one to one correspondence between $(\phi, \tau, c)$-twisted bundles on $\X$ with $Cl_{p, q}$-action and $(\pi^*\phi, \pi^*c)$-twisted bundles on $\mathcal{L}$ with $Cl_{p, q}$-action which are equivariant under the right $U(1)$-action on $S(L)$. 

Then, we can further reduce the problem, and it suffices to consider the case where the base groupoid $\X$ is the quotient groupoid $\pt//G$ with $G$ a compact Lie group. The key to this reduction is that we can glue locally universal twisted bundles together to form a locally universal twisted bundles. By design, a local quotient groupoid is covered by open full subgroupoids $\X_\alpha$ which are weak equivalent to the quotient groupoids $X_\alpha//G_\alpha$, with the cardinality of indices $\alpha$ countable. Here $G_\alpha$ are compact Lie groups and $X_\alpha$ are Hausdorff spaces which are locally contractible, paracompact and completely regular. Each $X_\alpha$ admits locally contractible slices, and each slice is $G_\alpha$-equivariantly homotopy equivalent to the space of the form $G_\alpha/H$ with $H \subset G_\alpha$ a closed subgroup. The inclusion induces a local equivalence $\pt//H \to (G_\alpha/H)//G_\alpha$. Thus, there is a locally universal bundle on $\X$, if there is a locally universal $(\phi, c)$-twisted bundle with $Cl_{p, q}$-action on the quotient groupoid of the form $\pt//G$ with $G$ any compact Lie group. 

Now, the remaining thing to show is the existence of a $(\phi, c)$-twisted (locally) universal vector bundle on $\pt//G$ with $Cl_{p, q}$-action, where $G$ is any compact Lie group, and $\phi : G \to \Z_2$ and $c : G \to \Z_2$ are any continuous homomorphisms. This existence is shown in Appendix \ref{sec:mackey_decomposition} (Lemma \ref{lem:locally_universal_bundle_point_case}), by using the so-called Mackey decomposition, which reduces the consideration of a representation of a group to that of projective representations of smaller groups.
\end{proof}


\section{Fredholm formulation of Freed-Moore $K$-theory}
\label{sec:Fredholm_formulation}

In this section, we provide the Fredholm formulation of the Freed-Moore $K$-theory, and prove its periodicity and the degree shift effects of some twists. The reproductions of known $K$-theories, a relationship to the finite rank formulation in \cite{F-M}, and the Thom isomorphism theorem are also provided.

\subsection{Fredholm family}
\label{subsec:fredholm_family}

Let $\X$ be a groupoid, $\phi : \X \to \pt//\Z_2$ a map of groupoids, $(L, \tau, c)$ a $\phi$-twisted $\Z_2$-graded extension of $\X$, and $(E, \epsilon, \rho, \gamma)$ a $(\phi, \tau, c)$-twisted vector bundle  on $\X$ with $Cl_{p, q}$-action. As in Lemma \ref{lem:hom_bundle}, we let $\mathrm{End}(E) = \mathrm{Hom}(E, E)$ be the $\phi$-twisted vector bundle on $\X$ whose sections are in one to one correspondence with continuous maps $(E, \epsilon, \rho) \to (E, \epsilon, \rho)$, where the Clifford action is ignored. The fiber and the structure group of $\mathrm{End}(E)$ are given the compact open topology \cite{A-Se}. We also let $\mathrm{K}(E) \to \X$ be a fiber bundle defined as follows.
\begin{itemize}
\item
The fiber of the underlying fiber bundle $\mathrm{K}(E) \to \X_0$ at $x \in \X_0$ consists of compact operators $K : E_x \to E_x$.

\item
The bundle isomorphism $\varrho : {}^\phi \partial_0^*\mathrm{K}(E) \to \partial_1^*\mathrm{K}(E)$ on $\X_1$ is given by $\varrho(K) = \rho \circ (\id_L \otimes K) \circ \rho^{-1}$, where $\id_L : L \to L$ is the identity map.
\end{itemize}
We topologize the fiber of $\mathrm{K}(E)$ by using the operator norm, while its structure group by the compact open topology in \cite{A-Se}.

\begin{dfn}[Fredholm family] \label{dfn:fredholm_family}
Let $\X$ be a groupoid, $\phi : \X \to \pt//\Z_2$ a map of groupoids, and $(L, \tau, c)$ a $\phi$-twisted $\Z_2$-graded extension of $\X$. For a $(\phi, \tau, c)$-twisted vector bundle $(E, \epsilon, \rho, \gamma)$ on $\X$ with $Cl_{p, q}$-action, we define a fiber bundle $\Fred(E) \to \X$ as follows:
\begin{itemize}
\item
The fiber of the underlying fiber bundle $\Fred(E) \to \X_0$ at $x \in \X_0$ consists bounded operators $A : E_x \to E_x$ such that
\begin{itemize}
\item[(i)]
$A$ are skew-adjoint: $A^* = -A$.

\item[(ii)]
$A^2 + \id$ are compact.

\item[(iii)]
$\mathrm{Spec}(A) \subset [-i, i]$.

\item[(iv)]
$A$ are degree $1$, and anti-commute with the $Cl_{p, q}$-action, that is,
\begin{align*}
A \epsilon &= - \epsilon A, &
A \gamma( e ) &= - \gamma( e ) A,
\end{align*}
for any unit norm element $e \in \R^{p+q}$.
\end{itemize}

\item
The bundle isomorphism $\varrho : {}^\phi \partial_0^*\Fred(E) \to \partial_1^*\Fred(E)$ on $\X_1$ is given by $\varrho(A) = \rho \circ (\id_L \otimes A) \circ \rho^{-1}$, where $\id_L : L \to L$ is the identity map.
\end{itemize}
The fiber bundle $\Fred(E)$ is topologized by the following map
\begin{align*}
\Fred(E) &\to \mathrm{End}(E) \times \mathrm{K}(E), &
A &\mapsto (A, A^2 + \id),
\end{align*}
where the fibers of $\mathrm{End}(E) = \mathrm{Hom}(E, E)$ and $\mathrm{K}(E)$ are respectively given the compact open topology and the operator norm topology, and the structure groups of $\mathrm{End}(E)$ and $\mathrm{K}(E)$ are topologized by the compact open topology in the sense of \cite{A-Se}. The space of sections is defined by
$$
\Gamma(\X, \Fred(E)) =
\{ A \in \Gamma(\X_0, \Fred(E)) |\ 
\varrho \circ {}^\phi \partial_0^*A = \partial_1^*A \}.
$$
\end{dfn}

\smallskip

We write $\Fred(E)^* \subset \Fred(E)$ for the subbundle such that the fiber of the underlying fiber bundle $\Fred(E)^* \to \X_0$ consists of invertible operators. We also write $\Fred(E)^\dagger \subset \Fred(E)^*$ for the subbundle such that the fiber of the underlying bundle $\Fred(E)^\dagger \to \X_0$ consists of operators squaring to $-\id$. Therefore we have
\begin{align*}
\Gamma(\X, \Fred(E)^*)
&= \{ A \in \Gamma(\X, \Fred(E)) |\ 
\mbox{$A_x$ is invertible for each $x \in \X_0$} 
\}, \\
\Gamma(\X, \Fred(E)^\dagger)
&= \{ A \in \Gamma(\X, \Fred(E)) |\ 
\mbox{$A_x^2 = - \id$ for each $x \in \X_0$} 
\}.
\end{align*}
By functional calculus, $\Gamma(\X, \Fred(E)^\dagger) \subset \Gamma(\X, \Fred(E)^*)$ is a deformation retract, where the compact open topology are considered in the space of sections.

\begin{lem} \label{lem:extend_Clifford_action}
Let $\X$ be a local quotient groupoid, $\phi : \X \to \pt//\Z_2$ a map of groupoids, and $(L, \tau, c)$ a $\phi$-twisted $\Z_2$-graded extension of $\X$. Suppose that $E$ is a $(\phi, \tau, c)$-twisted locally universal vector bundle on $\X$ with $Cl_{p, q}$-action. Then $\Gamma(\X, \Fred(E)^\dagger)$ is non-empty and weakly contractible.
\end{lem}

\begin{proof}
Let $\Pi E$ be $E$ with its $\Z_2$-grading reversed. By the local universality, we have $E \cong E \oplus \Pi E$. It is easy to see that $E \oplus \Pi E \cong E \otimes \C^2$, where $\C^2$ is a $\Z_2$-graded vector space such that its even part and odd part are $1$-dimensional. On this $\C^2$, we can let $Cl_{1, 0}$ act by
$$
\gamma_* =
\left(
\begin{array}{cc}
0 & -1 \\
1 & 0
\end{array}
\right).
$$
Then $1 \otimes \gamma_* \in \Gamma(\X, \Fred(E \otimes \C^2)^\dagger)$. 
To see that $\Gamma(\X, \Fred(E)^\dagger)$ is weakly contractible (i.e.\ weakly homotopy equivalent to the point), we apply the reduction argument as in the proof of Lemma \ref{lem:locally_unversal_bundle} and Proposition A.19 in \cite{FHT1} (which is based on \cite{Se1}). Then, it suffices to show that $\Gamma(\X, \Fred(E)^\dagger)$ is weakly contractible when $\X$ is the quotient groupoid $\pt//G$ with $G$ a compact Lie group, $\phi : \X \to \pt//\Z_2$ and $c : \X \to \pt//\Z_2$ are associated to homomorphisms $\phi : G \to \Z_2$ and $c : G \to \Z_2$, and $E$ is a $(\phi, c)$-twisted locally universal bundle on $\pt//G$ with $Cl_{p, q}$-action. In this case, $\Gamma(\X, \Fred(E)^\dagger)$ is contractible, as shown in Appendix \ref{sec:mackey_decomposition} (Lemma \ref{lem:contractible_point_case}).
\end{proof}

From a groupoid $\X$, we can construct a groupoid $\X \times [0, 1]$ so as to be $(\X \times [0, 1])_i = \X_i \times [0, 1]$. If $E$ is a twisted bundle on $\X$, then the pull-back of $E$ under the projection $\X \times [0, 1] \to \X$ is identified with $E \times [0, 1]$. A \textit{homotopy} between $A_0, A_1 \in \Gamma(\X, \Fred(E))$ is defined to be a section $\tilde{A} \in \Gamma(\X \times [0, 1], \Fred(E \times [0, 1]))$ such that $\tilde{A}|_{\X \times \{ i \}} = A_i$ for $i = 0, 1$. In this case, $A_0$ and $A_1$ are said to be homotopic, and we write $A_0 \sim A_1$.

\begin{lem} \label{lem:abelian_group_structure}
Let $\X$ be a local quotient groupoid, $\phi : \X \to \pt//\Z_2$ a map of groupoids, $(L, \tau, c)$ a $\phi$-twisted $\Z_2$-graded extension of $\X$, and $E$ a $(\phi, \tau, c)$-twisted locally universal vector bundle $E$ on $\X$ with $Cl_{p, q}$-action. Then the set of homotopy classes of sections 
$$
\Gamma(\X, \Fred(E))/\sim
$$
is an abelian group.
\end{lem}

\begin{proof}
We can prove the lemma in a standard manner: The addition is induced from the direct sum $(A, A') \mapsto A \oplus A'$. The zero element is represented by invertible sections $A \in \Gamma(\X, \Fred(E)^*)$. The inverse is realized by reversing the $\Z_2$-grading of the underlying twisted vector bundle. To show the axiom about the inversion, let $\Pi E$ denote the twisted bundle $E$ with its $\Z_2$-grading reversed. The direct sum $E \oplus \Pi E$ is isomorphic to the tensor product $E \otimes \Delta_{1, 0}^{\C}$ of $E$ and an irreducible $\Z_2$-graded complex $Cl_{1, 0}$-module $\Delta_{1, 0}^{\C}$. If we denote by $\gamma \in \Gamma(\X, \Fred(E \otimes \Delta_{1, 0}^{\C})^\dagger)$ the action of the generator of $Cl_{1, 0}$ on $\Delta_{1,0}^{\C}$, then $A \oplus \Pi A$ is homotopic to $\gamma$ by the homotopy $(A \oplus \Pi A) \cos \theta + \gamma\sin\theta$ for any $A \in \Gamma(\X, \Fred(E))$. 
\end{proof}

Now, suppose that, for a groupoid $\X$ and $\phi = (F : \tilde{\X} \to \X, \phi) \in \Phi(\X)$, we are given a $\phi$-twist on $\X$ consisting of a local equivalence $\tilde{F} : \tilde{\tilde{\X}} \to \tilde{\X}$ and a $\tilde{F}^*\phi$-twisted $\Z_2$-graded extension $(L, \tau, c)$ of $\tilde{\tilde{\X}}$. Suppose also that $E$ is a $(\tilde{F}^*\phi, \tau, c)$-twisted vector bundle on $\tilde{\tilde{X}}$ with $Cl_{p, q}$-action. By the nature of local equivalences \cite{FHT1}, fiber bundles on $\tilde{\tilde{X}}$ are in bijective correspondence with those on $\tilde{\X}$ under the pull-back, and the pull-back also induces a homeomorphism of the spaces of sections. This can be generalized to $\phi$-twisted bundles, so that the $\tilde{F}^*\phi$-twisted vector bundle $\mathrm{End}(E) \to \tilde{\tilde{\X}}$ is isomorphic the pull-back of a $\phi$-twisted vector bundle on $\tilde{\X}$ under $\tilde{F}$. As a result, the fiber bundle $\Fred(E) \to \tilde{\tilde{\X}}$ is isomorphic to the pull-back under $\tilde{F}$ of a $\phi$-twisted fiber bundle $\Fred(\tau) \to \tilde{X}$, and $\Gamma(\tilde{\X}, \Fred(\tau)) \cong \Gamma(\tilde{\tilde{\X}}, \Fred(E))$.

\begin{dfn} \label{dfn:bigraded_K}
Let $\X$ be a local quotient groupoid, $\phi = (F : \tilde{\X} \to \X, \phi) \in \Phi(\X)$ an object, and $(\tau, c) = (\tilde{F} : \tilde{\tilde{\X}} \to \tilde{\X}, L, \tau, c)$ a $\phi$-twist on $\X$. We define a group by
$$
{}^\phi K^{(\tau, c) + (p, q)}(\X)
= \Gamma(\tilde{\X}, \Fred(\tau))/\sim
\cong \Gamma(\tilde{\tilde{\X}}, \Fred(E))/\sim
$$
where $\Fred(\tau) \to \tilde{\X}$ is the $\phi$-twisted bundle such that $\tilde{F}^*\Fred(\tau) \cong \Fred(E)$ for a $(\tilde{F}^*\phi, \tau, c)$-twisted locally universal vector bundle $E \to \tilde{\tilde{\X}}$ with $Cl_{p, q}$-action. In the case that $\X$ is a quotient groupoid $X//G$, we may write
$$
{}^\phi K^{(\tau, c) + (p, q)}(X//G)
= {}^\phi K^{(\tau, c) + (p, q)}_G(X).
$$
\end{dfn}

The group ${}^\phi K^{(\tau, c) + (p, q)}(\X)$ is independent of the choice of $E$, because of the uniqueness of locally universal bundles up to unitary isomorphisms. As in the case of twisted complex $K$-theory \cite{FHT1}, a local equivalence of groupoids induces an isomorphism by pull-back, and hence ${}^\phi K^{(\tau, c) + (p, q)}(\X)$ is an invariant of the weak equivalence class of the groupoid $\X$ equipped with the twisting data $\phi$ and $(\tau, c)$.

\begin{lem}[weak periodicity] \label{lem:weak_periodicity}
Let $\X$ be a local quotient groupoid, $\phi \in \Phi(\X)$ an object, and $(\tau, c)$ a $\phi$-twist on $\X$. There are natural isomorphisms
\begin{align*}
{}^\phi K^{(\tau, c) + (p, q)}(\X) 
&\cong {}^\phi K^{(\tau, c) + (p+1, q+1)}(\X) \\
&\cong {}^\phi K^{(\tau, c) + (p+8, q)}(\X) 
\cong {}^\phi K^{(\tau, c) + (p, q+8)}(\X).
\end{align*}
In the case that $\phi$ is trivial, there are natural isomorphisms
\begin{align*}
K^{(\tau, c) + (p, q)}(\X) 
&\cong K^{(\tau, c) + (p+1, q+1)}(\X) \\
&\cong K^{(\tau, c) + (p+2, q)}(\X) 
\cong K^{(\tau, c) + (p, q+2)}(\X).
\end{align*}
\end{lem}

\begin{proof}
Let us consider ${}^\phi K^{(\tau, c) + (p, q)}(\X) \cong  {}^\phi K^{(\tau, c) + (p+1, q+1)}(\X)$. To realize this isomorphism, we let $\Delta_{1,1}$ be an irreducible $\Z_2$-graded real $Cl_{1,1}$-module. Concretely, we can choose $\Delta_{1,1} = \R^2$ and
\begin{align*}
\epsilon 
&=
\left(
\begin{array}{rr}
1 & 0 \\
0 & -1
\end{array}
\right),
&
\gamma_1
&=
\left(
\begin{array}{rr}
0 & -1 \\
1 & 0
\end{array}
\right),
&
\gamma_2
&=
\left(
\begin{array}{rr}
0 & 1 \\
1 & 0
\end{array}
\right).
\end{align*}
Its complexification $\Delta_{1, 1}^{\C} = \Delta_{1, 1} \otimes \C$, being irreducible also, has the obvious `Real' structure from the complex conjugation on $\C$, so that we can regard it as an $\mathrm{id}_{\Z_2}$-twisted vector bundle on $\pt//\Z_2$ with $Cl_{1,1}$-action. Furthermore, we pull this bundle back to $\X$ by the map $\phi : \X \to \pt//\Z_2$ to get
$$
\Delta_{1, 1}^{\C} \in {}^{\phi}\Vect^{(1, 1)}(\X).
$$
For a $(\phi, \tau, c)$-twisted locally universal twisted vector bundle $E$ on $\X$ with $Cl_{p, q}$-action, the tensor product $E \otimes \Delta_{1,1}^{\C}$ is a $(\phi, \tau, c)$-twisted locally universal twisted vector bundle on $\X$ with $Cl_{p+1, q+1}$-action. We then consider
\begin{align*}
&\Gamma(\X, \Fred(E)) \to \Gamma(\X, \Fred(E \otimes \Delta_{1,1}^{\C})), &
&a \mapsto a \otimes 1.
\end{align*}
We can directly see that any odd map $A : E \otimes \Delta_{1,1}^{\C} \to E \otimes \Delta_{1,1}^{\C}$ such that $A(1 \otimes \gamma_i) = - (1 \otimes \gamma_i)A$ is uniquely expressed as $A = a \otimes 1$ by using an odd map $a : E \to E$. As a result, $\Gamma(\X, \Fred(E)) \to \Gamma(\X, \Fred(E \otimes \Delta_{1,1}^{\C}))$ is a homeomorphism, and hence induces an isomorphism 
$$
{}^\phi K^{(\tau, c) + (p, q)}(\X) \to {}^\phi K^{(\tau, c) + (p+1, q+1)}(\X).
$$
The other isomorphisms follow from this: Iterating it, we get
$$
{}^\phi K^{(\tau, c) + (p, q)}(\X) \to {}^\phi K^{(\tau, c) + (p+4, q+4)}(\X).
$$
In general, if $\gamma_i$, ($i = 1, \ldots, 4$) realize a real $Cl_{4,0}$-module, then $\gamma'_i = \gamma_i\gamma_1 \cdots \gamma_4$, ($i = 1, \ldots, 4$) realize a real $Cl_{0, 4}$-module. This construction induces natural isomorphisms
$$
{}^\phi K^{(\tau, c) + (p+8, q)}(\X) \cong 
{}^\phi K^{(\tau, c) + (p+4, q+4)}(\X) \cong 
{}^\phi K^{(\tau, c) + (p, q+8)}(\X).
$$
In the case that $\phi$ is trivial, there are natural identifications of complex modules over $Cl_{2,0}$, $Cl_{1,1}$ and $Cl_{0, 2}$, so that
$$
K^{(\tau, c) + (p+2, q)}(\X) \cong 
K^{(\tau, c) + (p+1, q+1)}(\X) \cong 
K^{(\tau, c) + (p, q+2)}(\X),
$$
and the lemma is established.
\end{proof}

\begin{rem}
There are two equivalent variants of the fiber bundle $\Fred(E)$. A variant is the fiber bundle given by dropping the spectral condition $\mathrm{Spec}(A) \subset [-i, i]$ in Definition \ref{dfn:fredholm_family}. The resulting fiber bundle has a deformation retract to $\Fred(E)$, and we can use it to formulate the Freed-Moore $K$-theory. Another variant is to use self-adjoint operators instead of skew-adjoint operators. Its detail will be given in Remark \ref{rem:twist_change}.
\end{rem}

\subsection{The Bott periodicity}

Let $\X$ be a groupoid, $\phi : \X \to \pt//\Z_2$ a map of groupoid, $(L, \tau, c)$ a $\phi$-twisted $\Z_2$-graded extension of $\X$, and $E$ a $(\phi, \tau, c)$-twisted vector bundle $E$ of $\X$ with $Cl_{p, q}$-action. For a full subgroupoid $\Y \subset \X$, we write
$$
\Gamma(\X, \Y, \Fred(E))
= \{ A \in \Gamma(\X, \Fred(E)) |\ A|_{\Y} \in \Gamma(\Y, \Fred(E)^*) \}
$$
for the space of sections $A$ of $\Fred(E)$ such that $A_x : E_x \to E_x$ is invertible for all $x \in \Y_0$. A homotopy between such sections is defined by using sections in $\Gamma(\X \times [0, 1], \Y \times [0, 1], \Fred(E \times [0, 1]))$. For an object $\phi \in \Phi(\X)$ and a $\phi$-twist $(\tau, c)$, the space of sections $\Gamma(\X, \Y, \Fred(\tau))$ and their homotopy $\sim$ are defined in the obvious way.

\begin{dfn} \label{dfn:Freed_Moore_K_theory}
Let $\X$ be a local quotient groupoid, $\phi = (F : \tilde{\X} \to \X, \phi) \in \Phi(\X)$ an object, and $(\tau, c) = (\tilde{F} : \tilde{\tilde{\X}} \to \tilde{\X}, L, \tau, c)$ a $\phi$-twist on $\X$. 
\begin{itemize}
\item[(a)]
For a full subgroupoid $\Y \subset \X$, we define
$$
{}^\phi K^{(\tau, c)}(\X, \Y)
= \Gamma(\tilde{\X} \cup \tilde{\Y} \times [0, 1], 
\tilde{\Y} \times \{ 1 \}, \Fred(\tau \times [0, 1]))/\sim \\
$$
where $\tilde{\X} \cup \tilde{\Y} \times [0, 1] \subset \tilde{\X} \times [0, 1]$ is the mapping cylinder of the full subgroupoid $\tilde{\Y} = \tilde{F}^{-1}(\Y) \subset \tilde{\X}$, $\Fred(\tau)$ is the $\phi$-twisted bundle on $\tilde{\X}$ such that $\tilde{F}^*\Fred(\tau) \cong \Fred(E)$ for a $(\tilde{F}^*\phi, \tau, c)$-twisted locally universal vector bundle $E \to \tilde{\tilde{\X}}$ with $Cl_{0, 0}$-action, and $\tau \times [0, 1]$ is the pull-back of $\tau$ under the projection $\tilde{\X} \times [0, 1] \to \tilde{\X}$.

\item[(b)]
For a non-negative integer $n \ge 0$, we define
$$
{}^\phi K^{(\tau, c) - n}(\X, \Y) 
=
{}^\phi K^{(\tau, c)}(\X \times [0, 1]^n, \Y \times [0, 1]^n \cup
\X \times \partial [0, 1]^n).
$$
\end{itemize}
\end{dfn}

\begin{thm} \label{thm:Atiyah_Singer_map}
Let $\X$ be a local quotient groupoid, $\phi \in \Phi(\X)$ an object, and $(\tau, c)$ a $\phi$-twist on $\X$. For $n \ge 0$, there is a natural isomorphism of groups
$$
{}^\phi K^{(\tau, c) - n}(\X)
\cong
{}^\phi K^{(\tau, c) + (n, 0)}(\X).
$$
\end{thm}

\begin{proof}
The proof is essentially the same as in the complex case \cite{FHT1} (Proposition A.41): For a $(\phi, \tau, c)$-twisted locally universal vector bundle $E \to \X$ with $Cl_{1, 0}$-action, we have the Atiyah-Singer map \cite{A-Si}
$$
\mathrm{AS} : \
\Gamma(\X, \Fred(E)) \to 
\Gamma(\X \times [0, 1], \X \times \{ 0, 1 \}, \Fred(E \times [0, 1]))
$$
given by $A \mapsto \gamma_1 \cos\pi t + A \sin\pi t$, where $t \in [0, 1]$, and $\gamma_1 = \gamma(e_1)$ is the action of the generator $e_1 \in Cl_{1, 0}$. The iteration of this map defines a homomorphism
$$
{}^\phi K^{(\tau, c) + (n, 0)}(\X) \to {}^\phi K^{(\tau, c) - n}(\X).
$$
To prove that this homomorphism is bijective, we show that $\mathrm{AS}$ induces a weak homotopy equivalence on the spaces of sections. As before, thanks to the reduction argument as in Lemma \ref{lem:locally_unversal_bundle} and \cite{FHT1} (Proposition A.19), it suffices to consider the case of the quotient groupoid $\X = \pt//G$ with $G$ a compact Lie group and trivial $\tau$. In this case, the map $\mathrm{AS}$ provides a homotopy equivalence, as will be shown in Appendix \ref{sec:mackey_decomposition} (Lemma \ref{lem:atiyah_singer_map_point_case}).
\end{proof}

By the periodicities in Lemma \ref{lem:weak_periodicity}, we get:

\begin{cor}[Bott periodicity] \label{cor:bott_periodicity}
For $n \ge 0$, there is a natural isomorphism 
$$
{}^\phi K^{(\tau, c) - n}(\X) \cong {}^\phi K^{(\tau, c) - n - 8}(\X).
$$
If $\phi$ is trivial, then there is a natural isomorphism
$$
{}^\phi K^{(\tau, c) - n}(\X) \cong {}^\phi K^{(\tau, c) - n - 2}(\X).
$$
\end{cor}

\begin{cor} \label{cor:degree_correspondence}
For $- p+ q \le 0$, we have a natural isomorphism 
$$
{}^\phi K^{(\tau, c) - p + q}(\X) \cong {}^\phi K^{(\tau, c) + (p, q)}(\X).
$$
\end{cor}

\begin{thm} \label{thm:axiom}
Let $\X$ be a local quotient groupoid, $\phi \in \Phi(\X)$ an object, and $(\tau, c)$ a $\phi$-twist on $\X$. We can extend the $K$-group ${}^\phi K^{(\tau, c) + n}(\X)$ to define ${}^\phi K^{(\tau, c) + n}(\X, \Y)$ for a full subgroupoid $\Y \subset \X$ and $n \in \Z$ so that the Bott periodicity holds true:
$$
{}^\phi K^{(\tau, c) + n}(\X, \Y)
\cong
{}^\phi K^{(\tau, c) + n - 8}(\X, \Y).
$$
These groups have the following properties.
\begin{itemize}
\item[(a)]
(the homotopy axiom)
Let $\X'$ be another local quotient groupoid, and $\Y' \subset \X'$ a full subgroupoid. Let $f_0, f_1 : \X' \to \X$ be maps of groupoids such that $f_i(\Y') \subset \Y$. If $\tilde{f} : \X' \times [0, 1] \to \X$ is such that $\tilde{f}(\Y' \times \{ t \}) \subset \Y$ for all $t \in [0, 1]$, then there is an isomorphism of twists $\beta_{\tilde{f}} : f_0^*\tau \to f_1^*\tau$, and the following diagram becomes commutative
$$
\begin{CD}
{}^\phi K^{(\tau, c) + n}(\X, \Y) @>{f_0^*}>>
{}^{f_0^*\phi} K^{(f_0^*\tau, f_0^*c) + n}(\X', \Y') \\
@| @AA{\beta_{\tilde{f}}^*}A \\
{}^\phi K^{(\tau, c) + n}(\X, \Y) @>{f_1^*}>>
{}^{f_1^*\phi} K^{(f_1^*\tau, f_1^*c) + n}(\X', \Y').
\end{CD}
$$ 

\item[(b)]
(the excision axiom)
Let $\mathcal{A}$ and $\mathcal{B}$ be closed full subgroupoids in $\X$. Then the inclusion $\mathcal{A} \to \mathcal{A} \cup \mathcal{B}$ induces the isomorphisms
$$
{}^{\phi_{A \cup B}} 
K^{(\tau_{A \cup B}, c_{A \cup B}) + n}
(\mathcal{A} \cup \mathcal{B}, \mathcal{B})
\cong
{}^{\phi_A} K^{(\tau_A, c_A) + n}(\mathcal{A}, \mathcal{A} \cap \mathcal{B}),
$$
where $\phi_{A \cup B} = \phi|_{\mathcal{A} \cup \mathcal{B}}$, etc.\ mean the restrictions.

\item[(c)]
(the exactness axiom)
For a full subgroupoid $\Y \subset \X$, there is a long exact sequence
$$
\cdot\cdot \to
{}^\phi K^{(\tau, c) + n-1}(\Y) \to
{}^\phi K^{(\tau, c) + n}(\X, \Y) \overset{j^*}{\to}
{}^\phi K^{(\tau, c) + n}(\X) \overset{i^*}{\to}
{}^\phi K^{(\tau, c) + n}(\Y) \to
\cdot\cdot,
$$
in which $j^*$ is induced from the forgetful functor of $\Y$ and $i^*$ from the inclusion $i : \Y \to \X$. The restriction of twisting data ($\phi \mapsto \phi|_{\mathcal{Y}}$) is omitted.

\item[(d)]
(the additivity axiom)
For a family $\X_\lambda$ of local quotient groupoids and their full subgroupoids $\Y_\lambda \subset \X_\lambda$, the inclusions $\X_\lambda \to \sqcup \X_\lambda$ induce an isomorphism
$$
{}^{\sqcup \phi_\lambda}K^{(\sqcup \tau_\lambda, \sqcup c_\lambda) + n}
(\sqcup \X_\lambda, \sqcup \Y_\lambda)
\cong
\prod_\lambda 
{}^{\phi_\lambda} K^{(\tau_\lambda, c_\lambda) + n}(\X_\lambda, \Y_\lambda).
$$
\end{itemize}
\end{thm}

\begin{proof}
With the Bott periodicity in Corollary \ref{cor:bott_periodicity}, the proof of the theorem is a standard and rather formal procedure: In view of the so-called the cofibration (or Puppe) sequence, we get the non-positive part of the long exact sequence in (c). Using this part, the map $\mathrm{AS}$ in the proof of Theorem \ref{thm:Atiyah_Singer_map} induces an isomorphism
$$
{}^\phi K^{(\tau, c) + n}(\X, \Y) \cong
{}^\phi K^{(\tau, c) + n - 8}(\X, \Y).
$$
Based on this periodicity, for a positive integer $n > 0$, we define
$$
{}^\phi K^{(\tau, c) + n}(\X, \Y)
= {}^\phi K^{(\tau, c) + n - 8k}(\X, \Y),
$$
where $k$ is an integer such that $n - 8k < 0$. Because of this definition, the non-negative part of the long exact sequence is extended to the complete long exact sequence in (c). The homotopy axiom (a) follows from Lemma \ref{lem:homotopy_property} and the definition of the $K$-theory in terms of a homotopy. The excision axiom (b) follows from a deformation and an extension of a section $A \in \Gamma(\mathcal{A}, \mathcal{A} \cap \mathcal{B}, \Fred(\tau))$ by using Lemma \ref{lem:extend_Clifford_action}. Finally, (d) just follows from the definition of the $K$-group.
\end{proof}

The tensor product functor on the category of twisted vector bundles induces a multiplication on the $K$-groups
$$
\otimes :\ 
{}^\phi K^{(\tau, c) + n}(\X, \Y) \times
{}^\phi K^{(\tau', c') + n'}(\X, \Y') \to
{}^\phi K^{((\tau, c) + (\tau', c')) + (n + n')}(\X, \Y \cup \Y').
$$
Thus, for example, ${}^\phi K^0(\X)$ is a ring, and ${}^\phi K^{(\tau, c) + n}(\X, \Y)$ is a module over ${}^\phi K^0(\X)$. In the following, for a quotient groupoid $X//G$, we may apply the notation
\begin{align*}
{}^\phi K^{(\tau, c) + n}(X//G) = {}^\phi K^{(\tau, c) + n}_G(X).
\end{align*}
In this case, ${}^\phi K^{(\tau, c) + n}_G(X)$ is a module over the ring ${}^{\phi_G} K^0_G(\pt)$, where $\phi_G \in \Phi(\pt//G)$ is the restriction of $\phi \in \Phi(X//G)$, which is always equivalent to the object associated to a homomorphism $\phi_G : G \to \Z_2$.

\subsection{Twist and degree shift}
\label{subsec:twist_and_degree_shift}

For the quotient groupoid $\pt//\Z_2$, the identity homomorphism $\id : \Z_2 \to \Z_2$ defines a non-trivial object $\id \in \Phi(\pt//\Z_2)$. There are then two distinguished $\id$-twisted $\Z_2$-graded extensions of $\pt//\Z_2$.

\begin{itemize}
\item
The $\id$-twisted $\Z_2$-graded extension $\tau_{\id} = (\Z_2 \times \C, \tau_{\id}, 1)$ consisting of the product line bundle $\Z_2 \times \C \to \Z_2$, the $2$-cocycle $\tau_{\id} \in Z^2_{\mathrm{group}}(\Z_2; U(1)_{\id})$ given by $\tau_{\id}((-1)^{m_1}, (-1)^{m_2}) = \exp \pi i m_1m_2$,
$$
\begin{array}{|c|c|c|}
\hline
\tau_{\id}(g, h) & h = 1 & h = -1 \\
\hline
g = 1 & 1 & 1 \\
\hline
g = -1 & 1 & -1 \\
\hline
\end{array}
$$
and the trivial map of groupoids $1 : \pt//\Z_2 \to \pt//\Z_2$ induced from the trivial homomorphism $1 : \Z_2 \to \Z_2$.

\item
The $\id$-twisted $\Z_2$-graded extension $c_{\id} = (\Z_2 \times \C, 1, c_{\id})$ consisting of the product line bundle $\Z_2 \times \C \to \Z_2$, the trivial $2$-cocycle $1 \in Z^2_{\mathrm{group}}(\Z_2; U(1)_{\id})$ and the non-trivial map of groupoids $c_{\id} : \pt//\Z_2 \to \pt//\Z_2$ given by the identity homomorphism $c_{\id} = \id : \Z_2 \to \Z_2$.
\end{itemize}

The $\id$-twisted extension $\tau_{\id}$ can be seen as an ungraded twist, and generates
$$
H^3(\pt//\Z_2; \Z_\phi) \cong
H^3_{\mathrm{group}}(\pt; \Z_\phi) \cong 
H^3_{\mathrm{group}}(\pt; U(1)_\phi) \cong \Z_2,
$$ 
while $c_{\id}$ can be seen as the datum of the $\Z_2$-grading, and generates
$$
H^1(\pt//\Z_2; \Z_2) \cong
H^1_{\mathrm{group}}(\pt; \Z_2)
= \mathrm{Hom}(\Z_2, \Z_2) \cong \Z_2.
$$
A $c_{\id}$-twisted vector bundle on $\pt//\Z_2$ just amounts to a $\Z_2$-graded Hilbert space $E$ equipped with an odd anti-unitary map $\rho : E \to E$ such that $\rho^2 = \id_E$. By the sign rule, the square of their tensor product is calculated as
$$
(\rho \otimes \rho) \circ (\rho \otimes \rho)
= - (\rho \circ \rho) \otimes (\rho \circ \rho)
= - \id_E \otimes \id_E
= - \id_{E \otimes E}.
$$
The sign cannot be eliminated by multiplying the anti-unitary map $\rho$ with scalars. This explains that $(0, c_{\id}) + (0, c_{\id}) = (\tau_{\id}, 0)$ and $(\tau_{\id}, 0) + (\tau_{\id}, 0) = 0$ in the group
$$
\pi_0({}^{\id}\Twist(\pt//\Z_2)) \cong 
H^3_{\Z_2}(\pt; \Z_{\id}) \times H^1_{\Z_2}(\pt; \Z_2)
\cong \Z_4,
$$
so that $c_{\id}$ is a generator of this group.

Let $\X$ be a groupoid, and $\phi = \{ F : \tilde{\X} \to \X, \phi : \tilde{\X} \to \pt//\Z_2 \}$ an object in $\Phi(\X)$. By the pull-back under $\phi : \tilde{\X} \to \pt//\Z_2$, the $\id$-twisted $\Z_2$-graded extensions $\tau_{\id}$ and $c_{\id}$ define $\phi$-twisted $\Z_2$-graded extensions $\tau_\phi = \phi^*\tau_{\id}$ and $c_\phi = \phi^*c_{\id}$ of $\tilde{\X}$. Hence we have $\phi$-twists $\tau_\phi$ and $c_\phi$ on $\X$.

\begin{thm} \label{thm:degree_shift}
Let $\X$ be a local quotient groupoid, $\phi \in \Phi(\X)$ an object, and $(\tau, c)$ a $\phi$-twist on $\X$. For a full subgroupoid $\Y \subset \X$ and $n \in \Z$, there are natural isomorphisms
\begin{align*}
{}^\phi K^{(\tau, c) + c_{\phi} + n}(\X, \Y) 
&\cong 
{}^\phi K^{(\tau, c) + n + 2}(\X, \Y), \\
{}^\phi K^{(\tau, c) + \tau_{\phi} + n}(\X, \Y) 
&\cong 
{}^\phi K^{(\tau, c) + n + 4}(\X, \Y), \\
{}^\phi K^{(\tau, c) + (\tau_{\phi}, c_{\phi}) + n}(\X, \Y) 
&\cong 
{}^\phi K^{(\tau, c) + n + 6}(\X, \Y).
\end{align*}
\end{thm}

\begin{proof}
We define $\Delta \in {}^{\mathrm{id}}\Vect^{c_{\id} + (2, 0)}(\pt//\Z_2)$ as follows: The underlying vector space is $\Delta = \C \oplus \C$ with $\Delta^k = \C$. The twisted $\Z_2$-action $C$ and the $Cl_{2, 0}$-action $\gamma_i$ are 
\begin{align*}
C &=
\left(
\begin{array}{cc}
0 & 1 \\
1 & 0 
\end{array}
\right)K, 
&
\gamma_1 &=
\left(
\begin{array}{cc}
0 & -1 \\
1 & 0 
\end{array}
\right), 
&
\gamma_2 &=
\left(
\begin{array}{cc}
0 & i \\
i & 0 
\end{array}
\right), 
\end{align*}
where $K$ is the complex conjugation. For $\phi = \{ F : \tilde{\X} \to \X, \phi : \tilde{\X} \to \pt//\Z_2 \} \in \Phi(\X)$, we take the pull-back of the twisted bundle $\Delta$ above under $\phi$ to get
$$
\Delta \in {}^\phi\Vect^{c_\phi + (2, 0)}(\tilde{\X}).
$$
Thus, for a $(\phi, \tau, c)$-twisted locally universal vector bundle $E$ on $\tilde{\X}$ with $Cl_{p, q}$-action, the tensor product defines a map
\begin{align*}
&\Gamma(\tilde{\X}, \Fred(E)) \to 
\Gamma(\tilde{\X}, \Fred(E \otimes \Delta)), &
A& \mapsto A \otimes 1
\end{align*}
and this eventually induces a homomorphism
$$
\delta : 
{}^\phi K^{(\tau, c) + n}(\X, \Y) \to
{}^\phi K^{(\tau, c) + c_{\phi} + n -2}(\X, \Y).
$$
Recall that $c_{\phi} + c_{\phi} = \tau_{\phi}$ and $\tau_{\phi} + \tau_{\phi} = 0$. Therefore the present theorem will be established when $\delta$ is shown to be bijective. To prove the bijectivity of $\delta$, it is enough to consider the case of $\Y = \emptyset$, because of the exactness axiom. Now, let us consider the composition 
$$
\delta^4 : 
{}^\phi K^{(\tau, c) + n}(\X) \to
{}^\phi K^{(\tau, c) + n - 8}(\X).
$$
This map is induced from the tensor product with the twisted representation $\Delta^{\otimes 4}$, which is an $\id$-twisted representation of $\Z_2$ with $Cl_{8, 0}$-action
$$
\Delta^{\otimes 4}
\in {}^{\id}\Vect^{(8, 0)}(\pt//\Z_2).
$$
An $\id$-twisted representation of $\Z_2$ is nothing but a complex vector space with an anti-linear involution, or equivalently a real structure. By the operation of taking the real part, $\id$-twisted representations of $\Z_2$ with $Cl_{8, 0}$-action are in one to one correspondence with $\Z_2$-graded real representations of $Cl_{8, 0}$. The $\Z_2$-graded real representation of $Cl_{8, 0}$ corresponding to $\Delta^{\otimes 4}$ has the dimension $2^4 = 16$, and hence is an irreducible representation. Such an irreducible representation realizes the periodicity in Lemma \ref{lem:weak_periodicity}, so that $\delta$ turns out to be bijective.
\end{proof}

\subsection{Reproduction of familiar $K$-theories}

As mentioned in \S\ref{sec:introduction}, we can recover familiar $K$-theories by specifying twists. We review here some examples.

\subsubsection{Twisted equivariant complex $K$-theory}

For a local quotient groupoid $\X$, if $\phi \in \Phi(\X)$ is trivial, then the twisted $K$-theory in \cite{FHT1} is recovered. In particular, for the quotient groupoid $\X = X//G$ associated to an action of a compact Lie group $G$ on a compact Hausdorff space $X$, we recover the twisted $G$-equivariant complex $K$-theory
$$
{}^\phi K^{(\tau, c) + n}(X//G) = K^{(\tau, c) + n}_G(X).
$$
In this case, the twists $(\tau, c)$ are classified by the Borel equivariant cohomology
$$
H^3_G(X; \Z) \times H^1_G(X; \Z_2).
$$
If $G = \Z_2$, then there is a distinguished twist $(0, c)$ coming from 
$$
H^1_{\Z_2}(\pt; \Z_2) 
\cong \mathrm{Hom}(\Z_2, \Z_2) = \Z_2.
$$
We can identify the $K$-theory with this twist with a variant of $K$-theory $K^{c + n}_{\Z_2}(X) \cong K^n_{\pm}(X)$ in \cite{A-H,W}.

\subsubsection{Twisted equivariant $KO$-theory}

Let $G$ be a compact Lie group acting on a compact Hausdorff space $X$. By this $G$-action and the trivial $\Z_2$-action, we define an action of $\Z_2 \times G$ on $X$. We let $p_{\Z_2} : \Z_2 \times G \to \Z_2$ be the projection onto the $\Z_2$-factor, which defines an object $p_{\Z_2} \in \Phi(\pt//(\Z_2 \times G))$. In this case, we have the identification with the twisted $G$-equivariant real $K$-theory
$$
{}^{p_{\Z_2}} K^{(\tau, c) + n}(X//(\Z_2 \times G))
\cong KO^{(\tau, c) + n}_G(X).
$$
The twists $(\tau, c)$ are classified by the equivariant cohomology
$$
H^3_{\Z_2 \times G}(X; \Z_{p_{\Z_2}}) \times
H^1_{\Z_2 \times G}(X; \Z_2).
$$
We can see that
\begin{align*}
H^3_{\Z_2 \times G}(X; \Z_{p_{\Z_2}})
&\cong
H^2_G(X; \Z_2) \oplus H^0_G(X; \Z_2), \\
H^1_{\Z_2 \times G}(X; \Z_2)
&\cong
H^1_G(X; \Z_2) \oplus H^0_G(X; \Z_2).
\end{align*}
The factors $H^2_G(X; \Z_2)$ and $H^1_G(X; \Z_2)$ are consistent with the twists for $KO$-theory in \cite{D-K}. Notice that a $G$-equivariant complex line bundle $L \to X$ defines a twisted extension in this case. The twist given by $L$ is classified by the image of the equivariant Chern class $c_1^G(L) \in H^2_G(X; \Z)$ under the mod $2$ reduction $H^2_G(X; \Z) \to H^2_G(X; \Z_2)$. The $KO$-theory twisted by $L$ is the $G$-equivariant version of the twisted $KO$-theory in \cite{A-R}. The remaining factors 
\begin{align*}
H^0_G(X; \Z_2) &\subset H^3_{\Z_2 \times G}(X; \Z_{p_{\Z_2}}), &
H^0_G(X; \Z_2) &\subset H^1_{\Z_2 \times G}(X; \Z_2)
\end{align*}
are the contributions of the twists $c_\phi$ and $\tau_\phi$ in \S\S\ref{subsec:twist_and_degree_shift} with $\phi = p_{\Z_2}$. In view of the values of the $2$-cocycle defining $\tau_\phi$, we find that the equivariant $KO$-theory twisted by $\tau_\phi$ is the $K$-theory of $G$-equivariant quaternionic vector bundles.

\subsubsection{$KR$-theory}

Let us consider the quotient groupoid $\X = X//\Z_2$ associated to a compact Hausdorff space $X$ with an action of $\Z_2$. The identity homomorphism $\id : \Z_2 \to \Z_2$ defines an object $\id \in \Phi(X//\Z_2)$. Then we recover Atiyah's $KR$-theory \cite{A2}
$$
{}^{\id} K^{n}(X//\Z_2)
\cong KR^{n}(X).
$$
The twists $(\tau, c)$ in this case are classified by 
$$
H^3_{\Z_2}(X; \Z_{\id}) \times H^1_{\Z_2}(X; \Z_2),
$$
and we can define a twisted $KR$-theory by
$$
KR^{(\tau, c) + n}(X) = {}^{\id} K^{(\tau, c)+n}(X//\Z_2).
$$
The twist $\tau_{\id}$ in \S\S\ref{subsec:twist_and_degree_shift} lives in the factor $H^3_{\Z_2}(X; \Z_{\id})$, and the $K$-theory with this twist provides Dupont's `Symplectic' $K$-theory \cite{Dup}. 

We notice that twisted $KR$-theories are already introduced by Moutuou \cite{Mou1,Mou2,Mou3} and also by Fok \cite{Fok1,Fok2}. The twisted $KR$-theory ${}^{\id} K^{(\tau, c) + n}(X//\Z_2)$ above is anticipated to reproduce their twisted $KR$-theories. In particular, the twisted $G$-equivariant $KR$-theory $KR_G^{\tau}(X)$ for a `Real' $G$-space of $X$ in the sense of \cite{Fok1,Fok2} would be identified with ${}^{\phi}K^\tau(X//(\Z_2 \ltimes G))$, where the semi-direct product $\Z_2 \ltimes G$ is defined by using the `Real' structure on $G$ and $\phi : \Z_2 \ltimes G \to \Z_2$ is the projection.

\subsection{Finite rank realizability}
\label{subsec:finite_rank}

We here compare the Fredholm formulation of ${}^\phi K^{(\tau, c) + 0}(\X)$ with its finite-dimensional formulation in \cite{F-M}.

\begin{dfn} \label{dfn:finite_rank_freed_moore_K}
Let $\X$ be the quotient groupoid $X//G$ associated to an action of a finite group $G$ on a compact Hausdorff space $X$, $\phi : X//G \to \pt//\Z_2$ the map of groupoids associated to a homomorphism $\phi : G \to \Z_2$, and $(L, \tau, c)$ a $\phi$-twisted $\Z_2$-graded extension of $X//G$.
\begin{itemize}
\item[(a)]
We define ${}^\phi\Vect^{(\tau, c) + (p, q)}_G(X)_{\mathrm{fin}} \subset {}^\phi\Vect^{(\tau, c)+(p, q)}(X//G)$ to be the full subcategory whose objects are $(\phi, \tau, c)$-twisted vector bundles $E$ on $X//G$ with $Cl_{p, q}$-action such that the fibers of $E$ are finite dimensional.

\item[(b)]
We define ${}^\phi K^{(\tau, c)+n}_G(X)_{\mathrm{fin}}$ as the quotient monoid
$$
{}^\phi K^{(\tau, c)+n}_G(X)_{\mathrm{fin}}
= \pi_0({}^\phi\Vect^{(\tau, c) + (n, 0)}_G(X)_{\mathrm{fin}})/
\pi_0({}^\phi\Vect^{(\tau, c) + (n+1, 0)}_G(X)_{\mathrm{fin}}).
$$
\end{itemize}
\end{dfn}

Note that ${}^\phi\Vect^{(\tau, c) + (n+1, 0)}_G(X)_{\mathrm{fin}}$ is a full subcategory of ${}^\phi\Vect^{(\tau, c) + (n, 0)}_G(X)_{\mathrm{fin}}$, so that $\pi_0({}^\phi\Vect^{(\tau, c) + (1, 0)}_G(X)_{\mathrm{fin}}) \subset \pi_0({}^\phi\Vect^{(\tau, c) + (0, 0)}_G(X)_{\mathrm{fin}})$ is a submonoid. Using the idea of the proof of Lemma \ref{lem:abelian_group_structure}, we can show that the reversal of the $\Z_2$-grading of twisted bundles gives a monoid morphism satisfying the assumptions in Lemma \ref{appendix:lem_quotient_monoid}, from which ${}^\phi K^{(\tau, c)+n}_G(X)_{\mathrm{fin}}$ is an abelian group. If $c$ is trivial, then this group agrees with the Grothendieck construction of the monoid of isomorphism classes of ungraded vector bundles. 

For any $E \in {}^\phi \Vect_G^{(\tau, c) + (p, q)}(X)_{\mathrm{fin}}$, the trivial section $0 \in \Gamma(X//G, \Fred(E))$ makes sense. Moreover, any section $A \in \Gamma(X//G, \Fred(E))$ is homotopic to the trivial section. If $E_{\mathrm{uni}}$ is a locally universal twisted bundle, then we have an embedding $E \to E_{\mathrm{uni}}$. The orthogonal complement $E^\perp \subset E_{\mathrm{uni}}$ can be assumed to be locally universal. Then we have a map $\Gamma(X//G, \Fred(E)) \to \Gamma(X//G, \Fred(E_\mathrm{uni}))$ given by $A \mapsto A \oplus \gamma_*$, where $\gamma_* \in \Gamma(X//G, \Fred(E^\perp)^\dagger)$. Thus, taking $A = 0$, we have an induced homomorphism
$$
\imath: \
{}^\phi K^{(\tau, c)+n}_G(X)_{\mathrm{fin}} \longrightarrow
{}^\phi K^{(\tau, c)+n}_G(X).
$$
In the case of $X = \pt$, a standard argument shows that the operation of taking the kernel of skew-adjoint (Fredholm) operators induces a well-defined homomorphism
\begin{align*}
\varkappa &:
{}^\phi K^{(\tau, c)+n}_G(\pt) \to 
{}^\phi K^{(\tau, c)+n}_G(\pt)_{\mathrm{fin}},
&A \mapsto \mathrm{Ker}(A),
\end{align*}
which is inverse to $\imath$. The theorem of Atiyah-J\"{a}nich \cite{A1} is a generalization of this fact in the case that $n = 0$ and $\tau$ is trivial, and a twisted generalization due to \cite{F-M} is as follows.

\begin{prop}[\cite{F-M}] \label{prop:finite_rank_realizability_freed_moore}
Under the assumptions in Definition \ref{dfn:finite_rank_freed_moore_K}, if $c : X//G \to \pt//\Z_2$ is trivial and $n = 0$, then the homomorphism $\imath$ is an isomorphism.
\end{prop}

In \cite{F-M}, the result above is stated in Remark 7.37 and a proof is given in Appendix E, which is essentially the construction of the inverse $\varkappa$. However, this proof seems not to work in the presence of a non-trivial homomorphism $c : G \to \Z_2$, at the point that we apply the argument proving the Atiyah-J\"anich theorem in \cite{A1}. For this reason, $c$ is assumed to be trivial in the above proposition. We remark that the proposition will be reproved in a different way in \S\S\ref{subsec:finite_dimensional_formulations}. 

\medskip

In the presence of a non-trivial $c$, we can instead prove the following: 

\begin{prop} \label{prop:exact_sequence_with_c}
Under the assumptions in Definition \ref{dfn:finite_rank_freed_moore_K}, if $c : X//G \to \pt//\Z_2$ is associated to a non-trivial homomorphism $c : G \to \Z_2$, then there is an exact sequence of groups
$$
{}^\phi K^{\tau + 0}_G(X)_{\mathrm{fin}} 
\overset{\pi^*}{\longrightarrow}
{}^\phi K^{\tau + 0}_G(X \times \Z_2)_{\mathrm{fin}} 
\overset{\pi_*}{\longrightarrow}
{}^\phi K^{(\tau, c) + 0}_G(X)_{\mathrm{fin}} 
\longrightarrow
0, 
$$
where the original $G$-action on $X$ and the morphism $c : G \to \Z_2$ define the $G$-action on $X \times \Z_2$ by $(x, r) \mapsto (gx, c(g)r)$, and $\pi : X \times \Z_2 \to X$ is the projection.
\end{prop}

\begin{proof}
First of all, we construct $\pi_* : {}^\phi K^{\tau + 0}_G(X \times \Z_2)_{\mathrm{fin}} \to {}^\phi K^{(\tau, c) + 0}_G(X)_{\mathrm{fin}}$. For the construction, we notice that the group ${}^\phi K^{\tau + 0}_G(X \times \Z_2)_{\mathrm{fin}}$ agrees with the Grothendieck construction of the monoid of isomorphism classes of $(\phi, \tau)$-twisted \textit{ungraded} vector bundles on $(X \times \Z_2)//G$ (or equivalently $(\phi, \tau)$-twisted vector bundles with trivial odd parts). For such a twisted vector bundle $E = E^0 \oplus 0$ on $(X \times \Z_2)//G$, we define a $\Z_2$-graded Hermitian vector bundle $\hat{E} = \hat{E}^0 \oplus \hat{E}^1$ on $X$ by setting $\hat{E}^k = E|_{X \times \{ (-1)^k \}}$. The $(\phi, \tau)$-twisted $G$-action on $E$ induces a $(\phi, \tau, c)$-twisted $G$-action on $\hat{E}$. Then the assignment $E \mapsto \hat{E}$ extends to the homomorphism $\pi_* : {}^\phi K_G^{\tau + 0}(X \times \Z_2)_{\mathrm{fin}} \to {}^\phi K_G^{(\tau, 0) + c}(X)_{\mathrm{fin}}$. 

If $\hat{E} = \hat{E}^0 \oplus \hat{E}^1$ is a $(\phi, \tau, c)$-twisted vector bundle on $X//G$, then we can define an ungraded vector bundle $E$ on $X \times \Z_2$ by $E|_{X \times \{ (-1)^k \}} = \hat{E}^k$. The $(\phi, \tau, c)$-twisted $G$-action on $\hat{E}$ induces a $(\phi, \tau)$-twisted $G$-action on $E$. We can directly check $\pi_*E \cong \hat{E}$, so that $\pi_*$ is surjective.

If $E = \pi^*F$ for a $(\phi, \tau)$-twisted ungraded vector bundle $F$ on $X//G$, then we have $\hat{E} = \hat{E}^0 \oplus \hat{E}^1$ with $\hat{E}^0 = \hat{E}^1 = F$. This $\Z_2$-graded vector bundle $\hat{E}$ admits a $Cl_{1, 0}$-action
$$
\gamma 
=
\left(
\begin{array}{rr}
0 & -1 \\
1 & 0
\end{array}
\right).
$$
This $Cl_{1,0}$-action is compatible with the $(\phi, \tau, c)$-twisted $G$-action on $\hat{E}$, so that we have $\hat{E} \in {}^\phi\Vect^{(\tau, c) + (1, 0)}(X//G)_{\mathrm{fin}}$. This proves $\pi_*\pi^*F = 0$ in ${}^\phi K_G^{(\tau, c) + 0}(X)_{\mathrm{fin}}$. In the opposite direction, suppose first that a $(\phi, \tau, c)$-twisted vector bundle $\hat{E} = \pi_*E$ on $X//G$ is constructed from a $(\phi, \tau)$-twisted ungraded vector bundle $E$ on $(X \times \Z_2)//G$. We express the twisted $G$-action on $E$ as $\rho(g)_{ij} : L|_{\{ g \} \times X} \otimes E|_{X \times \{ (-1)^j\}} \to E|_{X \times \{ (-1)^i \}}$, where $(-1)^i = c(g)(-1)^j$ holds true, and $L \to G \times X$ is the Hermitian line bundle in the data $(L, \tau)$ of the $\phi$-twisted (trivially $\Z_2$-graded, or ungraded) extension $\tau$. With this notation, the $(\phi, \tau, c)$-twisted $G$-action on $\hat{E} = \hat{E}^0 \oplus \hat{E}^1$ is expressed as
$$
\hat{\rho}(g)
=
\left\{
\begin{array}{cl}
\left(
\begin{array}{cc}
\rho(g)_{00} & 0 \\
0 & \rho(g)_{11}
\end{array}
\right),
& (c(g) = 1) \\
\left(
\begin{array}{cc}
0 & \rho(g)_{01} \\
\rho(g)_{10} & 0
\end{array}
\right).
& (c(g) = -1)
\end{array}
\right.
$$
Suppose next that $\hat{E}$ admits a compatible $Cl_{1,0}$-action, which is expresses as
$$
\gamma
=
\left(
\begin{array}{cc}
0 & \gamma_{01} \\
\gamma_{10} & 0
\end{array}
\right),
$$
where $\gamma_{01}\gamma_{10} = -1$. Then we have a $(\phi, \tau)$-twisted vector bundle $F$ on $X//G$ by setting $F = \hat{E}^0 = E|_{X \times \{ 1 \}}$ and defining its $(\phi, \tau)$-twisted $G$-action as
$$
\rho_F(g)
=
\left\{
\begin{array}{ll}
\rho(g)_{00}, & (c(g) = 1) \\
\gamma_{01}\rho(g)_{10}. & (c(g) = -1)
\end{array}
\right.
$$
We can verify that $E$ is isomorphic to $\pi^*F$. This means that the sequence in question is exact at ${}^\phi K_G^{\tau + 0}(X \times \Z_2)_{\mathrm{fin}}$, and the proof is completed.
\end{proof}

Let us apply the result above to the case where $G = \Z_2$ acts on $X$ trivially, $c : G \to \Z_2$ is the identity, and $\tau$ is trivial. Then $\pi^*$ turns out to be surjective and $K^{c+0}_{\Z_2}(X)_{\mathrm{fin}} = 0$, whereas $K^{c+0}_{\Z_2}(X) \cong K^0_{\pm}(X) \cong K^1(X)$ as shown in \cite{A-H}. Thus, if $X = S^1$, then $\imath : K^{c+0}_{\Z_2}(S^1)_{\mathrm{fin}} \to K^{c+0}_{\Z_2}(S^1) \cong \Z$ is not bijective (cf.\ \cite{Th2}).

\medskip

\begin{rem}
Under the assumptions in Definition \ref{dfn:finite_rank_freed_moore_K} that $G$ is finite and $X$ is compact, the isomorphism class of the ungraded twist $[\tau] \in H^3_G(X; \Z)$ is a torsion class. Hence Proposition \ref{prop:finite_rank_realizability_freed_moore} is consistent with the conjecture in \cite{TXL}.
\end{rem}

\subsection{The Thom isomorphism}
\label{subsec:thom_isomorphism}

To state the Thom isomorphism in the Freed-Moore $K$-theory, let us recall, from \cite{L-M} for instance, that the $\Pin^c$-group $\mathrm{Pin}^c(r)$ is a central extension of the orthogonal group $O(r)$ by $U(1)$. This group sits in the complexified Clifford algebra $Cl_{r, 0} \otimes \C$, so that there is a natural complex conjugation on $\Pin^c(r)$. Using this, we define ${}^\phi g$ for $\phi = \pm 1$ and $g \in Cl_{r, 0} \otimes \C$ by
$$
{}^\phi g = 
\left\{
\begin{array}{ll}
g, & (\phi = 1) \\
\bar{g}. & (\phi = -1)
\end{array}
\right.
$$

\begin{dfn} \label{dfn:twisted_pinc}
Let $\X$ be a groupoid, $\phi \in \Phi(\X)$ an object, and $\pi : V \to \X$ a real vector bundle of rank $r$. We write $P = (P, \rho)$ for the principal $O(r)$-bundle arising as the frame bundle of $V$ with respect to a Riemannian metric. For $\phi$ realized as a map of groupoids $\phi : \X \to \pt//\Z_2$, a \textit{$\phi$-twisted $\Pin^c$-structure} on $V$ consists of 
\begin{itemize}
\item
A principal $\Pin^c(r)$-bundle $\tilde{\pi} : \tilde{P} \to \X_0$ which is a lift of the structure group $O(r)$ of $\pi : P \to \X_0$ to $\Pin^c(r)$ by an equivariant map $q : \tilde{P} \to P$. 

\item
A fiber preserving map $\tilde{\rho} : \partial_0^*\tilde{P} \to \partial_1^*\tilde{P}$ on $\X_1$ such that
\begin{itemize}
\item
$\partial_0^*q \circ \tilde{\rho} = \tilde{\rho} \circ \partial_1^*q$,

\item
$\tilde{\rho}(\tilde{p} \cdot \tilde{g}) = \tilde{\rho}(\tilde{p}) \cdot {}^{\phi(f)}\tilde{g}$ for $f \in \X_1$, $\tilde{p} \in \partial_0^*\tilde{P}|_f$ and $\tilde{g} \in \Pin^c(r)$,

\item
$\partial_2^*\tilde{\rho} \circ \partial_0^*\tilde{\rho} = \partial_1^*\tilde{\rho}$ on $\X_2$.
\end{itemize}
\end{itemize}
For general $\phi$ consisting of a local equivalence $\tilde{\X} \to \X$ and a map of groupoids $\phi : \tilde{\X} \to \pt//\Z_2$, a $\phi$-twisted $\Pin^c$-structure on $V \to \X$ means one on the pull-back of $V$ to $\tilde{\X}$. 
\end{dfn}

To help the understanding of the notion of $\phi$-twisted $\Pin^c$-structures, let us assume for a moment that the groupoid $\X$ is the quotient groupoid  $X//G$ associated to an action of a compact Lie group $G$ on $X$ and $\phi \in \Phi(X//G)$ is induced from a homomorphism $\phi : G \to \Z_2$. In this case, a real vector bundle $V$ on $X//G$ means a $G$-equivariant real vector bundle on $X$, so that its frame bundle $P$ is a $G$-equivariant principal $O(r)$-bundle, provided that the rank of $V$ is $r$. Then, a $\phi$-twisted $\Pin^c$-structure $\tilde{P}$ of $V$ is a $\Pin^c$-structure of the underlying vector bundle $V$ which has, for each $f \in G$, a $G$-action $\tilde{\rho}_f : \tilde{P} \to \tilde{P}$ covering the $G$-action $\rho_f : P \to P$ such that $\tilde{\rho}_f(\tilde{p}\tilde{g}) = \tilde{\rho}_f(\tilde{p}) \cdot {}^{\phi(f)}\tilde{g}$ for all $\tilde{p} \in \tilde{P}$ and $\tilde{g} \in \Pin^c(r)$.

\begin{lem} \label{lem:obstruction_to_twisted_Pin_c}
Let $\X$ be a groupoid,  $\phi \in \Phi(\X)$ an object, and $V$ a real vector bundle on $\X$ of rank $r$. There exists a $\phi$-twisted $\Pin^c$-structure on $V$ if and only if a cohomology class ${}^\phi W_3(V) \in H^3(\X; \Z_\phi)$ vanishes.
\end{lem}

\begin{proof}
We can assume that $\phi$ is realized as a map of groupoids $\phi : \X \to \pt//\Z_2$. Then, from the frame bundle $P$ of $V$, we can construct a groupoid $\mathcal{P}//O(r)$ admitting a local equivalence $\varpi : \mathcal{P}//O(r) \to \X$. Concretely, $(\mathcal{P}//O(r))_0 = P$ and $(\mathcal{P}//O(r))_1 = O(r) \times \partial_0^*P$. Further, from the central extension $\Pin^c(r)$ of $O(r)$, we can construct a $\varpi^*\phi$-twisted (ungraded) extension $(L_V, \tau_V)$ of $\mathcal{P}//O(r)$ whose trivializations are in bijective correspondence with $\phi$-twisted $\Pin^c$-structures on $V$. Concretely, $L_V \to (\mathcal{P}//O(r))_1$ is the pull-back under the projection $(\mathcal{P}//O(r))_1 \to O(r)$ of the Hermitian line bundle $L \to O(r)$ associated to $\Pin^c(r)$, and $\tau_V$ is induced from the group structure on $\Pin^c(r)$. In general, the class in $H^3(\X; \Z_\phi)$ that classifies a $\phi$-twisted extension is the obstruction to admitting a trivialization, from which the lemma follows.
\end{proof}

\begin{thm}
Let $\X$ be a local quotient groupoid, $\phi \in \Phi(\X)$ an object, and $(\tau, c)$ a $\phi$-twist of $\X$. For any real vector bundle $\pi : V \to \X$ of rank $r$, we write $D(V)$ and $S(V)$ for the unit disk bundle and the unit sphere bundle of $V$ with respect to a Riemannian metric. Then there is a natural isomorphism
$$
{}^\phi K^{(\tau, c) + n}(\X) \cong
{}^{\pi^*\phi} K^{\pi^*((\tau, c) + (\tau_V, c_V)) + n + r}(D(V), S(V)),
$$
where $\tau_V$ is classified by the obstruction class ${}^\phi W_3(V) \in H^3(\X; \Z_\phi)$ for $V$ admitting a $\phi$-twisted $\Pin^c$-structure, and $c_V$ by the obstruction class $w_1(V) \in H^1(\X; \Z_2)$ for $V$ to being orientable. 
\end{thm}

\begin{proof}
We sketch the proof following \cite{FHT1}.  By replacing $\X$ if necessary, we can assume that the object $\phi$ is realized as a map of groupoids $\phi : \X \to \pt//\Z_2$ and the $\phi$-twist $(\tau, c)$ as a $\phi$-twisted extension of $\X$. Let $E$ be a locally universal $(\phi, \tau, c)$-twisted vector bundle on $\X$. The disk bundle $D(V)$ gives rise to a local quotient groupoid, and $S(V)$ is its subgroupoid. By definition, we have
$$
{}^{\pi^*\phi} K^{\pi^*(\tau, c) + 0}(D(V), S(V))
= \Gamma(D(V), S(V), \Fred(\pi^*E))/\sim.
$$
Associated to $V$ is a $\phi$-twisted vector bundle $\C l(V) \to \X$ whose fibers are the complexified Clifford algebra. Then, we have a homotopy equivalence
$$
\Gamma(D(V), S(V), \Fred(\pi^*E))
\simeq
\Gamma(\X, \Fred_{C l(V)}(\C l(V) \otimes E)),
$$
where $\Fred_{Cl(V)}(\C l(V) \otimes E)$ is defined by replacing the Clifford action in Definition \ref{dfn:fredholm_family} by the natural left fiberwise Clifford action of $Cl(V)$ on $\C l(V) \otimes E$. Except for the use of the $O(r)$-equivariance of the Atiyah-Singer map in Lemma \ref{lem:atiyah_singer_map_basic_case}, the proof of the homotopy equivalence is essentially the repetition of that of Theorem \ref{thm:Atiyah_Singer_map}. From the frame bundle $P$ of $V$, we can construct a groupoid $\mathcal{P}//O(r)$ and a local equivalence $\varpi : \mathcal{P}//O(r) \to \X$, as in the proof of Lemma \ref{lem:obstruction_to_twisted_Pin_c}. The pull-back under $\varpi$ induces a homeomorphism
$$
\Gamma(\X, \Fred_{C l(V)}(\C l(V) \otimes E))
\cong
\Gamma(\mathcal{P}//O(r), \Fred_{\varpi^*Cl(V)}(\varpi^*(\C l(V) \otimes E))).
$$
On $\mathcal{P}//O(r)$ is the $\varpi^*\phi$-twisted ungraded extension $L_V$ whose trivializations are bijective correspondence with the $\phi$-twisted $\Pin^c$-structures on $V$. The determinant $\det : O(r) \to \Z_2$ induces a grading $c_V$ of $L_V$ which is classified by $w_1(V) \in H^1(\X; \Z_2)$. Note that the pull-back $\varpi^*V \to \mathcal{P}//O(r)$ is isomorphic to the product bundle $\underline{\R}^r$ of rank $r$, so that $\varpi^*\C l(V)$ is isomorphic to the product bundle $\underline{M}$ with fiber $M = \C l(r) = Cl_{r, 0} \otimes \C$, as a vector bundle. The product bundle $\underline{M}$ gives rise to a $(\varpi^*\phi, \tau_V, c_V)$-twisted vector bundle whose twisted action is induced from the left action of $\Pin^c(r) \subset \C l(r)$ on $\C l(r)$. Furthermore, this twisted bundle is a $\varpi^*Cl(V)-Cl_{r, 0}$-bimodule, where $\varpi^*Cl(V)$ acts from the left through the trivialization $\varpi^*Cl(V) \cong \underline{M}$ and $Cl_{r, 0}$ from the right through $Cl_{r, 0} \subset \C l(r)$. Then, by a Morita equivalence based on $\underline{M}$, we have a homeomorphism
$$
\Gamma(\mathcal{P}//O(r), \Fred_{\varpi^*Cl(V)}(\varpi^*(\C l(V) \otimes E)))
\cong
\Gamma(\mathcal{P}//O(r), \Fred(E')),
$$
where $E'$ is the locally universal $\varpi^*((\tau, c)-(\tau_V, c_V))$-twisted vector bundle with $Cl_{r, 0}$-action given by
$$
E' = \mathrm{Hom}_{\varpi^*Cl(V)}(\underline{M}, \varpi^*(\C l(V) \otimes E)).
$$
Summarizing, we get a natural isomorphism
$$
{}^{\pi^*\phi} K^{\pi^*(\tau, c) +0}(D(V), S(V))
\cong {}^\phi K^{((\tau, c) - (\tau_V, c_V)) - r}(\X).
$$
Since an action of the Clifford algebra accounts for the degree $n \in \Z$, we can generalize the argument so far to have
$$
{}^{\pi^*\phi} K^{\pi^*(\tau, c) -n}(D(V), S(V))
\cong {}^\phi K^{((\tau, c) - (\tau_V, c_V)) - n - r}(\X),
$$
which is equivalent to the isomorphism in question.
\end{proof}

To provide examples of the Thom isomorphism, let us consider the quotient groupoid $\X = X//G$ associated to an action of a finite group $G$ on a space $X$ and $\phi \in \Phi(X//G)$ associated to a homomorphism $\phi : G \to \Z_2$. We let $(\tau, c)$ be any $\phi$-twist.

Let $E$ be a $\phi$-twisted vector bundle on $X$ of rank $r$. For its underlying real vector bundle of rank $2r$, we can show by a direct computation that the ungraded twist $\tau_E$ is given by the group cocycle $\tau_\phi$ if $r = 1, 2 \mod 4$, and is trivial if $r = 3, 4 \mod 4$. We can also show that the grading $c_E$ is given by the homomorphism $c_\phi = \phi$ if $r = 1, 3 \mod 4$ and $c_\phi$ is trivial if $r = 2, 4 \mod 4$. Because $\tau_\phi$ and $c_\phi$ have the effect of degree shifts, we eventually get
$$
{}^\phi K^{(\tau, c) + n}_G(X)
\cong {}^\phi K^{\pi^*(\tau, c) + n}_G(D(E), S(E)).
$$
This generalizes the Thom isomorphism for a complex vector bundle in complex $K$-theory and that for a `Real' vector bundle in `Real' $K$-theory.

Let $f : G \to \Z_2$ be any homomorphism, and $\underline{\R}_f$ the product real line bundle $X \times \R$ with the $G$-action $(x, r) \mapsto (gx, f(g)r)$. We have $\tau_{\underline{\R}_f} = \tau_f$ and $c_{\underline{\R}_f} = c_f = f$. In the case of $f = \phi$, the Thom isomorphism is
$$
{}^\phi K^{(\tau, c) + n}_G(X) \cong 
{}^\phi K^{\pi^*(\tau, c) + n - 1}_G
(D(\underline{\R}_\phi), S(\underline{\R}_\phi)).
$$
If $\phi$ is non-trivial, then $\Z_2 \cong G/\mathrm{Ker}\phi$ and the inclusion $\mathrm{Ker} \phi \to G$ induces a local equivalence of groupoids 
$$
X//\mathrm{Ker}\phi \to (X \times G/\mathrm{Ker}\phi)//G
\cong (X \times \Z_2)//G 
= S(\underline{\R}_\phi)//G.
$$
Thus, from the long exact sequence for $(D(\underline{\R}_\phi), S(\underline{\R}_\phi))$, we get a generalization of an exact sequence for `Real' $K$-theory in \cite{A2} (p.377, (3.4))
$$
\cdots \to
{}^\phi K^{(\tau, c) + n + 1}_G(X) \to
{}^\phi K^{(\tau, c) + n}_G(X) \to
K^{(\tau, c) + n}_{\mathrm{Ker}\phi}(X) \to
{}^\phi K^{(\tau, c) + n + 2}_G(X) \to
\cdots.
$$
In the case of $f = c$, the Thom isomorphism is
$$
{}^\phi K^{(\tau, c) + n}_G(X)
\cong 
{}^\phi K^{\pi^*(\tau, 0) + n + 1}_G
(D(\underline{\R}_c), S(\underline{\R}_c)).
$$
The long exact sequence for the  pair $(D(\underline{\R}_c), S(\underline{\R}_c))$ gives us
$$
\cdots \to
{}^\phi K^{\tau + n}_G(X) \to
{}^\phi K^{\tau + n}_G(X \times \Z_2) \to
{}^\phi K^{(\tau, c) + n}_G(X) \to
{}^\phi K^{\tau + n + 1}_G(X) \to
\cdots,
$$
where $G$ acts on $X \times \Z_2$ by $(x, r) \mapsto (gx, c(g)r)$. We remark ${}^\phi K^{\tau + n}_G(X \times \Z_2)  \cong {}^\phi K^{\tau + n}_{\mathrm{Ker}(c)}(X)$ if $c$ is non-trivial. We also remark that the above exact sequence extends the one in Proposition \ref{prop:exact_sequence_with_c}.


\section{Karoubi formulation of Freed-Moore $K$-theory}
\label{sec:Karoubi_formulation}

This section is devoted to Karoubi's formulations of the Freed-Moore $K$-theory: We introduce the infinite-dimensional Karoubi formulation, and relate it with the Fredholm formulation. We then introduce the finite-dimensional Karoubi formulation, and relate it with the other formulations.


\subsection{Gradation}

\begin{dfn}[gradation] \label{dfn:gradation}
Let $\X$ be a groupoid, $\phi : \X \to \Z_2$ a map of groupoids, and $(L, \tau, c)$ a $\phi$-twisted $\Z_2$-graded extension of $\X$. For a $(\phi, \tau, c)$-twisted vector bundle $(E, \epsilon, \rho, \gamma)$ on $\X$ with $Cl_{p, q}$-action, we define a fiber bundle $\Gr(E) \to \X$ as follows:
\begin{itemize}
\item
The fiber of the underlying fiber bundle $\Gr(E) \to \X_0$ at $x \in \X_0$ consists of bounded operators $\eta : E_x \to E_x$ such that:
\begin{itemize}
\item[(i)]
$\eta$ are self-adjoint involutions: $\eta^* = \eta$ and $\eta^2 = \id$.

\item[(ii)]
$\eta - \epsilon$ are compact.

\item[(iii)]
$\eta$ anti-commute with the $Cl_{p, q}$-action, that is, 
\begin{align*}
\eta \gamma(e) &= - \gamma(e)\eta,
\end{align*}
for any unit norm element $e \in \R^{p+q}$.
\end{itemize}

\item
The bundle isomorphism $\varrho : {}^\phi \partial_0^*\Gr(E) \to \partial_1^*\Gr(E)$ on $\X_1$ is given by $\varrho(\eta) = \rho \circ (\id_L \otimes \eta) \circ \rho^{-1}$, where $\id_L : L \to L$ is the identity map.
\end{itemize}
The fiber of $\Gr(E)$ is topologized by using the operator norm topology, and its structure group by the compact open topology in \cite{A-Se}. The space of sections is defined by
$$
\Gamma(\X, \Gr(E)) =
\{ \eta \in \Gamma(\X_0, \Gr(E)) |\ 
\varrho \circ {}^\phi \partial_0^*\eta = \partial_1^*\eta \}.
$$
We call a section $\eta \in \Gamma(\X, \Gr(E))$ a \textit{gradation} of $E$.
\end{dfn}

As in the case of $\mathrm{K}(E)$, the $\phi$-twisted action $\varrho$ is continuous with respect to the topology on $E$ given by the compact open topology.

We notice that no commutation relation among $\eta$ and $\epsilon$ is imposed. We also notice that $(E, \eta, \rho, \gamma)$ is a $(\phi, \tau, c)$-twisted vector bundle on $\X$ with $Cl_{p, q}$-action. Thus a gradation $\eta$ on $E$ is regarded as another choice of a $\Z_2$-grading of $E$. It is clear that $\epsilon \in \Gamma(\X, \Gr(E))$. A \textit{homotopy} between gradations $\eta_0$ and $\eta_1$ of $E$ is defined by a gradation $\tilde{\eta}$ of the twisted bundle $E \times [0, 1]$ on $\X \times [0, 1]$ such that $\tilde{\eta}|_{\X \times \{ i \}} = \eta_i$ for $i = 0, 1$. In this case, $\eta_0$ and $\eta_1$ are said to be homotopic, and we write $\eta_0 \sim \eta_1$.

\medskip

Suppose that for a groupoid $\X$ and $\phi = (F : \tilde{\X} \to \X, \phi) \in \Phi(\X)$, we have a $\phi$-twist on $\X$ consisting of a local equivalence $\tilde{F} : \tilde{\tilde{\X}} \to \tilde{\X}$, a $\tilde{F}^*\phi$-twisted extension $(L, \tau, c)$ of $\tilde{\tilde{\X}}$, and a $(\tilde{F}^*\phi, \tau, c)$-twisted vector bundle $E$ on $\tilde{\tilde{\X}}$ with $Cl_{p, q}$-action. As in the case of $\Fred(E)$, there uniquely exists a $\phi$-twisted bundle $\Gr(\tau)$ on $\tilde{\X}$ up to isomorphisms such that $\tilde{F}^*\Gr(\tau)$ is isomorphic to $\Gr(E)$ and $\tilde{F}^*$ induces a homeomorphism from $\Gamma(\tilde{\X}, \Gr(\tau))$ to $\Gamma(\tilde{\tilde{\X}}, \Gr(E))$ preserving the equivalence relations $\sim$ given by fiberwise homotopies of sections.

\begin{dfn} \label{dfn:bigraded_Karoubi_K}
Let $\X$ be a local quotient groupoid, $\phi = (F : \tilde{\X} \to \X, \phi) \in \Phi(\X)$ an object, and $(\tau, c) = (\tilde{F} : \tilde{\tilde{\X}} \to \tilde{\X}, L, \tau, c)$ a $\phi$-twist on $\X$. We define 
$$
{}^\phi \K^{(\tau, c) + (p, q)}(\X)
= \Gamma(\tilde{\X}, \Gr(\tau))/\sim
\cong \Gamma(\tilde{\tilde{\X}}, \Gr(E))/\sim
$$
where $\Gr(\tau) \to \tilde{\X}$ is the $\phi$-twisted bundle such that $\tilde{F}^*\Gr(\tau) \cong \Gr(E)$ for a $(\tilde{F}^*\phi, \tau, c)$-twisted locally universal vector bundle $E \to \tilde{\tilde{\X}}$ with $Cl_{p, q}$-action. In the case that $\X$ is a quotient groupoid $X//G$, we may write
$$
{}^\phi \K^{(\tau, c) + (p, q)}(X//G)
= {}^\phi \K^{(\tau, c) + (p, q)}_G(X).
$$
\end{dfn}

The operation of taking a direct sum makes ${}^\phi \K^{(\tau, c) + (p, q)}(\X)$ into an abelian monoid, in which the zero element is represented by $\epsilon$. This monoidal structure turns out to be a group structure, as will be seen shortly. Though will not be detailed, the same construction as in Lemma \ref{lem:weak_periodicity} provides isomorphisms, such as
$$
{}^\phi \K^{(\tau, c) + (p, q)}(\X)
\cong {}^\phi \K^{(\tau, c) + (p+1, q+1)}(\X).	
$$


\begin{rem}
Let us consider the trivial setting that $\X = \pt//1$, $\phi$, $\tau$ and $c$ are trivial, and $p = q = 0$. In this setting, a locally universal bundle is just a $\Z_2$-graded vector space $(E, \epsilon)$ such that $E^k = \mathrm{Ker}(\epsilon - (-1)^k)$ are separable infinite-dimensional Hilbert spaces. Then $\Gr(E)$ is
$$
\Gr(E)
= \{
\eta \in \mathrm{End}(E) |\ 
\eta - \epsilon \in \mathrm{K}(E), \
\eta^* = \eta, \ \eta^2 = \id
\},
$$
where $\mathrm{End}(E)$ is the space of bounded operators and $\mathrm{K}(E)$ that of compact operators. All these spaces are topologized by the operator norm. The space $\Gr(E)$ appears in \cite{Qui} as a model of the classifying space of the even complex $K$-theory, and admits the identification
$$
\Gr(E) \cong
\{ U \in \mathrm{End}(E) |\ U - \id \in \mathrm{K}(E), \
\epsilon U \epsilon = U^* = U^{-1} \}
$$
given by $\eta \mapsto U = \eta \epsilon$. The space $\Gr(E)$ also admits the identification 
$$
\Gr(E) \cong
\{ W \subset E |\ 
\mbox{closed, $\mathrm{pr}_0 : W \to E^0$ Fredholm, \
$\mathrm{pr}_1 : W \to E^1$ compact}
\}
$$
given by $\eta \mapsto W = \mathrm{Im}(1 + \eta) = \mathrm{Ker}(1 - \eta)$, where $\mathrm{pr}_k : W \to E^k$ are the orthogonal projections. This space essentially agrees with the infinite-dimensional Grassmannian in \cite{P-S}, where Hilbert-Schmidt operators are used in place of compact operators.
\end{rem}


\subsection{Relationship with Fredholm formulation}
\label{subsec:Fredholm_vs_Karoubi}

Now, we relate the $K$-theory ${}^\phi K^{(\tau, c) + (p, q)}(\X)$ in the Fredholm formulation with ${}^\phi \K^{(\tau, c) + (p, q)}(\X)$ in the Karoubi formulation. As is mentioned in \S\ref{sec:introduction}, a change of twists enters into the relationship of these two formulations.

\begin{dfn} \label{dfn:cocycle_fredholm_vs_karoubi}
Let $p_1 : \Z_2 \times \Z_2 \to \Z_2$ be the projection onto the first factor. We define a group $2$-cocycle $\mu \in Z^2_{\mathrm{group}}(\Z_2 \times \Z_2; U(1)_{p_1})$ by
$$
\mu(((-1)^m, (-1)^{n}), ((-1)^{m'}, (-1)^{n'}))
= \exp \pi i nn'.
$$ 
\end{dfn}

The $2$-cocycle $\mu$ defines a $p_1$-twisted $\Z_2$-graded extension $(\Z_2 \times \Z_2 \times \C, \mu, 1)$ of the quotient groupoid $\pt//(\Z_2 \times \Z_2)$, in which the line bundle on the space of morphisms is the product bundle $\Z_2 \times \Z_2 \times \C \to \Z_2 \times \Z_2$ with the trivial $\Z_2$-grading. Thus, if $\phi : \X \to \pt//\Z_2$ and $c : \X \to \pt//\Z_2$ are maps of groupoids, then we get a $\phi$-twisted $\Z_2$-graded extension $(\phi, c)^*\mu = (\X_1 \times \C, (\phi, c)^*\mu, 1)$ by the pull-back under $(\phi, c) : \X \to \pt//(\Z_2 \times \Z_2)$.

\begin{lem} \label{lem:triviality_of_twist_change}
Let $\phi : \X \to \pt//\Z_2$ and $c : \X \to \pt//\Z_2$ be maps of groupoids. If $\phi$ or $c$ is trivial, then $(\phi, c)^*\mu$ is trivial.
\end{lem}

\begin{proof}
If $c$ is trivial, then $(\phi, c)^*\mu$ is clearly trivial by the definition of $\mu$. If $\phi$ is trivial, then the map of groupoids $(\phi, c) : \X \to \pt//(\Z_2 \times \Z_2)$ factors as follows
$$
\X \overset{c}{\longrightarrow} 
\pt//\Z_2 \overset{(1, \id)}{\longrightarrow}
\pt//(\Z_2 \times \Z_2). 
$$
The pull-back cocycle $(1, \id)^*\mu \in Z^2_{\mathrm{group}}(\Z_2; U(1))$ can be trivialized by the group $1$-cochain $\beta \in C^1_{\mathrm{group}}(\Z_2; U(1))$ given by $\beta(1) = 1$ and $\beta(-1) = i$. Hence $(\phi, c)^*\mu$ is also trivialized.
\end{proof}

\begin{dfn}
Let $\X$ be a groupoid and $\phi : \X \to \pt//\Z_2$ a map of groupoids. For a $\phi$-twisted $\Z_2$-graded extension $(\tau, c) = (L, \tau, c)$ of $\X$, we define a $\phi$-twisted $\Z_2$-graded extension $(\acute{\tau}, c)$ by $(\acute{\tau}, c) = (L, \tau, c) + (\phi, c)^*\mu = (L, \tau \mu, c)$.
\end{dfn}

Since the $2$-cocycle $\mu^2$ is clearly trivial, we have $\acute{\acute{\tau}} = \tau$.

As is seen, for a quotient groupoid $\X = X//G$ and a map of groupoids $\phi : X//G \to \pt//\Z_2$ induced from a homomorphism $\phi : G \to \Z_2$, a group $2$-cocycle $\tau \in Z^2_{\mathrm{group}}(G; C(X, U(1))_\phi)$ defines a $\phi$-twisted $\Z_2$-graded extension $(\tau, c) = (G \times X \times \C, \tau, c)$ of $X//G$, in which the line bundle on the space of morphisms is the product bundle $G \times X \times \C \to G \times X$, and $c : \X \to \pt//\Z_2$ is a map of groupoids. In this case, $(\acute{\tau}, c) = (G \times X \times \C, \acute{\tau}, c)$ is the $\phi$-twisted $\Z_2$-graded extension of $X//G$ associated to the group $2$-cocycle $\acute{\tau} \in Z^2_{\mathrm{group}}(G; C(X, U(1))_\phi)$ given by
$$
\acute{\tau}(g, h; x)
=
\left\{
\begin{array}{ll}
\tau(g, h; x), & (\mbox{$c(g, hx) = 1$ or $c(h, x) = 1$}) \\
-\tau(g, h; x). & (\mbox{$c(g, hx) = c(h, x) = -1$})
\end{array}
\right.
$$

\begin{lem} \label{lem:categorical_correspondence_by_epsilon}
Let $\X$ be a groupoid, $\phi : \X \to \pt//\Z_2$ a map of groupoids, and $(\tau, c)$ a $\phi$-twisted $\Z_2$-graded extension of $\X$. 
\begin{itemize}
\item[(a)]
For a $(\phi, \tau, c)$-twisted vector bundle $E$ on $\X$ with $Cl_{p, q}$-action, there exists a $(\phi, \acute{\tau}, c)$-twisted vector bundle $\acute{E}$ on $\X$ with $Cl_{q, p}$-action such that $\acute{\acute{E}}$ is isomorphic to $E$.

\item[(b)]
For a degree $k$ map $f : E_1 \to E_2$ of $(\phi, \tau, c)$-twisted vector bundles on $\X$ with $Cl_{p, q}$-actions, there is a degree $k$ map $\acute{f} : \acute{E}_1 \to \acute{E}_2$ of $(\phi, \acute{\tau}, c)$-twisted vector bundles on $\X$ with $Cl_{q, p}$-action such that $\acute{f}_1 \circ \acute{f}_2 = (-1)^{k_1k_2} \acute{(f_1 \circ f_2)}$ for maps $f_i$ of degree $k_i$.
\end{itemize}
\end{lem}

\begin{proof}
For (a), let $E = (E, \epsilon, \rho, \gamma)$ be a $(\phi, \tau, c)$-twisted vector bundles $E$ on $\X$ with $Cl_{p, q}$-action. We then have a $(\phi, \acute{\tau}, c)$-twisted vector bundle $\acute{E} = (\acute{E}, \acute{\epsilon}, \acute{\rho}, \acute{\gamma})$ on $\X$ with $Cl_{q, p}$-actions as follows: The underlying Hermitian vector bundle on $\X_0$ is given by $\acute{E} = E$, and its $\Z_2$-grading by $\acute{\epsilon} = \epsilon$. The $Cl_{q, p}$-action $\acute{\gamma}$ on $\acute{E} = E$ is given by $\acute{\gamma}(e) = \gamma(e)\epsilon$ for unit norm vectors $e \in \R^{p+q}$. Finally, the twisted action $\acute{\rho} : L \otimes {}^\phi\partial_0^*E \to \partial_1^*E$ is given by the composition
$$
L \otimes {}^\phi \partial_0^* E 
\overset{\id \otimes \epsilon_c}{\longrightarrow}
L \otimes {}^\phi \partial_0^* E 
\overset{\rho}{\longrightarrow}
\partial_1^*E,
$$
where $\epsilon_c : \partial_0^*E \to \partial_0^*E$ is the identity on the component $c^{-1}(1) \subset \X_1$ and ${}^\phi\partial_0^*\epsilon$ on $c^{-1}(-1) \subset \X_1$. With some direct calculations, we can verify that $\acute{E}$ is a $(\phi, \acute{\tau}, c)$-twisted vector bundle on $\X$ with $Cl_{q, p}$-action. By construction, the identity map gives an isomorphism $\acute{\acute{E}} \to E$. For (b), we define $\acute{f} : \acute{E}_1 \to \acute{E}_2$ by $\acute{f} = f$ if $k$ is even and $\acute{f} = f\epsilon$ if $k$ is odd. We can readily verify that $\acute{f}$ is a degree $k$ map of $(\phi, \acute{\tau}, c)$-twisted vector bundles on $\X$ with $Cl_{q, p}$-action. The claim about the composition is also verified readily.
\end{proof}

For a better understanding of the construction of $\acute{E}$, let us assume that the groupoid $\X$ is a quotient groupoid $X//G$, the map of groupoids $\phi$ is induced from a homomorphism $\phi : G \to \Z_2$, and $c : X//G \to \pt//\Z_2$ is also induced from a homomorphism $c : G \to \Z_2$. As is reviewed already in \S\S\ref{subsec:twisted_vector_bundle}, a $(\phi, \tau, c)$-twisted vector bundle $E$ is a $\Z_2$-graded vector bundle $(E, \epsilon)$ with real orthogonal map $\rho(g) : E \to E$, ($g \in G$) realizing the twisted action and with unitary maps $\gamma_i : E \to E$ realizing the $Cl_{p, q}$-action. Then the twisted action $\acute{\rho}(g)$ and the $Cl_{q, p}$-action $\acute{\gamma}$ on $\acute{E} = E$ are given by
\begin{align*}
\acute{\rho}(g)
&=
\left\{
\begin{array}{ll}
\rho(g), & (c(g) = 1) \\
\rho(g)\epsilon, & (c(g) = -1)
\end{array}
\right.
&
\acute{\gamma}_i
&=
\gamma_i\epsilon.
\end{align*}

\medskip

We now introduce a map relating the Fredholm formulation with the Karoubi formulation, which originates from \cite{A-Si}.

\begin{dfn} \label{dfn:Aityha_Singer_map}
Let $\X$ be a groupoid, $\phi : \X \to \pt//\Z_2$ a map of groupoids, $(\tau, c)$ a $\phi$-twisted $\Z_2$-graded extension of $\X$, and $(E, \epsilon, \rho, \gamma)$ a $(\phi, \tau, c)$-twisted vector bundle on $\X$ with $Cl_{p, q}$-action. We define a map of fiber bundles on $\X_0$ 
$$
\vartheta : \Fred(E) \to \Gr(\acute{E})
$$
by $\vartheta(A) = -e^{\pi A}\epsilon$ for $A : E|_x \to E|_x$ belonging to the fiber of $\Fred(E)$ at $x \in \X_0$.
\end{dfn}

To see that $\vartheta$ is a well-defined continuous map, we start with the simplest case.

\begin{lem} \label{lem:key_continuity}
For a separable $\Z_2$-graded Hilbert space $(E, \epsilon)$, the map $\vartheta : \Fred(E) \to \Gr(E)$ is well-defined and continuous.
\end{lem}

\begin{proof}
To see that $\vartheta$ is well-defined, we notice that $-iA : E \to E$ is a self-adjoint operator. Hence the functional calculus gives a bounded operator $e^{i(-i \pi A)} = e^{\pi A}$ on $E$. Because $A^2 + \id$ is compact, the spectral set $\mathrm{Spec}(A^2)$ and hence $\mathrm{Spec}(A)$ consist of eigenvalues only. The eigenspaces of $A$ whose eigenvalues differ from $\pm i$ are finite-dimensional. The eigenspaces of $A$ with their eigenvalues $\pm i$ are the only possible infinite-dimensional eigenspaces, on which $-e^{\pi A}$ acts by $\id$. This proves that $-e^{\pi A} - \id$ is compact, and so is $-e^{\pi A}\epsilon - \epsilon$. We can directly check that $-e^{\pi A}\epsilon$ is a self-adjoint involution. Hence $\vartheta$ is well-defined as a map. To prove that $\vartheta : \Fred(E) \to \Gr(E)$ is continuous, let $A_m \in \Fred(E)$ be a sequence convergent to $A_\infty \in \Fred(E)$ as $m \to \infty$. By definition, this means that the sequence of maps $A_m|_C : C \to E$ uniformly converges to $A_\infty|_C : C \to E$ on any compact subset $C \subset E$ and $A_m^2$ converges to $A_\infty^2$ in the operator norm, i.e.\ $\lVert A_m^2 - A_\infty^2 \rVert \to 0$ as $m \to \infty$. The continuity of $\vartheta$ will be established when we prove that $-\vartheta(A_m)\epsilon$ converges to $-\vartheta(A_\infty)\epsilon$ in the operator norm. For this aim, we express $-\vartheta(A)\epsilon$ as follows
\begin{align*}
-\vartheta(A)\epsilon = e^{i(-i \pi A)}
&= \cos(-i \pi A) + i \sin(-i \pi A) \\
&= \sum_{n \ge 0}\frac{(-1)^n}{(2n)!}(-i \pi A)^{2n}
+ i(-i \pi A) \sum_{n \ge 0}\frac{(-1)^n}{(2n+1)!}(-i \pi A)^{2n} \\
&= \sum_{n \ge 0} \frac{\pi^{2n}}{(2n)!}(A^2)^n
+ i(-iA) \pi \sum_{n \ge 0} \frac{\pi^{2n}}{(2n+1)!}(A^2)^n.
\end{align*}
Let $C(z)$ and $S(z)$ be the following power series in $z$
\begin{align*}
C(z) &= \sum_{n \ge 0} \frac{\pi^{2n}}{(2n)!}z^n, &
S(z) &= \pi \sum_{n \ge 0} \frac{\pi^{2n}}{(2n+1)!}z^n.
\end{align*}
These series are convergent on the whole of $\C$, and hence define holomorphic functions. As a result, $\cos(- i \pi A_m) = C(A_m^2)$ converges to $\cos(-i \pi A_\infty) = C(A_\infty^2)$ in the operator norm, since $A^2_m$ converges to $A^2_\infty$ in the operator norm. Let us define a holomorphic function $T(z)$ by $T(z) = S(z-1)$. In the expansion of $T(z)$ in $z$, the constant part is absent, because $T(0) = S(-1) = \sin \pi = 0$. Therefore $T(A^2 + \id)$ is a compact operator for any $A \in \Fred(E)$. Now, let us see the estimate
\begin{multline*}
\lVert \sin(-i\pi A_m) - \sin(-i\pi A_\infty) \rVert
= \lVert A_m T(A_m^2 + \id) 
- A_\infty T(A_\infty^2 + \id) \rVert \\
\le 
\lVert A_m \rVert \cdot
\lVert 
T(A_m^2 + \id) - T(A_\infty^2 + \id) 
\rVert
+
\lVert (A_m - A_\infty) T(A_\infty^2 + \id) \rVert.
\end{multline*}
Because $A_m$ is skew-adjoint, we have $\lVert A_m \rVert^2 = \lVert A_m^*A_m \rVert = \lVert A_m^2 \rVert$. Thus, we have $\lVert A_m \rVert \to \lVert A_\infty \rVert$ and $\lVert T(A_m^2 + \id) - T(A_\infty^2 + \id) \rVert \to 0$ as $m \to \infty$, since $A^2_m$ converges to $A^2_\infty$ in the operator norm. The compact operator $T(A^2_\infty + 1)$ maps the unit sphere $\{ v \in E |\ \lVert v \rVert = 1 \}$ in $E$ to a compact subset in $E$. On the compact subset, $A_m$ converges to $A_\infty$ uniformly, so that $\lVert (A_m - A_\infty) T(A_\infty^2 + \id) \rVert \to 0$. In summary, $\vartheta(A_m)\epsilon$ converges to $\vartheta(A_\infty)\epsilon$ in the operator norm, and $\vartheta : \Fred(E) \to \Gr(E)$ is a continuous map.
\end{proof}

\begin{lem}
The map $\vartheta$ in Definition \ref{dfn:Aityha_Singer_map} is well-defined, continuous, and gives rise to a map of $\phi$-twisted fiber bundles on the groupoid $\X$.
\end{lem}

\begin{proof}
In the same way as in Lemma \ref{lem:key_continuity}, we can show that $A \mapsto -e^{\pi A}\epsilon$ is well-defined. It is easy to verify the anti-commutation relation between $-e^{\pi A}\epsilon$ and the Clifford action on $\acute{E}$. This proves that the map $\vartheta : \Fred(E) \to \Gr(\acute{E})$ is well-defined as a map of the fiber bundles on $\X_0$. The continuity of $\vartheta$ follows from Lemma \ref{lem:key_continuity}, because of the local triviality of $E$. We can also directly verify that $\vartheta(A)$ is compatible with the twisted action on $\acute{E}$, which means that $\vartheta : \Fred(E) \to \Gr(\acute{E})$ is a map of fiber bundles on $\X$.
\end{proof}

As a result of the lemma above, we have a map of sections
$$
\vartheta : \Gamma(\X, \Fred(E)) \to \Gamma(\X, \Gr(\acute{E})),
$$
which is continuous, provided that the spaces of sections are topologized by the the compact open topologies. Clearly, this map preserves the operations of taking the direct sum. If $E$ is locally universal and $\gamma_* \in \Gamma(\X, \Fred(E)^\dagger)$, then $\vartheta(\gamma_*) = \epsilon$, because $\mathrm{Spec}(\gamma_*) = \{ \pm i \}$. Consequently, $\vartheta$ induces a well-defined map of monoids
$$
\vartheta : {}^\phi K^{(\tau, c) + (p, q)}(\X) \to
{}^\phi \K^{(\acute{\tau}, c) + (q, p)}(\X).
$$

\begin{thm} \label{thm:Fredholm_vs_Karoubi}
Let $\X$ be a local quotient groupoid, $\phi \in \Phi(\X)$ an object, and $(\tau, c)$ a $\phi$-twist on $\X$. For any $p, q \ge 0$, the monoid map
$$
\vartheta :\
{}^\phi K^{(\tau, c) + (p, q)}(\X)
\to
{}^\phi \K^{(\acute{\tau}, c) + (q, p)}(\X)
$$
is bijective. In particular, ${}^\phi \K^{(\acute{\tau}, c) + (q, p)}(\X)$ gives rise to an abelian group. Further, if we put ${}^\phi \K^{(\acute{\tau}, c) + n}(\X) = {}^\phi \K^{(\acute{\tau}, c) + (n, 0)}(\X)$, then $\vartheta$ induces an isomorphism of groups
$$
{}^\phi K^{(\tau, c) + n}(\X) \cong {}^\phi \K^{(\acute{\tau}, c) + n}(\X).
$$
\end{thm}

\begin{proof}
As before, we apply the reduction argument as in Lemma \ref{lem:locally_unversal_bundle} and \cite{FHT1} (Proposition A.19). Then, it is enough to see that $\vartheta : \Gamma(\X, \Fred(E)) \to \Gamma(\X, \Gr(\acute{E}))$ is a weak homotopy equivalence when $\X$ is the groupoid $\pt//G$ associated to a compact Lie group $G$ and $\tau$ is trivial. In this case, $\vartheta$ is a homotopy equivalence, as shown in Appendix \ref{sec:mackey_decomposition} (Lemma \ref{lem:Fredholm_vs_Karoubi_point_case}).
\end{proof}


\begin{rem} \label{rem:twist_change}
One can avoid the twist change $\tau \mapsto \acute{\tau}$ by using self-adjoint operators instead of skew-adjoint operators in Definition \ref{dfn:fredholm_family}. To be precise, we define $\acute{\mathrm{F}}\mathrm{red}(E)$ by replacing (i), (ii) and (iii) in Definition \ref{dfn:fredholm_family} by
\begin{itemize}
\item[(i)']
$A$ are self-adjoint: $A^* = A$.

\item[(ii)']
$A^2 - \id$ are compact.

\item[(iii)']
$\mathrm{Spec}(A) \subset [-1, 1]$.
\end{itemize}
Using $\acute{\mathrm{F}}\mathrm{red}(E)$, we also define ${}^\phi\acute{K}^{(\tau, c) + (p, q)}(\X)$ and ${}^\phi\acute{K}^{(\tau, c) -n}(\X, \Y)$ as in Definition \ref{dfn:bigraded_K} and Definition \ref{dfn:Freed_Moore_K_theory}. By Lemma \ref{lem:categorical_correspondence_by_epsilon}, we have an isomorphism  of fiber bundles
\begin{align*}
&\Fred(E) \to \acute{\mathrm{F}}\mathrm{red}(\acute{E}), &
&A \mapsto \acute{A} = A \epsilon,
\end{align*}
where $\epsilon$ is the $\Z_2$-grading of $E$. This induces the isomorphisms of groups
\begin{align*}
{}^\phi K^{(\tau, c) + (p, q)}(\X)
&\cong
{}^\phi \acute{K}^{(\acute{\tau}, c) + (q, p)}(\X), &
{}^\phi K^{(\tau, c) + n}(\X)
&\cong
{}^\phi \acute{K}^{(\acute{\tau}, c) + n}(\X).
\end{align*}
Hence Theorem \ref{thm:Fredholm_vs_Karoubi} provides the isomorphisms without the twist change
\begin{align*}
{}^\phi \acute{K}^{(\tau, c) + (q, p)}(\X)
&\cong
{}^\phi \K^{(\tau, c) + (q, p)}(\X), &
{}^\phi \acute{K}^{(\tau, c) + n}(\X)
&\cong
{}^\phi \K^{(\tau, c) + n}(\X).
\end{align*}
Note that the counterpart of Corollary \ref{cor:degree_correspondence} reads
$$
{}^\phi\acute{K}^{(\tau, c) + p - q}(\X)
\cong {}^\phi\acute{K}^{(\tau, c) + (p, q)}(\X).
$$
Note also that the counterpart of $\imath$ in \S\S\ref{subsec:finite_rank} for ${}^\phi\acute{K}^{(\tau, c) + n}_G(X)$ is defined for
$$
{}^\phi\acute{K}^{(\tau, c) + n}(X)_{\mathrm{fin}}
= \pi_0({}^\phi\Vect_G^{(\tau, c) + (0, n)}(X)_{\mathrm{fin}})/
\pi_0({}^\phi\Vect_G^{(\tau, c) + (0, n+1)}(X)_{\mathrm{fin}}),
$$
which is isomorphic to ${}^\phi K^{(\acute{\tau}, c) + n}(X)_{\mathrm{fin}}$ by $E \mapsto \acute{E}$. Finally, we remark that the use of self-adjoint operators affects the signs of the degree shifts corresponding to twists in Theorem \ref{thm:degree_shift}, for example,
$$
{}^\phi \acute{K}^{(\tau, c) + c_{\phi} + n}(\X, \Y) 
\cong 
{}^\phi \acute{K}^{(\tau, c) + n - 2}(\X, \Y).
$$
\end{rem}


\subsection{Finite-dimensional Karoubi formulation}

Let us consider the same setup as in \S\S\ref{subsec:finite_rank} to introduce a finite-dimensional Karoubi formulation.

\begin{dfn}[triple] \label{dfn:karoubi_triple}
Let $\X$ be the quotient groupoid $X//G$ associated to an action of a finite group $G$ on a compact Hausdorff space $X$, $\phi : X//G \to \pt//\Z_2$ the map of groupoids associated to a homomorphism $\phi : G \to \Z_2$, and $(L, \tau, c)$ a $\phi$-twisted $\Z_2$-graded extension of $X//G$.
\begin{itemize}
\item[(a)]
We define a \textit{triple} $(E, \eta_0, \eta_1)$ on $X//G$ by the requirement that $(E, \eta_0, \rho, \gamma)$ and $(E, \eta_1, \rho, \gamma)$ are objects of ${}^\phi\Vect^{(\tau, c) + (p, q)}_G(X)_{\mathrm{fin}}$.

\item[(b)]
We define an isomorphism of triples $f : (E, \eta_0, \eta_1) \to (E', \eta'_0, \eta'_1)$ to be an isomorphism of vector bundles $f : E \to E'$ on $X$ which gives isomorphisms $f : (E, \eta_i, \rho, \gamma) \to (E', \eta'_i, \rho', \gamma')$ in ${}^\phi\Vect^{(\tau, c) + (p, q)}_G(X)_{\mathrm{fin}}$ for $i = 0, 1$.
\end{itemize}
\end{dfn}

By definition, a triple $(E, \eta_0, \eta_1)$ can be regarded as a twisted vector bundle $(E, \eta_1, \rho, \gamma)$ on $X//G$ equipped with a gradation $\eta_0$ in the sense of Definition \ref{dfn:gradation}, where the compactness of $\eta_0 - \eta_1$ is automatically satisfied by the finite-dimensionality of $E$. The direct sum of triples is defined by $(E, \eta_0, \eta_1) \oplus (E', \eta'_0, \eta'_1) = (E \oplus E', \eta_0 \oplus \eta'_0, \eta_1 \oplus \eta'_1)$.

\begin{dfn} \label{dfn:finite_rank_karoubi}
We assume the same setting as in Definition \ref{dfn:karoubi_triple}.
\begin{itemize}
\item[(a)]
We define ${}^\phi\mathcal{M}^{(\tau, c) + (p, q)}_G(X)_{\mathrm{fin}}$ to be the monoid of the isomorphism classes of triples on $X//G$.

\item[(b)]
We define ${}^\phi\mathcal{Z}^{(\tau, c) + (p, q)}_G(X)_{\mathrm{fin}} \subset {}^\phi\mathcal{M}^{(\tau, c) + (p, q)}_G(X)_{\mathrm{fin}}$ to be the submonoid consisting of triples $(E, \eta_0, \eta_1)$ such that $\eta_0$ is homotopic to $\eta_1$ as gradations.

\item[(c)]
We define ${}^\phi\K^{(\tau, c) + (p, q)}_G(X)_{\mathrm{fin}} = {}^\phi\mathcal{M}^{(\tau, c) + (p, q)}_G(X)_{\mathrm{fin}}/{}^\phi\mathcal{Z}^{(\tau, c) + (p, q)}_G(X)_{\mathrm{fin}}$ to be the quotient monoid.
\end{itemize}
\end{dfn}

\begin{lem} \label{lem:formula_in_finite_Karoubi}
The monoid ${}^\phi \K^{(\tau, c) + (p, q)}_G(X)_{\mathrm{fin}}$ is an abelian group, in which the additive inverse of $[E, \eta_0, \eta_1]$ is given by $[E, \eta_1, \eta_0]$. It also holds that
$$
[E, \eta_0, \eta_1] + [E, \eta_1, \eta_2] = [E, \eta_0, \eta_2].
$$
\end{lem}

\begin{proof}
To see that the quotient monoid is an abelian group, we verify the monoid morphism $I([E, \eta_0, \eta_1]) = [E, \eta_1, \eta_0]$ satisfies the assumptions in Lemma \ref{appendix:lem_quotient_monoid}. Then the non-trivial thing is that $[E, \eta_0, \eta_1] + [E, \eta_1, \eta_0] = 0$, namely, the gradations $\eta_0 \oplus \eta_1$ and $\eta_1 \oplus \eta_0$ on $E \oplus E$ are homotopic. As given in \cite{Kar} (4.16 Lemma), the family of gradations
\begin{align*}
\tilde{\eta}(\theta)
&=
\left(
\begin{array}{rr}
\cos \theta & - \sin \theta \\
\sin \theta & \cos\theta
\end{array}
\right)
\left(
\begin{array}{cc}
\eta_0 & 0 \\
0 & \eta_1
\end{array}
\right)
\left(
\begin{array}{rr}
\cos\theta & \sin\theta \\
-\sin\theta & \cos\theta
\end{array}
\right) \\
&=
\left(
\begin{array}{cc}
\eta_0 \cos^2\theta + \eta_1 \sin^2\theta & 
(\eta_0 - \eta_1)\cos\theta\sin\theta \\
(\eta_0 - \eta_1)\cos\theta\sin\theta & 
\eta_0\sin^2\theta + \eta_1\cos^2\theta
\end{array}
\right)
\end{align*}
realizes such a homotopy in our setting. The remaining formula can be shown in the same way as in \cite{Kar} (4.17 Lemma).
\end{proof}

Using the idea of the proof of Lemma \ref{lem:weak_periodicity}, we can also prove
$$
{}^\phi \K^{(\tau, c) + (p, q)}(\X)_{\mathrm{fin}}
\cong {}^\phi \K^{(\tau, c) + (p+1, q+1)}(\X)_{\mathrm{fin}}.	
$$

\medskip

As in \S\S\ref{subsec:finite_rank}, we can relate finite-dimensional formulation ${}^\phi \K^{(\tau, c) + (p, q)}_G(X)_{\mathrm{fin}}$ with the infinite-dimensional formulation ${}^\phi \K^{(\tau, c) + (p, q)}_G(X)$.

\begin{lem}
Under the assumptions in Definition \ref{dfn:karoubi_triple}, there is a homomorphism of monoids
$$
\jmath :\ 
{}^\phi \K^{(\tau, c) + (p, q)}_G(X)_{\mathrm{fin}} \to
{}^\phi \K^{(\tau, c) + (p, q)}_G(X).
$$
\end{lem}

\begin{proof}
Given a triple $(E, \eta_0, \eta_1)$ on $X//G$, we put $\epsilon_E = \eta_1$ to regard $E$ as a $\Z_2$-graded vector bundle. In particular, $E \in {}^\phi\Vect^{(\tau, c) + (p, q)}(X//G)$. Then we can embed $E$ into a locally universal bundle $E_{\mathrm{uni}}$. If $\epsilon_{E^\perp}$ denotes the $\Z_2$-grading of the orthogonal complement $E^\perp$ of $E$, then the $\Z_2$-grading $\epsilon_{\mathrm{uni}}$ of $E_{\mathrm{uni}}$ is expressed as $\epsilon_\mathrm{uni} = \epsilon_E \oplus \epsilon_{E^\perp}$. Now, we have a self-adjoint involution $\eta = \eta_0 \oplus \epsilon_{E^\perp}$ on $E_{\mathrm{uni}}$ such that $\eta - \epsilon_{\mathrm{uni}} = (\eta_0 - \eta_1) \oplus 0$ is compact. This gives a gradation $\eta \in \Gamma(X//G, \Gr(E_{\mathrm{uni}}))$, and we define $\jmath$ by the assignment $(E, \eta_0, \eta_1) \mapsto \eta$. Using the property $E_{\mathrm{uni}} \oplus E_{\mathrm{uni}} \cong E_{\mathrm{uni}}$, we can show that $\jmath$ is well-defined. It is then clear that $\jmath$ is a homomorphism. 
\end{proof}

To prove that $\jmath$ is bijective, we show that any gradation on a locally universal bundle admits a ``finite dimensional approximation''.

\begin{lem} \label{lem:approximate_compact_on_vector_space}
Let $X$ be a compact Hausdorff space.

\begin{itemize}
\item[(a)]
Let $\mathcal{E}$ be a separable Hilbert space, and $\{ K_x \}_{x \in X}$ a family of self-adjoint compact operators on $\mathcal{E}$ which are continuous in the operator norm. Then, for any $r > 0$, there is a finite rank subspace $\mathcal{F} \subset \mathcal{E}$ such that $\lVert K_x - PK_x \rVert < r$ for all $x \in X$, where $P : \mathcal{E} \to \mathcal{E}$ is the orthogonal projection onto $\mathcal{F} \subset \mathcal{E}$.

\item[(b)]
We additionally suppose in (a) that $\mathcal{E}$ gives rise to a $(\phi, c)$-twisted vector bundle on $\pt//G$ with compatible $Cl_{p, q}$-action for homomorphisms $\phi : G \to \Z_2$ and $c : G \to \Z_2$. Then we can take the subspace $\mathcal{F}$ in (a) so that $\mathcal{F}$ is a $(\phi, c)$-twisted vector bundle on $\pt//G$.
\end{itemize}
\end{lem}

\begin{proof}
For (a), let $B(v; R) \subset \mathcal{E}$ denote the open ball centered at $v \in \mathcal{E}$ and radius $R$, and $\overline{B(v; R)}$ its closure. Then $\{ K_xv \in \mathcal{E} |\ x \in X, v \in \overline{B(0; 1)} \}$ is a compact subset. To see this, we define a continuous map $k : X \times \mathcal{E} \to X \times \mathcal{E}$ by $k(x, v) = K_xv$. As shown in \cite{Se3} (Proposition 2.1), the projection $X \times \mathcal{E} \to X$ restricts to a proper map $k(X \times \overline{B(0; 1)}) \to X$. Since $X$ is compact, so are $k(X \times \overline{B(0; 1)})$ and its image under the projection $X \times \mathcal{E} \to \mathcal{E}$.

As a result, we can find a finite number of vectors $v_1, \cdots, v_n \in \mathcal{E}$ so that the open balls $B(v_i; r/2)$ cover $\{ K_xv \in \mathcal{E} |\ x \in X, v \in \overline{B(0; 1)} \}$. Let $\mathcal{F} = \mathrm{Span}\{ v_1, \ldots, v_n \}$ be the subspace spanned by the vectors $v_1, \ldots, v_n$, and $P$ the orthogonal projection onto $\mathcal{F}$. Then, for any $x \in X$ and $v \in \overline{B(0; 1)}$, we can find a vector $v_i$ from $v_1, \cdots, v_n$ such that $\lVert K_xv - v_i \rVert < r/2$, so that
\begin{align*}
\lVert K_xv - PK_xv \rVert
&\le
\lVert K_xv - v_i \rVert + \lVert v_i - PK_xv \rVert
=
\lVert K_xv - v_i \rVert 
+ \lVert Pv_i - PK_xv \rVert \\
&\le
\Vert K_xv - v_i \rVert + \lVert v_i - K_xv \rVert 
< \frac{r}{2} + \frac{r}{2} = r.
\end{align*}
This implies that $\lVert K_x - PK_x \rVert < r$ for any $x \in X$.

For (b), we take the subspace $\mathcal{F} \subset \mathcal{E}$ in (a) to be
$$
\mathcal{F}
= \mathrm{Span}\{ g \xi v_i \in \mathcal{E} |\ 
g \in G, \xi \in Cl_{p, q}, 1 \le i \le n \}.
$$
As a vector space, $\mathcal{F}$ still remains finite rank. Since $v_1, \ldots, v_n \in \mathcal{F}$, the succeeding argument in (a) works without change.
\end{proof}

\begin{lem} \label{lem:approximation_gradation}
Under the assumptions in Definition \ref{dfn:karoubi_triple}, let $E$ be the locally universal $(\phi, \tau, c)$-twisted vector bundle on $X//G$ with $Cl_{p, q}$-action. For any gradation $\eta \in \Gamma(X//G, \Gr(E))$, there exists a finite rank $(\phi, \tau, c)$-twisted subbundle $F \subset E$ with $Cl_{p, q}$-action such that $\eta$ is homotopic to $\eta_F \oplus \epsilon_{E/F}$ within gradations, where $\eta_F \in \Gamma(X//G, \Gr(F))$ is the section expressed as $\eta_F = \lvert R \eta i \rvert^{-1} R\eta i$ by using the inclusion $i : F \to E$ and the projection $R : E \to F$, and $\epsilon_{E/F}$ is the $\Z_2$-grading of the orthogonal complement $F^\perp$ of $F \subset E = F \oplus F^\perp$.
\end{lem}

\begin{proof}
Let $\X = X//G$ be the quotient groupoid. First of all, we point out that it suffices to consider the case where the twisting data $\phi$, $(L, \tau)$ and $c$ are all trivial. This is a consequence of a reduction argument in \cite{F-M}:
\begin{enumerate}

\item
Let $(g, (-1)^i) \in G' = G \times \Z_2$ act on $X' = X \times \Z_2$ by $(x, (-1)^j) \mapsto (gx, \phi(g)(-1)^{i+j})$. Because the normal subgroup $1 \times \Z_2 \subset G'$ acts on $X'$ freely, the projections $X' \to X$ and $G' \to X$ induce a local equivalence of groupoids $\pi : \X' \to \X$. Recall that the homomorphism $\phi : G \to \Z_2$ defines an object $\phi \in \Phi(\X)$ in the category introduced in \S\ref{sec:twisted_bundle}. There is an isomorphism $\psi : \pi^*\phi \to \phi'$ in $\Phi(\X')$, where $\phi'$ is the object associated to the projection $\phi' : G' \to \Z_2$. The isomorphism $\psi$ induces equivalences of categories ${}^{\pi^*\phi}\Twist(\X') \to {}^{\phi'}\Twist(\X')$ and ${}^{\pi^*\phi}\Vect^{(\pi^*\tau', \pi^*c)}(\X') \to {}^{\phi'}\Vect^{(\tau', c')}(\X')$, where $(L', \tau', c')$ is the $\phi'$-twist corresponding to $\pi^*(L, \tau, c)$ under the former equivalence. The inclusion $G \to G'$ induces a map from the quotient groupoid $\tilde{\X} = X'//G$ to $\X' = X'//G'$. The restriction $(L', \tau')|_{\tilde{\X}}$ is an untwisted central extension of $\X'$, so that any $(\phi', \tau', c')$-twisted vector bundle on $\X'$ restricts to a $(\tau'|_{\tilde{\X}}, c'|_{\tilde{\X}})$-twisted vector bundle on $\tilde{\X}$. As shown in \cite{F-M} (proof of Lemma 10.17), an involution on the groupoid $\X'$ is able to recover the information on $(\phi', \tau', c')$-twisted vector bundles on $\X'$ from the $(\tau'|_{\tilde{\X}}, c'|_{\tilde{\X}})$-twisted vector bundles on $\tilde{\X}$. As a result, we can consider the quotient groupoid $\tilde{\X} = X'//G$ with the $\Z_2$-graded central extension $(L', \tau', c')$ instead of the original groupoid $\X = X//G$ with the $\phi$-twisted $\Z_2$-graded extension $(L, \tau, c)$. Put differently, we can assume that $\phi$ is trivial from the beginning. As is pointed out in \cite{F-M} (Appendix E), a similar argument can be applied in order to suppress the twisting datum $c$ (cf.\ Proposition \ref{prop:exact_sequence_with_c}), so that we can also assume that $c$ is trivial.

\item
Let $\X = X//G$ be the quotient groupoid associated to an action of a finite group $G$ on a space $X$, and $(L, \tau)$ a central extension of $\X$. By an argument in \cite{F-M} (Lemma E.1), one can show that $\X$ is weakly equivalent to a global quotient $Y//H$ of a compact space $Y$ by a compact Lie group $H$, on which $(L, \tau)$ can be trivialized. Concretely, the push-forward $(\partial_1)_*L \to X$ gives rise to a $\tau$-twisted $G$-equivariant vector bundle whose rank $n$ agrees with the order of $G$. We write $P$ for the unitary frame bundle of $(\partial_1)_*L$, and $P/U(1)$ for the associated principal bundle whose structure group is the projective unitary group $PU(n) = U(n)/U(1)$. Because $L_g = L|_{\{ g \} \times X}$ is a line bundle on $X$ for each $g \in G$, a unitary frame of $(\partial_1)_*L$ at $x \in X$ gives a unitary frame of $L_g \otimes (\partial_1)_*L$ at $x$, which is unique up to a multiple of an element of $U(1)$. This unitary frame is mapped to a unitary frame of $(\partial_1)_*L$ at $gx \in X$ by the $\tau$-twisted action $\rho_g : L_g \otimes (\partial_1)_*L \to (\partial_1)_*L$ of $g \in G$. This transformation of unitary frames up to multiples of elements in $U(1)$ defines an honest $G$-action on $P/U(1)$ commuting with the right action of $PU(n)$. It turns out that this $G$-action is free. Now, we get the space $Y = P/U(1)$ with the right action of $H = PU(n)$.

\end{enumerate}

Under the assumption that $\phi$, $L \to \X_1$, $\tau$ and $c$ are trivial, one can apply the argument in \cite{Se2} to realize the locally universal $\Z_2$-graded $G$-equivariant vector bundle $E$ with $Cl_{p, q}$-action as $X \times \mathcal{E}$, where $\mathcal{E}$ is a separable infinite-dimensional Hilbert space such that $\mathcal{E}$ gives rise to a $\Z_2$-graded vector bundle on $\pt//G$ with $Cl_{p, q}$-action. Then $\eta \in \Gamma(X//G, \Gr(E))$ is expressed as $\eta(x, v) = (x, \eta_xv)$ by using a self-adjoint involution $\eta_x : \mathcal{E} \to \mathcal{E}$ which is continuous in $x$ with respect to the operator norm and is compatible with the $G$-action and the $Cl_{p, q}$-action. We take $r > 0$ to be a positive real number such that: for each $x \in X$, any bounded operator $\eta' : \mathcal{E} \to \mathcal{E}$ satisfying $\lVert \eta_x - \eta' \rVert < r$ admits a bounded inverse. Such an $r$ exists because $X$ is compact and the invertible bounded operators on $\mathcal{E}$ form an open subset in the space of bounded operators equipped with the operator norm topology.

We put $K_x = \eta_x - \epsilon$ to define a continuous family of self-adjoint compact operators on $\mathcal{E}$. Then, by Lemma \ref{lem:approximate_compact_on_vector_space} (b), we have a finite rank invariant subspace $\mathcal{F} \subset \mathcal{E}$ such that $\lVert K_x - PK_x \rVert < r/2$ for all $x \in X$, where $P : \mathcal{E} \to \mathcal{E}$ is the orthogonal projection onto $\mathcal{F}$. By construction, $P$ commutes with the $\Z_2$-grading $\epsilon$, the $G$-action and the $Cl_{p, q}$-action on $\mathcal{E}$.

Now, we have a finite rank $\Z_2$-graded $G$-equivariant subbundle $F = X \times \mathcal{F} \subset X \times \mathcal{E}$ with $Cl_{p, q}$-action. We put $H_{x, t} = \epsilon + (1-t)K_x + tPK_xP$ for $x \in X$ and $t \in [0, 1]$. Since $K_x$ as well as $P$ are self-adjoint, we get
\begin{align*}
\lVert K_x - PK_xP \rVert
&\le
\lVert K_x - PK_x \rVert + \lVert PK_x - PK_xP \rVert \\
&\le
\lVert K_x - PK_x \rVert + \lVert K_x - K_xP \rVert \\
&= \lVert K_x - PK_x \rVert + \lVert (K_x - K_xP)^*\rVert \\
&= \lVert K_x - PK_x \rVert + \lVert K_x - PK_x \rVert
<
\frac{r}{2} + \frac{r}{2} = r.
\end{align*}
Because of the estimate
$$
\lVert H_{x, t} - \eta_x \rVert
=
t \lVert K_x - PK_xP \rVert
< r,
$$
the operator $H_{x, t} : \mathcal{E} \to \mathcal{E}$ is invertible for all $x \in X$ and $t \in [0, 1]$. We then define $\eta_{x, t} : \mathcal{E} \to \mathcal{E}$ by $\eta_{x, t} = \lvert H_{x, t} \rvert^{-1} H_{x, t}$, which is a self-adjoint involution. Notice that $T = H_{x, t}\epsilon$ is an invertible operator on $\mathcal{E}$ which differs from the identity by a compact operator. By the spectral theorem for compact operators, the unitary operator $(TT^*)^{-1/2}T$ differs from the identity by a compact operator. Therefore 
$$
\eta_{x, t} = \lvert H_{x, t} \rvert^{-1}H_{x, t} = (TT^*)^{-1/2}T\epsilon
$$
differs from $\epsilon$ by a compact operator. As a result, we get a gradation $\eta_t \in \Gamma(X//G, \Gr(E))$ by defining the bundle map $\eta_t : E \to E$ as $\eta_t(x, v) = (x, \eta_{x, t}v)$. This is a homotopy from $\eta_0 = \eta$ to $\eta_1$ within gradations. Since $\epsilon + PK_xP$ commutes with $P$, we can decompose $\eta_1$ as $\eta_1 = \eta_F \oplus \eta_F^\perp$ by using gradations $\eta_F \in \Gamma(X//G, \Gr(F))$ and $\eta_F^\perp \in \gamma(X//G, \Gr(F^\perp))$. Because of the expression $H_{x, 1} = P \eta_x P$, if $i : F \to E$ and $R : E \to F$ respectively denote the inclusion and the projection, then
$$
\eta_F = R \eta_1 i
= R(\lvert P \eta P \rvert^{-1}{P \eta P})i
= \lvert R \eta i \rvert^{-1}(R \eta i).
$$
Similarly, we find $\eta_F^\perp = \epsilon_{E/F}$.
\end{proof}

\begin{lem} \label{lem:another_approximation}
Let $F'$ be another subbundle of $E$ as in Lemma \ref{lem:approximation_gradation}, and $\eta_{F'} \in \Gamma(X//G, \Gr(F'))$ the associated gradation. Then there exists a subbundle $F'' \subset E$ such that $F$ and $F'$ are subbundle of $F''$ and $\eta_{F} \oplus \epsilon_{F''/F}$ and $\eta_{F'} \oplus \epsilon_{F''/F'}$ are homotopic within the gradations of $F''$, where $\epsilon_{F''/F}$ and $\epsilon_{F''/F'}$ are the $\Z_2$-gradings of the orthogonal complements of $F \subset F''$ and $F' \subset F''$, respectively.
\end{lem}

\begin{proof}
As in the proof of Lemma \ref{lem:approximation_gradation}, we assume that $\phi$, $L$, $\tau$ and $c$ are trivial, so that the locally universal bundle is realized as $E = X \times \mathcal{E}$. We express $\eta$ as $\eta(x, v) = (x, \eta_xv)$, and put $K_x = \eta_x - \epsilon$. We suppose that the finite rank subbundle $F' \subset E$ and the gradation $\eta_{F'}$ are constructed from a certain suitable choice of a real number $r' > 0$ and a finite rank invariant subspace $\mathcal{F}' \subset \mathcal{E}$ along the proof of Lemma \ref{lem:approximation_gradation}. Therefore the fibers of $F = X \times \mathcal{F}$ and $F' = X \times \mathcal{F}'$ at $x \in X$ are respectively the images of the orthogonal projections $P$ and $P'$ satisfying
\begin{align*}
\lVert K_x - PK_xP \rVert &< r, &
\lVert K_x - P'K_xP' \rVert & < r'.
\end{align*}
We put $H_x = \epsilon + PK_xP$ and $H'_x = \epsilon + P' K_x P'$, which are self-adjoint invertible operators. Then the gradations $\eta_F$ and $\eta_{F'}$ at $x \in X$ are realized as
\begin{align*}
\eta_{F, x} &= R \lvert H_x \rvert^{-1} H_x i, &
\eta_{F', x} &= R' \lvert H'_x \rvert^{-1} H'_x i',
\end{align*}
where $R : \mathcal{E} \to \mathcal{F}$ and $R' : \mathcal{E} \to \mathcal{F'}$ are the projections, and $i : \mathcal{F} \to \mathcal{E}$ and $i' : \mathcal{F}' \to \mathcal{E}$ are the inclusions. 

To give a subbundle $F'' \subset E$, we put $r'' = \max\{ r, r' \}$. For each $x \in X$, any bounded operator $\eta'$ on $\mathcal{E}$ such that $\lVert \eta' - \eta_x \rVert < r''$ admits a bounded inverse. Let $\mathcal{F}''$ be a finite rank invariant subspace which contains both $\mathcal{F}$ and $\mathcal{F}'$. In view of the proof of Lemma \ref{lem:approximate_compact_on_vector_space}, the orthogonal projection $P''$ onto $\mathcal{F}''$ satisfies $\lVert K_x - P'' K_x \rVert < r''/2$ for all $x \in X$. By construction, the vector bundle $F'' = X \times \mathcal{F}''$ on $X//G$ with $Cl_{p, q}$-action contains both $F$ and $F'$. It also holds that $\lVert K_x - P'' K_x P'' \rVert < r''$ for all $x \in X$.

To show that $\eta_F \oplus \epsilon_{F''/F}$ and $\eta_{F'} \oplus \epsilon_{F''/F'}$ are homotopic within gradations of $F''$, we use an intermediate gradation. We put $H''_x = \epsilon + P'' K_x P''$. By the proof of Lemma \ref{lem:approximation_gradation}, this is a self-adjoint invertible operator for each $x \in X$, and we have a gradation $\eta_{F''}$ of $F''$ given by
$$
\eta''_x = R'' \lvert H''_x \rvert^{-1} H''_x i'',
$$
where $R'' : \mathcal{E} \to \mathcal{F}''$ is the projection and $i'' : \mathcal{F}'' \to \mathcal{E}$ is the inclusion. We then put $H_{x, t} = (1 - t)H_x + t H''_x$ for $x \in X$ and $t \in [0, 1]$. This is self-adjoint on $\mathcal{E}$, and further invertible, since 
\begin{align*}
\lVert H_{x, t} - \eta_x \rVert
&=
\lVert (1 - t)(PK_xP - K_x) + t(P''K_xP'' - K_x) \rVert \\
&<
(1 - t) r'' + t r'' = r''.
\end{align*}
Then $\eta_{x, t} = R'' \lvert H_{x, t} \rvert^{-1} H_{x, t} i''$ defines a homotopy of gradations $\eta_t$ on $F''$ which connects $\eta_0 = \eta_F \oplus \epsilon_{F''/F}$ with $\eta_1 = \eta''$. The same construction gives a homotopy of gradations $\eta'_t$ from $\eta'_0 = \eta_{F'} \oplus \epsilon_{F''/F'}$ to $\eta'_1 = \eta''$. Hence $\eta_F \oplus \epsilon_{F''/F}$ and $\eta_{F'} \oplus \epsilon_{F''/F'}$ are homotopic within the gradations of $F''$.
\end{proof}

\begin{thm} \label{thm:finite_and_infinite_dimensional_Karoubi_formulations}
Under the assumptions in Definition \ref{dfn:karoubi_triple}, the homomorphism 
$$
\jmath :\ 
{}^\phi \K^{(\tau, c) + (p, q)}_G(X)_{\mathrm{fin}} \to
{}^\phi \K^{(\tau, c) + (p, q)}_G(X)
$$
is bijective. In particular, if we put ${}^\phi \K^{(\tau, c) + n}_G(X)_{\mathrm{fin}} \cong {}^\phi \K^{(\tau, c) + (n, 0)}_G(X)_{\mathrm{fin}}$, then $\jmath$ induces an isomorphism of groups
$$
{}^\phi \K^{(\tau, c) + n}_G(X)_{\mathrm{fin}} \cong
{}^\phi \K^{(\tau, c) + n}_G(X).
$$
\end{thm}

\begin{proof}
We construct a homomorphism of monoids
$$
\alpha :\ {}^\phi \K^{(\tau, c) + (p, q)}_G(X) \to
{}^\phi \K^{(\tau, c) + (p, q)}_G(X)_{\mathrm{fin}}
$$
which gives the inverse to $\jmath$. For this construction, let $E = E_{\mathrm{uni}}$ be a $(\phi, \tau, c)$-twisted locally universal bundle on $X//G$ with $Cl_{p, q}$-action, and $\eta \in \Gamma(X//G, \Gr(E))$ a gradation. Thanks to Lemma \ref{lem:approximation_gradation}, there is a finite rank subbundle $F \subset E$ such that $\eta$ is homotopic to $\eta_F \oplus \epsilon_{E/F}$, where $\eta_F \in \Gamma(X//G, \Gr(F))$ is a gradation on $F$ and $\epsilon_{E/F} \in \Gamma(X//G, \Gr(F^\perp))$ is the $\Z_2$-grading of the orthogonal complement of $F \subset E$. Denote by $\epsilon_F$ the $\Z_2$-grading of $F$. Then we have a triple $(F, \eta_F, \epsilon_F)$, and let it represent $\alpha([\eta]) \in \K^{(\tau, c) + (p, q)}_G(X)_{\mathrm{fin}}$. Once $\alpha$ is shown to be well-defined, it is clear that $\alpha$ is a homomorphism and gives the inverse to $\jmath$. The definition of $\alpha$ is independent of the choice of a subbundle $F$ as in Lemma \ref{lem:approximation_gradation}. This is a direct consequence of Lemma \ref{lem:another_approximation}. If $\eta$ and $\eta'$ are homotopic within the gradations of $E$, then $\alpha([\eta]) = \alpha([\eta'])$. This is a consequence of the definition that a homotopy between $\eta$ and $\eta'$ is a gradation on $E \times [0, 1]$.
\end{proof}


\subsection{Relationship of finite-dimensional formulations}
\label{subsec:finite_dimensional_formulations}

Summarizing the formulations of the Freed-Moore $K$-theory so far, we get the following diagram under the setting of Definition \ref{dfn:finite_rank_freed_moore_K} and Definition \ref{dfn:karoubi_triple}:
$$
\begin{CD}
{}^\phi K^{(\tau, c) + n}_G(X)_{\mathrm{fin}}
@>{\imath}>>
{}^\phi K^{(\tau, c) + n}_G(X) \\
@. @V{\cong}V{\vartheta}V \\
{}^\phi \K^{(\acute{\tau}, c) + n}_G(X)_{\mathrm{fin}}
@>{\jmath}>{\cong}>
{}^\phi \K^{(\acute{\tau}, c) + n}_G(X).
\end{CD}
$$
Here $\jmath$ has the inverse $\alpha$ given in the proof of Theorem \ref{thm:finite_and_infinite_dimensional_Karoubi_formulations}.

\begin{prop}
Under the assumptions in Definition \ref{dfn:finite_rank_freed_moore_K} and Definition \ref{dfn:karoubi_triple}, the composition
$$
\alpha \circ \vartheta \circ \imath : \
{}^\phi K^{(\tau, c) + n}_G(X)_{\mathrm{fin}}
\longrightarrow
{}^\phi \K^{(\acute{\tau}, c) + n}_G(X)_{\mathrm{fin}}
$$
is induced from the assignment of $(\acute{E}, - \acute{\epsilon}, \acute{\epsilon})$ to $E \in {}^\phi\Vect^{(\tau, c) + (n, 0)}_G(X)_{\mathrm{fin}}$, where $\acute{\epsilon} = \epsilon$ is the $\Z_2$-grading of $E$. If $c$ is trivial and $n = 0$, then $\alpha \circ \vartheta \circ \imath$ is bijective.
\end{prop}

\begin{proof}
We can readily see the description of $\alpha \circ \vartheta \circ \imath$ along the definitions of $\alpha$, $\vartheta$ and $\imath$. In the case that $c$ is trivial and $n = 0$, the inverse of $\alpha \circ \vartheta \circ \imath$ can be constructed as in \cite{Kar}: Let $(E, \eta_0, \eta_1)$ be a triple representing an element of ${}^\phi \K^{\tau + (0, 0)}_G(X)_{\mathrm{fin}}$. For $k = 0, 1$, the subbundle $\mathrm{Ker}(1 - \eta_k) \subset E$ gives rise to a $(\phi, \tau)$-twisted ungraded vector bundle on $X//G$. Therefore we have a $(\phi, \tau)$-twisted (graded) vector bundle $\mathrm{Ker}(1 - \eta_0) \oplus \mathrm{Ker}(1 - \eta_1) \in {}^\phi\Vect^{\tau + (0, 0)}_G(X)$.  Then the assignment $(E, \eta_0, \eta_1) \mapsto \mathrm{Ker}(1 - \eta_0) \oplus \mathrm{Ker}(1 - \eta_1)$ induces the inverse of $\alpha \circ \vartheta \circ \imath$. (Because of Lemma \ref{lem:triviality_of_twist_change}, the difference of $\tau$ and $\acute{\tau}$ does not matter in this case.)
\end{proof}

Since $\vartheta$ and $\jmath$ are bijective, Proposition \ref{prop:finite_rank_realizability_freed_moore} is reproved:

\begin{cor}
Under the assumptions in Definition \ref{dfn:finite_rank_freed_moore_K} and Definition \ref{dfn:karoubi_triple}, the homomorphism $\imath$ is bijective, if $c$ is trivial and $n = 0$.
\end{cor}

As is clear from the proof above, the construction of the inverse of $\alpha \circ \vartheta \circ \imath$ does not work in the presence of a non-trivial $c$. An example in which $\alpha \circ \vartheta \circ \imath$ is not bijective can be constructed from the example in \S\S\ref{subsec:finite_rank}. An example in which $\alpha \circ \vartheta \circ \imath$ is bijective is as follows: Let $G = \Z_2$ act on the unit circle $S^1 \subset \C$ trivially. As studied in \cite{G1}, we have $H^3_{\Z_2}(S^1; \Z) \cong \Z_2$, and its generator can be represented by a group $2$-cocycle $\tau \in Z^2_{\mathrm{group}}(\Z_2; C(S^1, U(1)))$ which takes the following values:
$$
\begin{array}{|c|c|c|}
\hline
\tau(g, h; u) & h = 1 & h = -1 \\
\hline
g = 1 & 1 & 1 \\
\hline
g = -1 & 1 & u \\
\hline
\end{array}
$$
We take $\phi : \Z_2 \to \Z_2$ to be the trivial homomorphism, but $c : \Z_2 \to \Z_2$ to be the identity. By using the Mayer-Vietoris exact sequence for example (\cite{SSG3}, VIII, E, 2), we have
\begin{align*}
K^{(\tau, c) + 0}_{\Z_2}(S^1) &\cong \Z_2, &
K^{(\tau, c) + 1}_{\Z_2}(S^1) &= 0.
\end{align*}
Let $E = S^1 \times \C^2$ be the product bundle on $S^1$. This bundle gives rise to a $(\tau, c)$-twisted vector bundle on $S^1//\Z_2$ by the following $\Z_2$-grading $\epsilon$ and the $(\tau, c)$-twisted $\Z_2$-action $\rho(g) : E \to E$,
\begin{align*}
\epsilon &=
\left(
\begin{array}{cc}
1 & 0 \\
0 & -1
\end{array}
\right), &
\rho(1)(u, v) &= (u, v), &
\rho(-1)(u, v) &=
(u, 
\left(
\begin{array}{cc}
0 & u \\
1 & 0
\end{array}
\right)v
).
\end{align*}
It is easy to see that any finite rank $(\tau, c)$-twisted vector bundle on $S^1//\Z_2$ is isomorphic to the direct sum of some copies of $E$ above. A consequence of this classification is that $\imath : K^{(\tau, c)+ 0}_{\Z_2}(S^1)_{\mathrm{fin}} \to K^{(\tau, c)+ 0}_{\Z_2}(S^1)$ gives rise to an isomorphism
$$
K^{(\tau, c)+ 0}_{\Z_2}(S^1)_{\mathrm{fin}} \cong
K^{(\tau, c)+ 0}_{\Z_2}(S^1).
$$
Therefore $\alpha \circ \vartheta \circ \imath : K^{(\tau, c)+ 0}_{\Z_2}(S^1)_{\mathrm{fin}} \to \K^{(\tau, c)+ 0}_{\Z_2}(S^1)_{\mathrm{fin}}$ is also an isomorphism, where the isomorphism $\tau \cong \acute{\tau}$ due to the triviality of $\phi$ is understood. As is seen, the image of $[E] \in K^{(\tau, c)+ 0}_{\Z_2}(S^1)_{\mathrm{fin}}$ is represented by the triple $(\acute{E}, -\acute{\epsilon}, \acute{\epsilon})$. Since $[E]$ is a generator, so is $[E, -\epsilon, \epsilon]$. The triple $(E, -\epsilon, \epsilon)$, which turns out to be non-trivial by the argument here, is essentially the same as the building block of nonsymmorphic topological crystalline insulators in \cite{SSG1}.


\appendix

\section{Classification of some twists}
\label{sec:classification_of_twists}

This appendix classifies some twists to be used in Appendix \ref{sec:mackey_decomposition}.

\subsection{Classification of some twists}
\label{subsec:classification_of_twists}

Let $G$ be a compact Lie group. A typical space with $G$-action is $G/H$, where $H \subset G$ is a closed subgroup and $G$ acts on $G/H$ by the left multiplication. Since the inclusions $\pt \subset G/H$ and $H \subset G$ induce the local equivalence $\pt//H \to (G/H)//G$, we have $H^1((G/H)//G; \Z_2) \cong H^1(\pt//H; \Z_2)$, and hence any object $\phi$ in $\Phi((G/H)//G)$ comes from a group homomorphism $H \to \Z_2$. Furthermore, we have
$$
H^3((G/H)//G; \Z_\phi) \cong H^3(\pt//H; \Z_{\phi|_H}),
$$
where $\phi|_H : H \to \Z_2$ is the homomorphism induced from $\phi \in \Phi((G/H)//G)$ by restriction. Thus, the classification of (ungraded) twists on $(G/H)//G$ amounts the that of twists on $\pt//H$.

Applying the argument above and computations of cohomology groups in \cite{G3}, we give the classification of (ungraded) twists in the case of $G = \Z_2$ and $G = \Z_2 \times \Z_2$, which we will need later on. In these cases, we can assume that an object $\phi \in \Phi((G/H)//G)$ is induced from a homomorphism $\phi : G \to \Z_2$. In the below, $H \subset G$ will be a subgroup.

\begin{itemize}
\item
The case where $\phi$ is trivial. 
$$
\begin{array}{|c|c|c|}
\hline
G & H & H^3((G/H)//G; \Z) \\
\hline
\Z_2 & \mbox{any subgroup} &  0 \\
\hline
\Z_2 \times \Z_2 & \mbox{not $\Z_2 \times \Z_2$} & 0 \\
\hline
\Z_2 \times \Z_2 & \Z_2 \times \Z_2 & \Z_2 \\
\hline
\end{array}
$$

\item
The case where $\phi$ is non-trivial. For $G = \Z_2$, the identity $\phi = \id$ is the unique non-trivial homomorphism $\phi : G \to \Z_2$. We have:
$$
\begin{array}{|c|c|c|c|}
\hline
G & H & \phi : G \to \Z_2 & H^3((G/H)//G; \Z_\phi) \\
\hline
\Z_2 & 1 & \id & 0 \\
\hline
\Z_2 & \Z_2 & \id & \Z_2 \\
\hline
\end{array}
$$
For $G = \Z_2 \times \Z_2$, the three non-trivial homomorphisms $\phi : \Z_2 \times \Z_2 \to \Z_2$ are permuted by the outer automorphisms of $\Z_2 \times \Z_2$. Thus, it suffices to consider a non-trivial homomorphism, for example the first projection $\phi = p_1$. In $\Z_2 \times \Z_2$, there are three non-trivial subgroups of order two: $\Z_2 \times 1$, $1 \times \Z_2$ and the image $\Delta(\Z_2)$ of the diagonal embedding $\Delta: \Z_2 \to \Z_2 \times \Z_2$. We then have:
$$
\begin{array}{|c|c|c|c|}
\hline
G & H & \phi : G \to \Z_2 & H^3((G/H)//G; \Z_\phi) \\
\hline
\Z_2 \times \Z_2 & 1 & p_1 & 0 \\
\hline
\Z_2 \times \Z_2 & \Z_2 \times 1 & p_1 & \Z_2 \\
\hline
\Z_2 \times \Z_2 & 1 \times \Z_2 & p_1 & 0 \\
\hline
\Z_2 \times \Z_2 & \Delta(\Z_2) & p_1 & \Z_2 \\
\hline
\Z_2 \times \Z_2 & \Z_2 \times \Z_2 & p_1 & \Z_2 \oplus \Z_2 \\
\hline
\end{array}
$$
\end{itemize}

\subsection{Realization by group cocycle}
\label{subsec:group_cocycle}

We next realize the non-trivial ungraded twists classified in \S\S\ref{subsec:classification_of_twists}. For this aim, we start with a review of group cocycles. 

Let $G$ be a compact Lie group, and $M$ a two sided $G$-module. We define the group of $n$-cochains of $G$ with coefficients in $M$ to be
$$
C^n_{\mathrm{group}}(G; M) = \{ \tau : G^n \to M |\ \mbox{continuous} \}
$$
and the coboundary operator $\partial : C^n_{\mathrm{group}}(G; M) \to C^{n+1}_{\mathrm{group}}(G; M)$ to be
\begin{align*}
(\partial \tau)(g_0, \ldots, g_n)
&= g_0 \cdot \tau(g_1, \ldots, g_n)
+ \sum_{i = 1}^{n-1}(-1)^i\tau(g_1, \ldots, g_ig_{i+1}, \ldots, g_n) \\
&\quad
+ (-1)^n\tau(g_0, \ldots, g_{n-1}) \cdot g_n,
\end{align*}
by using the two sided action of $G$. As usual, the group of $n$-cocycles $Z^n_{\mathrm{group}}(G; M)$ and that of $n$-coboundaries $B^n_{\mathrm{group}}(G; M)$ are defined. Then the group cohomology of $G$ with coefficients in $M$ is defined as the quotient group
$$
H^n_{\mathrm{group}}(G; M) = 
Z^n_{\mathrm{group}}(G; M)/B^n_{\mathrm{group}}(G; M).
$$

The two sided $G$-module $C(X, U(1))_\phi$ in the body of this paper is defined when a compact Lie group $G$ acts on a space $X$ from the left and a homomorphism $\phi : G \to \Z_2$ is given. The underlying group is the group $C(X, U(1))$ of $U(1)$-valued functions on $X$. The left action of $g \in G$ on $f \in C(X, U(1))$ is $f \mapsto f^{\phi(g)}$, and the right action is $f \mapsto g^*f$. In this setting, we identify a group cochain $\tau \in C^n_{\mathrm{grup}}(G; C(X, U(1))_\phi)$ with a continuous map $\tau : G^n \times X \to U(1)$. 

An example of a $2$-cocycle
$$
\tau_f \in Z^2_{\mathrm{group}}(G; C(X, U(1))_{\phi})
$$
is constructed from any homomorphism $f : G \to \Z_2$ by setting
$$
\tau_f(g, h; x)
=
\left\{
\begin{array}{ll}
1, & (\mbox{$f(g) = 1$ or $f(h) = 1$}) \\
-1. & (f(g) = f(h) = -1)
\end{array}
\right.
$$

We remark that the group cohomology $H^*_{\mathrm{group}}(G; C(X, U(1)_\phi))$ is an invariant of the quotient groupoid $X//G$ under the local equivalences. Thus, if $H \subset G$ is a closed subgroup, then the inclusion $i : H \to G$ induces an isomorphism
$$
i^* : H^n_{\mathrm{group}}(G; C(G/H, U(1))_\phi) \to
H^n_{\mathrm{group}}(H; U(1)_{\phi|_H}),
$$
where $C(\pt, U(1))$ is identified with $U(1)$. 

We also remark that the exponential exact sequence of two sided $G$-modules
$$
0 \to \Z_\phi \to \R_\phi \to U(1)_\phi \to 0
$$
induces the long exact sequence
$$
\cdot\cdot \to
H^n_{\mathrm{group}}(G; \Z_\phi) \to
H^n_{\mathrm{group}}(G; \R_\phi) \to
H^n_{\mathrm{group}}(G; U(1)_\phi) \to
H^{n+1}_{\mathrm{group}}(G; \Z_\phi) \to
\cdot\cdot,
$$
where, for $A = \Z, \R$, the two sided $H$-module $A_\phi$ above has $A$ as the underlying group, on which the left action of $g \in G$ is defined as $a \mapsto \phi(g)a$ by using a homomorphism $\phi : G \to \Z_2$ and the right action is trivial. By an averaging argument based on the Haar measure on $G$, we have $H^n_{\mathrm{group}}(G; \R_\phi) = 0$ for $n > 0$, so that $H^n_{\mathrm{group}}(G; U(1)_\phi) \cong H^{n+1}_{\mathrm{group}}(G; \Z_\phi)$ for $n > 0$. 

Suppose that $G$ is a finite group. Then the group cohomology $H_{\mathrm{group}}^n(G; A_\phi)$ appears as the $E_2$-term of a spectral sequence computing $H^n(\pt//G; A_\phi)$. Furthermore, the spectral sequence collapses at $E_2$, so that
$$
H^n_{\mathrm{group}}(G; A_\phi) \cong
H^n(\pt//G; A_\phi).
$$
In view of this isomorphism, we represent below the non-trivial ungraded twists classified in \S\S\ref{subsec:classification_of_twists} by group $2$-cocycles with coefficients in $U(1)_\phi$. 

\begin{itemize}
\item
In the case that $G = \Z_2 \times \Z_2$, $H = 1$ and $\phi : \Z_2 \times \Z_2 \to \Z_2$ is trivial,
$$
H^3(\pt//(\Z_2 \times \Z_2); \Z)
\cong H^2_{\mathrm{group}}(\Z_2 \times \Z_2; U(1)) \cong \Z_2.
$$
A group $2$-cocycle $\tau \in Z^2_{\mathrm{group}}(\Z_2 \times \Z_2; U(1))$ representing this non-trivial cohomology class is given by
$$
\tau(((-1)^{m_1}, (-1)^{n_1}), ((-1)^{m_2}, (-1)^{n_2}) 
= \exp \pi i n_1m_2.
$$

\item
In the case that $G = \Z_2$, $H = 1$ and $\phi : \Z_2 \to \Z_2$ is the identity $\phi = \id$,
$$
H^3(\pt//\Z_2; \Z_{\id})
\cong H^2_{\mathrm{group}}(\Z_2; U(1)_{\id}) \cong \Z_2.
$$
A group $2$-cocycle $\tau_{\id} \in Z^2_{\mathrm{group}}(\Z_2; U(1)_{\id})$ representing this non-trivial cohomology class is 
$$
\tau_{\id}((-1)^{m_1}, (-1)^{m_2}) = \exp\pi i m_1m_2.
$$
The values of this $2$-cocycle is as follows:
$$
\begin{array}{|c|c|c|}
\hline
\tau_{\id}(g, h) & h = 1 & h = -1 \\
\hline
g = 1 & 1 & 1 \\
\hline
g = -1 & 1 & -1 \\
\hline
\end{array}
$$
We remark that $\tau_{\id}$ is the unique cocycle that represents the non-trivial cohomology class and is subject to the normalization condition $\tau_{\id}(1, g) = \tau_{\id}(g, 1) = 1$ for all $g \in \Z_2$.

\item
In the case that $G = \Z_2 \times \Z_2$, $H = \Z_2 \times 1$ and $\phi : \Z_2 \times \Z_2 \to \Z_2$ is the first projection $\phi = p_1$, we have
$$
H^3((\Z_2 \times \Z_2)//(\Z_2 \times 1); \Z_{p_1})
\cong H^3(\pt//\Z_2; \Z_{\id})
\cong \Z_2.
$$
Thus, this non-trivial cohomology class is essentially represented by $\tau_{\id}$.

\item
In the case of $G = \Z_2 \times \Z_2$, $H = \Delta(\Z_2)$ and $\phi : \Z_2 \times \Z_2 \to \Z_2$ is the first projection $\phi = p_1$, we have
$$
H^3((\Z_2 \times \Z_2)//\Delta(\Z_2); \Z_{p_1})
\cong H^3(\pt//\Z_2; \Z_{\id})
\cong \Z_2.
$$
Thus, the non-trivial class is essentially represented by $\tau_{\id}$ also.

\item
In the case of $G = \Z_2 \times \Z_2$, $H = 1$ and $\phi : \Z_2 \times \Z_2 \to \Z_2$ is the first projection $\phi = p_1$, we have
$$
H^3(\pt//(\Z_2 \times \Z_2); \Z_{p_1})
\cong H^2_{\mathrm{group}}(\Z_2 \times \Z_2; U(1)_{p_1}) 
\cong \Z_2 \oplus \Z_2.
$$
This group is generated by the cocycles $\tau_{p_i} \in Z^2_{\mathrm{group}}(\Z_2 \times \Z_2; U(1)_{p_1})$ associated to the $i$th projection $p_i : \Z_2 \times \Z_2 \to \Z_2$
\begin{align*}
\tau_{p_1}(((-1)^{m_1}, (-1)^{n_1}), ((-1)^{m_2}, (-1)^{n_2}) 
&= \exp \pi i m_1m_2, \\
\tau_{p_2}(((-1)^{m_1}, (-1)^{n_1}), ((-1)^{m_2}, (-1)^{n_2}) 
&= \exp \pi i n_1n_2.
\end{align*}
The cocycle $\tau_{p_1}$ agrees with the pull-back $p_1^*\tau_{\id}$, and $\tau_{p_2}$ with the cocycle $\mu$ introduced in Definition \ref{dfn:cocycle_fredholm_vs_karoubi}.
\end{itemize}

We notice that, for $\beta \in C^1_{\mathrm{group}}(\Z_2 \times \Z_2; U(1)_{p_1})$ such that $\beta(1) = 1$, its coboundary $\partial \beta$ takes the following values.
$$
\begin{array}{|c|c|c|c|c|}
\hline
\partial \beta (g, h) & h = (1, 1) & (-1, 1) & (1, -1) & (-1, -1) \\
\hline
g = (1, 1) & 1 & 1 & 1 & 1 \\
\hline
g = (-1, 1) & 1 & 1 & Y^{-1} & Y^{-1} \\
\hline
g = (1, -1) & 1 & X & XY & Y \\
\hline
g = (-1, -1) & 1 & X^{-1} & X^{-1} & 1 \\
\hline
\end{array}
$$ 
Here $X$ and $Y$ are given by
\begin{align*}
X &= \beta(-1, 1)\beta(-1, -1)^{-1}\beta(1, -1), &
Y &= \beta(-1, -1)\beta(-1, 1)^{-1}\beta(1, -1).
\end{align*}
Using this fact, we can verify that $\tau_{p_1}$ and $\tau_{p_2}$ generate $H^2_{\mathrm{group}}(\Z_2 \times \Z_2); U(1)_{p_1})$. Furthermore, we can also verify that 
\begin{align*}
\tilde{\beta}(c_{p_1} \cup c_{p_1}) &= \tau_{p_1}, &
\tilde{\beta}(c_{p_1} \cup c_{p_2}) &= \tilde{\beta}(c_{p_2} \cup c_{p_1})
= \tau_{p_2}, &
\tilde{\beta}(c_{p_2} \cup c_{p_2}) &= \tau_{p_2}, 
\end{align*}
where $c_{p_i} = p_i$ forms a basis of $H^1(\pt//(\Z_2 \times \Z_2; \Z_2) \cong \mathrm{Hom}(\Z_2 \times \Z_2, \Z_2)$ and $\tilde{\beta}$ is the Bockstein homomorphism (recall \S\S\ref{subsec:twist}.) This result shows that the product of $p_1$-twists makes the set 
$$
\pi_0({}^{p_1}\Twist(\pt//(\Z_2 \times \Z_2)))
\cong H^3(\pt//(\Z_2 \times \Z_2); \Z_{p_1})
\times H^1(\pt//(\Z_2 \times \Z_2); \Z_2)
$$
into the abelian group $\Z_4 \times \Z_4$.


\section{Mackey decomposition and the periodicity on a point}
\label{sec:mackey_decomposition}

This appendix contains the argument needed to complete the proof of Lemma \ref{lem:locally_unversal_bundle}, Lemma \ref{lem:extend_Clifford_action}, Theorem \ref{thm:Atiyah_Singer_map} and Theorem \ref{thm:Fredholm_vs_Karoubi}. The argument is to reduce the problem of showing certain properties on the quotient groupoid $\pt//G$, with $G$ a compact Lie group, to one in the case with $G$ trivial. The reduction is based on a categorical lift of the so-called Mackey decomposition considered in \cite{F-M} (Theorem 9.8). We then prove the properties on the point, describing some details in the application of results in \cite{A-Se,A-Si}.

\subsection{Mackey decomposition}

Let $K$ be a compact Lie group. We denote by $\widehat{K} = \{ \lambda \}$ the complete set of (labels of) finite-dimensional irreducible unitary representations of $K$. Since $K$ is compact, $\widehat{K}$ is a discrete set and its cardinality is at most countable. For each $\lambda \in \widehat{K}$, we choose and fix its realization $(V_\lambda, \rho_\lambda)$, where $V_\lambda$ is a Hermitian vector space of finite rank and $\rho_\lambda : K \to U(V_\lambda)$ is a homomorphism. 

Suppose that $K$ is a closed normal subgroup of a compact Lie group $G$. Then, for any $g \in G$, we define a representation $({}^gV_\lambda, {}^g\rho_\lambda)$ of $K$ by setting ${}^gV_\lambda = V_\lambda$ and ${}^g\rho_\lambda(k) = \rho_\lambda(g^{-1}kg)$. Since $({}^gV_\lambda, {}^g\rho_\lambda)$ is irreducible, there uniquely exists an element $g\lambda \in \widehat{K}$ such that $({}^gV_\lambda, {}^g\rho_\lambda)$ is equivalent to $(V_{g\lambda}, \rho_{g\lambda})$. Then the assignment $\lambda \mapsto g\lambda$ defines a left action of $G$ on $\widehat{K}$ which descends to an action of $G/K$. 

Given $(V_\lambda, \rho_\lambda)$, its complex conjugation $(\overline{V_\lambda}, \overline{\rho_\lambda})$ is also an irreducible representation. We write $\overline{\lambda} \in \widehat{K}$ for the corresponding label. The assignment $\lambda \mapsto \overline{\lambda}$ defines a $\Z_2$-action on $\widehat{K}$ commuting with the action of $G$, so that we have an action of $G \times \Z_2$ on $\widehat{K}$.

\begin{lem} \label{lem:irrep_bundle}
Let $G$ be a compact Lie group, and $K \subset G$ a closed normal subgroup such that $G/K$ is finite. We write $p_{\Z_2} : G \times \Z_2 \to \Z_2$ for the projection, and $\pi : G \to G/K$ for the quotient. Then the Hermitian vector bundle $V \to \widehat{K}$ given by
$$
V = \bigcup_{\lambda \in \widehat{K}} V_\lambda.
$$
can be made into a $(p_{\Z_2}, (\pi, \id)^*\tau_{G/K \times \Z_2})$-twisted vector bundle on $\widehat{K}//(G \times \Z_2)$ such that: 
\begin{itemize}
\item
The subgroup $K \subset G \times \Z_2$ acts on the fiber of $\lambda \in \widehat{K}$ by the representation $\rho_\lambda : K \to U(V_\lambda)$ of $K$.

\item
$\tau_{G/K \times \Z_2} \in Z^2_{\mathrm{group}}(G/K \times \Z_2; C(\widehat{K}, U(1))_{p_{\Z_2}})$ is a group $2$-cocycle, where $p_{\Z_2} : G/K \times \Z_2 \to \Z_2$ is also the projection.

\end{itemize}
\end{lem}

We remark that the $\Z_2$-grading of $V$ is assumed to be trivial, so that the even part $V^0 = V$ is $V$ and the odd part $V^1 = 0$ is trivial.

\begin{proof}
We choose representatives $g_i$ of the coset $G/K$ as well as unitary equivalences of $K$-modules $\alpha(g_i; \lambda) : {}^{g_i}V_\lambda \to V_{g_i\lambda}$ for all $\lambda \in \widehat{K}$. For any $g \in G$, we have the unique decomposition $g = g_ik$ for an $i$. Using this decomposition, we define $\rho(g; \lambda) : V_\lambda \to V_{g\lambda}$ to be the composition of 
$$
V_\lambda \overset{\rho_\lambda(k)}{\to}
V_\lambda = V_{k\lambda} = {}^{g_i}V_{k\lambda} 
\overset{\alpha(g_i; k\lambda)}{\to}
V_{g_i k\lambda} = V_{g\lambda}.
$$
We also choose unitary equivalences of $K$-modules $\beta(\lambda) : \overline{V_\lambda} \to V_{\overline{\lambda}}$ for all $\lambda \in \widehat{K}$ and define $\rho(-1; \lambda) : V_\lambda \to V_{\overline{\lambda}}$ to be the composition of
$$
V_\lambda \to \overline{V_\lambda} 
\overset{\beta(\lambda)}{\to}
V_{\overline{\lambda}},
$$
where the first map is $v \mapsto \overline{v}$. The maps $\rho(g; \lambda)$ and $\rho(-1; \lambda)$ generate an action of $G \times \Z_2$ up to $U(1)$-phases, since each $V_\lambda$ is irreducible. The $U(1)$-phase factor yields a group $2$-cocycle $\tau_{G \times \Z_2} \in Z^2_{\mathrm{group}}(G \times \Z_2; C(\widehat{K}, U(1))_{p_{\Z_2}})$, and $V$ is a $(p_{\Z_2}, \tau_{G \times \Z_2})$-twisted vector bundle on $\widehat{K}//(G \times \Z_2)$. By construction, it holds that
$$
\tau_{G \times \Z_2}(g_1k_1, g_2k_2; k_3\lambda)
= \tau_{G \times \Z_2}(g_1, g_2; \lambda)
$$
for all $k_1, k_2, k_3 \in K$, $g_1, g_2 \in G$ and $\lambda \in \widehat{K}$. This means $\tau_{G \times \Z_2} = (\pi, \id)^*\tau_{G/K \times \Z_2}$ for a cocycle $\tau_{G/K \times \Z_2}$, and the lemma is proved.
\end{proof}

\begin{thm} \label{thm:mackey_decomposition}
Let $G$ be a compact Lie group, $K \subset G$ a closed normal subgroup such that $G/K$ is finite, and $\pi : G \to G/K$ the projection. For homomorphisms $\overline{\phi} : G /K \to \Z_2$ and $\overline{c} : G/K \to \Z_2$, there is an equivalence of categories
$$
\Phi : {}^{\pi^*\overline{\phi}}\Vect^{\pi^*\overline{c} + (p, q)}(\pt//G) \to
{}^{\overline{\phi}}\Vect^{(-\tau_{\widehat{K}}, \overline{c}) + (p, q)}(\widehat{K}//(G/K)),
$$
where $\tau_{\widehat{K}} = (\id, \overline{\phi})^*\tau_{G/K \times \Z_2} \in Z^2_{\mathrm{group}}(G/K; C(\widehat{K}, U(1))_{\overline{\phi}})$.
\end{thm}

\begin{proof}
By means of the embedding $(\id, \pi \circ \overline{\phi}) : G \to G \times \Z_2$, the vector bundle $V \to \widehat{K}$ in Lemma \ref{lem:irrep_bundle} gives rise to a $(\pi \circ \overline{\phi}, \tau_G, \pi \circ \overline{c})$-twisted vector bundle on $\widehat{K}//G$, where the $2$-cocycle $\tau_G$ is given by the pull-back
$$
\tau_G = (\id, \pi \circ \overline{\phi})^*(\pi, \id)^*\tau_{G/K \times \Z_2}
= \pi^*(\id, \overline{\phi})^*\tau_{G/K \times \Z_2}
= \pi^*\tau_{\widehat{K}}.
$$
For a $(\pi^*\overline{\phi}, \pi^*\overline{c})$-twisted vector bundle $(E, \epsilon, \rho_E, \gamma)$ on $\pt//G$ with $Cl_{p, q}$-action, we define a vector bundle $\hat{E} \to \widehat{K}$ by
$$
\hat{E} = \bigcup_{\lambda \in \widehat{K}}
\mathrm{Hom}_K(V_\lambda, E),
$$
where $\mathrm{Hom}_K(V_\lambda, E)$ is the space of complex linear maps $f : V_\lambda \to E$ commuting with the actions of $K \subset G$ on $V_\lambda$ and $E$. Note that $\mathrm{Hom}(V_\lambda, E) = V_\lambda^* \otimes E$ is a Hilbert space, and so is its subspace $\mathrm{Hom}_K(V_\lambda, E)$. The topology of $\hat{E}$ is given by this Hilbert space structure (rather than the compact open topology). By the $\Z_2$-grading $\epsilon$ on $E$ (and the trivial $\Z_2$-grading on $V$), the vector bundle $\hat{E}$ has the $\Z_2$-grading $\hat{\epsilon}(f) = \epsilon \circ f$. For $\overline{g} \in G/K$, we choose $g \in \pi^{-1}(\overline{g})$ and put $\hat{\rho}(\overline{g})(f) = \rho_E(g) \circ f \circ \rho_V(g)^{-1}$. This turns out to be independent of the choice of $g$, and $\hat{E}$ gives rise to a $(\overline{\phi}, -\tau_{\widehat{K}}, \overline{c})$-twisted vector bundle on $\widehat{K}//(G/K)$, which has a $Cl_{p, q}$-action defined by the composition with $\gamma$. The assignment $E \mapsto \hat{E}$ defines the functor $\Phi$, where $\Phi : \mathrm{Hom}(E, E') \to \mathrm{Hom}(\hat{E}, \hat{E'})$ is defined by the composition of homomorphisms. To complete the proof, we construct a functor in the opposite direction
$$
\Psi : {}^{\overline{\phi}}
\Vect^{(-\tau_{\widehat{K}}, \overline{c}) + (p, q)}(\widehat{K}//(G/K))
\to
{}^{\pi^*\overline{\phi}}\Vect^{\pi^*\overline{c} + (p, q)}(\pt//G).
$$
To construct $\Psi$, let $F$ be a $(\overline{\phi}, -\tau_{\widehat{K}}, \overline{c})$-twisted vector bundle on $\widehat{K}//(G/K)$ with $Cl_{p, q}$-action. We then define a vector space $\check{F}$ by
$$
\check{F} = \widehat{\bigoplus}_{\lambda \in \widehat{K}}
V_\lambda \otimes F|_\lambda,
$$
where $\widehat{\oplus}$ means the $L^2$-completion of the algebraic direct sum $\oplus$. The vector bundle $V$ has a twisted $G$-action, and $F$ also has a twisted $G$-action induced from $\pi : G \to G/K$. With these twisted $G$-actions, $\check{F}$ gives rise to a $(\overline{\phi} \circ \pi, \overline{c} \circ \pi)$-twisted vector bundle on $\pt//G$, which inherits a $Cl_{p, q}$-action from $F$. The functor $\Psi$ is induced from the assignment $F \mapsto \check{F}$. For objects
\begin{align*}
E &\in 
{}^{\pi^*\overline{\phi}}\Vect^{\pi^*\overline{c} + (p, q)}(\pt//G), &
F &\in
{}^{\overline{\phi}}
\Vect^{(-\tau_{\widehat{K}}, \overline{c}) + (p, q)}(\widehat{K}//(G/K)),
\end{align*}
we can see the maps
\begin{align*}
&\Psi\Phi(E) = 
\widehat{\bigoplus}_{\lambda \in \widehat{K}} 
V_\lambda \otimes \mathrm{Hom}_K(V_\lambda, E)
\to
E, &
&\widehat{\oplus}_\lambda v_\lambda \otimes f_\lambda
\mapsto
\widehat{\oplus}_\lambda f_\lambda(v_\lambda), \\
&F \to 
\Phi\Psi(F) =
\bigcup_{\lambda \in \widehat{K}}
\mathrm{Hom}_K(V_\lambda, V_\lambda) \otimes F|_\lambda,
&
&f_\lambda \mapsto \id_{V_\lambda} \otimes f_\lambda,
\end{align*}
provide the natural equivalences of functors $\Psi \Phi \Rightarrow \id$ and $\id \Rightarrow \Phi \Psi$, which proves that $\Phi$ is an equivalence of categories. 
\end{proof}

We now apply the theorem above to some concrete cases. In the following, we use $\simeq$ to mean an equivalence of categories. We also use the notation $\Vect_{\C}^{(p, q)} = \Vect^{(p, q)}(\pt)$ for the category of $\Z_2$-graded complex modules over $Cl_{p, q}$, and $\Vect_{\R}^{(p, q)}$ for the category of $\Z_2$-graded real modules over $Cl_{p, q}$. As in the body of this paper, infinite-dimensional modules are allowed, and the vector spaces underlying infinite dimensional modules are separable Hilbert spaces. We notice that there is an equivalence of categories 
$$
{}^{\id}\Vect^{(p, q)}(\pt//\Z_2) \simeq \Vect^{(p, q)}_{\R},
$$
since an $\id$-twisted vector bundle $E$ on $\pt//\Z_2$ is nothing but a complex vector space with a real structure (i.e.\ an anti-unitary involution) $T : E \to E$. Thus, the $T$-invariant part $E^T$ is a real vector space, and $E \mapsto E^T$ defines the equivalence of categories.

\begin{lem} \label{lem:mackey_decomposition_complex_case}
Let $G$ be a compact Lie group. 
\begin{itemize}
\item[(a)]
There is an equivalence of categories
$$
\Vect^{(p, q)}(\pt//G) \simeq \Vect^{(p, q)}(\widehat{G})
= \prod_{\lambda \in \widehat{G}} \Vect_{\C}^{(p, q)}.
$$

\item[(b)]
Let $c : G \to \Z_2$ be a non-trivial homomorphism, and $K = \mathrm{Ker}(c)$ the kernel of $c$. Then there is an equivalence of categories
$$
\Vect^{c + (p, q)}(\pt//G) \simeq 
\Vect^{\id_{\Z_2} + (p, q)}(\widehat{K}//\Z_2),
$$
and the category $\Vect^{\id_{\Z_2} + (p, q)}(\widehat{K}//\Z_2)$ is equivalent to the product of some copies of the following categories
\begin{align*}
&\Vect_{\C}^{(p, q)}, &
&\Vect_{\C}^{(p, q+1)}.
\end{align*}
\end{itemize}
\end{lem}

\begin{proof}
For (a), the equivalence of categories just follows from a direct application of Theorem \ref{thm:mackey_decomposition}. For (b), Theorem \ref{thm:mackey_decomposition} provides the equivalence of categories
$$
\Vect^{c + (p, q)}(\pt//G) \simeq 
\Vect^{(-\tau_{\widehat{K}}, \id_{\Z_2}) + (p, q)}(\widehat{K}//\Z_2).
$$
Since $\widehat{K}$ is a discrete set, it is a disjoint union of $\Z_2$-spaces of the form $\Z_2/H$, where $H \subset \Z_2$ is a subgroup. On the groupoid $(\Z_2/H)//\Z_2$, all the twists are trivial, as seen in \S\S\ref{subsec:classification_of_twists}. Hence a trivialization of the twist $-\tau_{\widehat{K}}$ leads to the equivalence of categories in (b). To show the remaining claim, we focus on the $\Z_2$-orbits $(\Z_2/H)//\Z_2$ of $\widehat{K}//\Z_2$. If $H = 1$, then we have the equivalence of categories
$$
\Vect^{\id_{\Z_2} + (p, q)}(\Z_2//\Z_2)
\simeq
\Vect^{(p, q)}(\pt) = \Vect_{\C}^{(p, q)}
$$
in view of the local equivalence $\pt//1 \to \Z_2//\Z_2$. If $H = \Z_2$, then we have
$$
\Vect^{\id_{\Z_2} + (p, q)}(\pt//\Z_2)
\simeq
\Vect_{\C}^{(p, q+1)},
$$
since an $\id_{\Z_2}$-twisted $\Z_2$-action can be regarded as an additional $Cl_{0, 1}$-action.
\end{proof}

\begin{lem} \label{lem:mackey_decomposition_Z2_case}
Let $G$ be a compact Lie group, $\phi : G \to \Z_2$ a non-trivial homomorphism, and $K = \mathrm{Ker}(\phi)$ the kernel of $\phi$.
\begin{itemize}
\item[(a)]
There is an equivalence of categories
$$
{}^{\phi}\Vect^{(p, q)}(\pt//G) \simeq 
{}^{\id_{\Z_2}}\Vect^{-\tau_{\widehat{K}} + (p, q)}(\widehat{K}//\Z_2),
$$
and the category ${}^{\id_{\Z_2}}\Vect^{-\tau_{\widehat{K}} + (p, q)}(\widehat{K}//\Z_2)$ is equivalent to the product of some copies of the following categories
\begin{align*}
&\Vect^{(p, q)}_{\C},
&\Vect^{(p, q)}_{\R}, &
&\Vect^{(p+4, q)}_{\R}.
\end{align*}

\item[(b)]
If $c = \phi : G \to \Z/2$, then there is an equivalence of categories
$$
{}^{\phi}\Vect^{c + (p, q)}(\pt//G) \simeq 
{}^{\id_{\Z_2}}\Vect^{(-\tau_{\widehat{K}}, \id_{\Z_2}) + (p, q)}
(\widehat{K}//\Z_2),
$$
and the category $\Vect^{(-\tau_{\widehat{K}}, \id_{\Z_2}) + (p, q)}(\widehat{K}//\Z_2)$ is equivalent to the product of some copies of the following categories
\begin{align*}
&\Vect^{(p, q)}_{\C},
&\Vect^{(p+2, q)}_{\R}, &
&\Vect^{(p, q+2)}_{\R}.
\end{align*}
\end{itemize}
\end{lem}

\begin{proof}
For (a), Theorem \ref{thm:mackey_decomposition} gives the equivalence of categories. The $\Z_2$-space $\widehat{K}$ is a disjoint union of $\Z_2/H$, where $H \subset \Z_2$ is a subgroup. By the classification of twists in \S\S\ref{subsec:classification_of_twists}, the category ${}^{\id_{\Z_2}}\Vect^{-\tau_{\widehat{K}} + (p, q)}(\widehat{K}//\Z_2)$ is the product of some copies of the following categories
\begin{align*}
&{}^{\id_{\Z_2}}\Vect^{(p, q)}(\Z_2//\Z_2), &
&{}^{\id_{\Z_2}}\Vect^{(p, q)}(\pt//\Z_2), &
&{}^{\id_{\Z_2}}\Vect^{\tau_{\id} + (p, q)}(\pt//\Z_2), 
\end{align*}
where $\tau_{\id}$ represents the non-trivial twist in $H^3(\pt//\Z_2; \Z_{\id})$. The local equivalence $\pt//1 \to \Z_2//\Z_2$ induces the equivalence of categories
$$
{}^{\id_{\Z_2}}\Vect^{(p, q)}(\Z_2//\Z_2) \simeq
\Vect^{(p, q)}_{\C}.
$$
As is pointed out already, if $E \in {}^{\id_{\Z_2}}\Vect^{(p, q)}(\pt//\Z_2)$ is a twisted vector bundle with $Cl_{p, q}$-action, then the twisted $\Z_2$-action on $E$ provides a real structure $T : E \to E$ commuting with the $Cl_{p, q}$-action. Thus, $E \mapsto E^T$ provides the equivalence
$$
{}^{\id_{\Z_2}}\Vect^{(p, q)}(\pt//\Z_2) \simeq
\Vect_{\R}^{(p, q)}.
$$
If $E \in {}^{\id_{\Z_2}}\Vect^{\tau_{\id} + (p, q)}(\pt//\Z_2)$ is a twisted vector bundle with $Cl_{p, q}$-action, then the twisted $\Z_2$-action on $E$ provides a quaternionic structure (i.e.\ anti-unitary map whose square is $-1$) $T : E \to E$ commuting with the $Cl_{p, q}$-action. As is known (for example Proposition B.4, \cite{F-M}), the category of vector spaces over the skew field $\mathbb{H}$ of quaternions are in one to one correspondence with that of $Cl_{4,0}$-modules. With the $Cl_{p, q}$-actions, $E$ induces a real $Cl_{p+4,q}$-modules, and we get the equivalence
$$
{}^{\id_{\Z_2}}\Vect^{\tau_{\id} + (p, q)}(\pt//\Z_2) \simeq
\Vect_{\R}^{(p+4, q)}. 
$$

For (b), the same argument as in (a) proves that ${}^{\id_{\Z_2}}\Vect^{(-\tau_{\widehat{K}}, \id_{\Z_2}) + (p, q)}(\widehat{K}//\Z_2)$ is the product of some copies of the following categories
\begin{align*}
&{}^{\id}\Vect^{\id + (p, q)}(\Z_2//\Z_2), &
&{}^{\id}\Vect^{\id + (p, q)}(\pt//\Z_2), &
&{}^{\id}\Vect^{(\tau_{\id}, \id) + (p, q)}(\pt//\Z_2),
\end{align*}
where $\id : \Z_2 \to \Z_2$ is the identity homomorphism. The local equivalence $\pt//1 \to \Z_2//\Z_2$ induces the equivalence of categories
$$
{}^{\id}\Vect^{\id + (p, q)}(\Z_2//\Z_2) \simeq
\Vect^{(p, q)}(\pt) = \Vect^{(p, q)}_{\C}.
$$
If $E \in {}^{\id}\Vect^{\id + (p, q)}(\pt//\Z_2)$ is a twisted vector bundle, then the twisted $\Z_2$-action on $E$ induces an odd real structure $T : E \to E$. On real vector space $E_{\R}$ underlying $E$, we have an additional $Cl_{0,2}$-action generated by $T$ and $iT$. Consequently, $E_{\R}$ is a real $Cl_{p,q+2}$-module, and this construction leads to the equivalence 
$$
{}^{\id}\Vect^{\id + (p, q)}(\pt//\Z_2) \simeq
\Vect_{\R}^{(p, q+2)}.
$$
Similarly, if $E \in {}^{\id}\Vect^{(\tau_{\id}, \id) + (p, q)}(\pt//\Z_2)$ is a twisted vector bundle, then the twisted $\Z_2$-action induces an odd quaternionic structure $T : E \to E$. The real vector space $E_{\R}$ acquires an additional $Cl_{2, 0}$-action generated by $T$ and $iT$. Hence we get the equivalence of categories
$$
{}^{\id}\Vect^{(\tau_{\id}, \id) + (p, q)}(\pt//\Z_2) \simeq
\Vect_{\R}^{(p+2, q)}
$$
induced by the assignment $E \mapsto E_{\R}$.
\end{proof}

\begin{lem} \label{lem:mackey_decomposition_D2_case}
Let $G'$ be a compact Lie group. Suppose that $\phi : G' \to \Z_2$ and $c : G' \to \Z_2$ are non-trivial homomorphisms such that $\phi \neq c$. We write $K = \mathrm{Ker}(\phi, c)$ for the kernel of $(\phi, c) : G' \to \Z_2 \times \Z_2$, and $p_i : \Z_2 \times \Z_2 \to \Z_2$ the $i$th projection. Then there is an equivalence of categories
$$
{}^{\phi}\Vect^{c + (p, q)}(\pt//G') \simeq 
{}^{p_1}\Vect^{(-\tau_{\widehat{K}}, p_2) + (p, q)}
(\widehat{K}//(\Z_2 \times \Z_2)),
$$
and the category ${}^{p_1}\Vect^{(-\tau_{\widehat{K}}, p_2) + (p, q)}(\widehat{K}//(\Z_2 \times \Z_2))$ is equivalent to the product of some copies of the following categories
\begin{align*}
&\Vect^{(p, q)}_{\C}, &
&\Vect^{(p, q+1)}_{\C}, &
&\Vect^{(p, q)}_{\R}, &
&\Vect^{(p+4, q)}_{\R}, &
&\Vect^{(p, q+2)}_{\R}, \\
&\Vect^{(p+2, q)}_{\R}, &
&\Vect^{(p, q+1)}_{\R}, &
&\Vect^{(p+1, q)}_{\R}, &
&\Vect^{(p+4, q+1)}_{\R}, &
&\Vect^{(p+5, q)}_{\R}.
\end{align*}
\end{lem}

\begin{proof}
The equivalence of categories is given by Theorem \ref{thm:mackey_decomposition}. To suppress notations, we put $G = \Z_2 \times \Z_2$. The groupoid $\widehat{K}//G$ is the disjoint union of $(G/H)//G$, where $H \subset G$ is a subgroup. Hence ${}^{p_1}\Vect^{(-\tau_{\widehat{K}}, p_2) + (p, q)}(\widehat{K}//G)$ is the product of the categories of the form
$$
{}^{p_1}\Vect^{(\tau, p_2) + (p, q)}
((G/H)//G)
\simeq
{}^{p_1|_H}\Vect^{(\tau|_H, p_2|_H) + (p, q)}
(\pt//H),
$$
where $\tau$ is an ungraded twist. There are four subgroups $H$, as seen in \S\S\ref{subsec:classification_of_twists}. In the case of $H = 1$, the twist $\tau$ is trivial, and
$$
{}^{p_1}\Vect^{(\tau, p_2) + (p, q)}(G//G)
\simeq \Vect^{(p, q)}(\pt) 
\simeq \Vect^{(p, q)}_{\C}.
$$
In the case of $H = 1 \times \Z_2$, the twist $\tau$ can be trivialized, and we use the argument in the proof of Lemma \ref{lem:mackey_decomposition_complex_case} (b) to get the equivalence of categories
$$
{}^{p_1}\Vect^{(\tau, p_2) + (p, q)}
((G/(1 \times \Z_2))//G)
\simeq
\Vect^{\id + (p, q)}(\pt//\Z_2)
\simeq
\Vect^{(p, q+1)}_{\C}.
$$
In the case of $H = \Z_2 \times 1$, the twist $\tau|_H$ is isomorphic to the trivial twist or $\tau_{\id}$. Then, as in the proof of Lemma \ref{lem:mackey_decomposition_Z2_case} (a), we get the equivalence of categories
$$
{}^{p_1}\Vect^{(\tau, p_2) + (p, q)}
((G/(\Z_2 \times 1))//G)
\simeq
\left\{
\begin{array}{l}
{}^{\id}\Vect^{(p, q)}(\pt//\Z_2)
\simeq 
\Vect_{\R}^{(p, q)}, \\
{}^{\id}\Vect^{\tau_{\id} + (p, q)}(\pt//\Z_2)
\simeq
\Vect_{\R}^{(p+4, q)}.
\end{array}
\right.
$$
In the case of $H = \Delta(\Z_2)$, the twist $\tau|_H$ is again isomorphic to the trivial twist or $\tau_{\id}$. By the proof of Lemma \ref{lem:mackey_decomposition_Z2_case} (b), we get the equivalence of categories
$$
{}^{p_1}\Vect^{(\tau, p_2) + (p, q)}
((G/\Delta(\Z_2))//G)
\simeq
\left\{
\begin{array}{l}
{}^{\id}\Vect^{\id + (p, q)}(\pt//\Z_2)
\simeq
\Vect_{\R}^{(p, q+2)}, \\
{}^{\id}\Vect^{(\tau_{\id}, \id) + (p, q)}
(\pt//\Z_2)
\simeq
\Vect_{\R}^{(p+2, q)}.
\end{array}
\right.
$$
Finally, in the case of $H = \Z_2 \times \Z_2$, the twist $\tau|_H = \tau$ is isomorphic to $1$ (trivial twist), $\tau_{p_1}$, $\tau_{p_2}$ or $\tau_{p_1} + \tau_{p_2}$, where $\tau_{p_1}$ and $\tau_{p_2}$ are the $2$-cocycles given in \S\S\ref{subsec:group_cocycle}. Let $E \in {}^{p_1}\Vect^{(\tau, p_2) + (p, q)}(\pt//G)$ be a twisted vector bundle. We write $T = \rho(-1, 1)$ and $S = \rho(1, -1)$ for the twisted action of the generators $(-1, 1), (1, -1)$ of $G$. By construction, $T : E \to E$ is even and anti-unitary, whereas $S : E \to E$ is odd and unitary. If $\tau$ is one of the four twists above, then $S$ and $T$ are commutative. Now, in the case of $\tau = 1$, we have $T^2 = 1$ and $S^2 = 1$. Hence $T$ is a real structure on $E$, and $S$ gives an additional $Cl_{0, 1}$-action on the real vector space $E^T$. Together with the original $Cl_{p, q}$-action, $E^T$ is a real $Cl_{p, q+1}$-module, so that the assignment $E \mapsto E^T$ induces 
$$
{}^{p_1}\Vect^{(1, p_2) + (p, q)}(\pt//G)
\simeq
\Vect_{\R}^{(p, q+1)}.
$$
In the case of $\tau = \tau_{p_2}$, we have $T^2 = 1$ and $S^2 = -1$. Hence $S$ defines an additional $Cl_{1, 0}$-action. Thus, in the same way as above, we have the equivalence
$$
{}^{p_1}\Vect^{(\tau_{p_2}, p_2) + (p, q)}(\pt//G)
\simeq
\Vect_{\R}^{(p+1, q)}.
$$
In the case of $\tau = \tau_{p_1}$, we have $T^2 = -1$ and $S^2 = 1$. Hence $T$ defines a quaternionic structure on $E$, and $S$ an additional $Cl_{0, 1}$-action. Then the equivalence of the categories of quaternionic vector spaces and that of real $Cl_{4, 0}$-modules induces
$$
{}^{p_1}\Vect^{(\tau_{p_1}, p_2) + (p, q)}(\pt//G)
\simeq
\Vect_{\R}^{(p+4, q+1)}.
$$
In the case of $\tau = \tau_{p_1} + \tau_{p_2}$, we have $T^2 = -1$ and $S^2 = -1$. By the same consideration as above, we have the equivalence of categories
$$
{}^{p_1}\Vect^{(\tau_{p_1} + \tau_{p_2}, p_2) + (p, q)}(\pt//G)
\simeq
\Vect_{\R}^{(p+5, q)},
$$
which completes the proof.
\end{proof}

\subsection{The space of Fredholm operators}

This subsection summarizes some properties of the spaces of Fredholm operators as models of the classifying spaces of complex and real $K$-theories.

\begin{lem} \label{lem:universal_clifford_module}
There exists a universal $Cl_{p, q}$-module in $\Vect^{(p, q)}_k$.
\end{lem}

\begin{proof}
Let $E \in \Vect_k^{(p, q)}$ be a $\Z_2$-graded $Cl_{p, q}$-module over $k$ which contains all the inequivalent $\Z_2$-graded $Cl_{p, q}$-modules over $k$ infinitely many times. (Actually, we can take $E = (Cl_{p, q} \otimes k) \otimes \mathcal{E}$ with $\mathcal{E}$ a separable infinite-dimensional Hilbert space over $k$.) Then any $\Z_2$-graded $Cl_{p, q}$-module over $k$ can be embedded into $E$, and hence $E$ has the universality. 
\end{proof}

From now on, we assume that $E \in \Vect^{(p, q)}_k$ is a universal module. As in \S\S\ref{subsec:fredholm_family}, we denote by $\mathrm{End}(E)$ the space of bounded operators with the compact open topology, and by $\mathrm{K}(E)$ the space of compact operators with the operator norm topology. Changing slightly the notation in Definition \ref{dfn:fredholm_family}, we denote
$$
\Fred^{(p, q)}_k(E) =
\left\{
A \in \mathrm{End}(E) \bigg|\ 
\begin{array}{ll}
A^* = -A, \ 
A^2 + \id \in \mathrm{K}(E), \
\mathrm{Spec}(A) \subset [-i, i] \\
\mbox{degree $1$}, \ A\gamma_i = -\gamma_iA \ (i = 1, \ldots, p+q) 
\end{array}
\right\},
$$
where $\gamma_1, \ldots, \gamma_{p+q}$ are the Clifford actions of fixed vectors $e_1, \ldots, e_{p+q} \in \R^{p+q}$ forming an orthonormal basis. As before, $\Fred^{(p, q)}_k(E)$ is topologized by
\begin{align*}
&\Fred^{(p, q)}_k(E) \to \mathrm{End}(E) \times \mathrm{K}(E), &
&A \mapsto (A, A^2 + \id),
\end{align*}
where $\mathrm{End}(E)$ has the compact open topology and $\mathrm{K}(E)$ the operator norm topology. We sometimes omit $E$ to write $\Fred^{(p, q)}_k = \Fred^{(p, q)}_k(E)$. We define a subspace
$$
\Fred^{(p, q)}_k(E)^\dagger
= \{ A \in \Fred^{(p, q)}_k(E) |\ A^2 = - \id \}.
$$

\begin{lem}
$\Fred^{(p, q)}_k(E)^\dagger$ is non-empty for any universal $E$.
\end{lem}

\begin{proof}
The same construction as in the proof of Lemma \ref{lem:extend_Clifford_action} applies: Let $\Pi E$ be the $\Z_2$-graded $k$-vector space with the $\Z_2$-grading of $E$ reversed. Since $\Pi E$ is also universal, we have an isometry $E \cong E \oplus \Pi E$. As $\Z_2$-graded $k$-vector spaces, we also have $E \oplus \Pi E \cong E \otimes k^2$, where $k^2$ is the $\Z_2$-graded $k$-vector space whose degree $0$ and $1$ parts are $k$. The action of $e_i \in \R^{p+q}$ on $E \oplus \Pi E$ is then identified with $\gamma_i \otimes 1$. Note that $k^2$ can be a $\Z_2$-graded $Cl_{1, 0}$-module, and hence the $Cl_{p, q}$-action on $E \otimes k^2$ extends to a $Cl_{p+1, q}$-action. Now the additional $Cl_{1, 0}$-action provides $\gamma \in \Fred^{(p, q)}_k(E)^\dagger$. 
\end{proof}

\begin{lem} \label{lem:contractible_basic_case}
$\Fred^{(p, q)}_k(E)^\dagger$ is contractible for any universal $E$.
\end{lem}

\begin{proof}
First of all, we notice the identification
$$
\Fred^{(p, q)}_k(E)^\dagger =
\left\{
A \in \mathrm{End}(E) \bigg|\ 
\begin{array}{ll}
A^* = -A, \ 
A^2 + \id = 0, \
A \epsilon = - \epsilon A, \\
A\gamma_i = -\gamma_iA \ (i = 1, \ldots, p+q) 
\end{array}
\right\},
$$
where $\epsilon$ and $\gamma_i$ are the $\Z_2$-grading and the $Cl_{p, q}$-action on $E$, respectively. Thus, $\Fred^{(p, q)}_k(E)^\dagger$ is topologized by the inclusion $\Fred^{(p, q)}_k(E)^\dagger \to \mathrm{End}(E)$ and the compact open topology on $\mathrm{End}(E)$, which allows us to prove the present lemma as a generalization of Proposition A2.1 in \cite{A-Se}. Since $E$ is universal, we have an isometry $E \cong L^2([0, 1], E) = L^2([0, 1]) \otimes E$, where $L^2([0, 1]) = L^2([0, 1], k)$ is the space of $k$-valued $L^2$-functions on the interval $[0, 1]$. For $t \in [0, 1]$, we let $R_t$ and $i_t$ be
\begin{align*}
R_t &: L^2([0, 1], E) \to L^2([0, t], E), &
&\mbox{restriction}, \\
i_t &: L^2([0, t], E) \to L^2([0, 1], E), &
&\mbox{inclusion}. 
\end{align*}
By construction, the composition $P_t = i_tR_t$ is the orthogonal projection onto $L^2([0, t], E) \subset L^2([0, 1], E)$, and $i_tR_t$ is the identity of $L^2([0, t], E)$. For $t \in (0, 1]$, we also let $Q_t$ be the isometric isomorphism given by
\begin{align*}
Q_t &: L^2([0, t], E) \to L^2([0, 1], E), &
(Q_tf)(x) &= t^{\frac{1}{2}}f(tx).
\end{align*}
As a base point $\gamma_* \in \Fred^{(p, q)}_k(L^2([0, 1], E))^\dagger$, we choose $\gamma_* = 1 \otimes \gamma$, where $\gamma \in \Fred^{(p, q)}_k(E)^\dagger$. We now define $h_t : \mathrm{End}(L^2([0, 1], E)) \to \mathrm{End}(L^2([0, 1], E))$ by
$$
h_t(A) = i_t Q_t^{-1}AQ_tR_t + (1 - P_t)\gamma_*
$$
for $t \in (0, 1]$, and $h_0(A) = \gamma_*$. As in \cite{A-Se}, we can see the continuity of
\begin{align*}
h &: \
\mathrm{End}(L^2([0, 1], E)) \times [0, 1] \to \mathrm{End}(L^2([0, 1], E)), 
&
(A, t) \mapsto h_t(A).
\end{align*}
We can also see $h_t(A) \in \Fred^{(p, q)}_k(L^2([0, 1], E))^\dagger$ for any $A \in \Fred^{(p, q)}_k(L^2([0, 1], E))^\dagger$ and $t \in [0, 1]$. Therefore $h$ contracts $\Fred^{(p, q)}_k(L^2([0, 1], E))^\dagger$ to the base point. 
\end{proof}

\begin{lem}[weak periodicity] \label{lem:weak_periodicity_most_basic}
In the real case $k = \R$, there are natural homeomorphisms
$$
\Fred^{(p, q)}_{\R} \cong
\Fred^{(p+1, q+1)}_{\R} \cong
\Fred^{(p+8, q)}_{\R} \cong
\Fred^{(p, q+8)}_{\R}.
$$
In the complex case $k = \C$, there are natural homeomorphisms
$$
\Fred^{(p, q)}_{\C} \cong
\Fred^{(p+1, q+1)}_{\C} \cong
\Fred^{(p+2, q)}_{\C} \cong
\Fred^{(p, q+2)}_{\C}.
$$
\end{lem}

\begin{proof}
The proof is essentially the same as Lemma \ref{lem:weak_periodicity}. In the real case, let $\Delta_{1,1} = \R^2$ be the $\Z_2$-graded $Cl_{1,1}$-module whose $\Z_2$-grading $\epsilon$ and $Cl_{1,1}$-action $\gamma_i$ are
\begin{align*}
\epsilon 
&=
\left(
\begin{array}{rr}
1 & 0 \\
0 & -1
\end{array}
\right),
&
\gamma_1
&=
\left(
\begin{array}{rr}
0 & -1 \\
1 & 0
\end{array}
\right),
&
\gamma_2
&=
\left(
\begin{array}{rr}
0 & 1 \\
1 & 0
\end{array}
\right).
\end{align*}
Since $\Delta_{1,1}$ is irreducible, the tensor product induces an equivalence of categories
$$
\cdot \otimes \Delta_{1,1} : \
\Vect^{(p, q)}_{\R} \to \Vect^{(p+1, q+1)}_{\R}.
$$
Thus, in particular, if $E \in \Vect_{\R}^{(p, q)}$ is universal, then so is $E \otimes \Delta_{1,1} \in \Vect_{\R}^{(p, q)}$. Now, the functor induces a continuous map
\begin{align*}
&\Fred^{(p, q)}_{\R}(E) \to \Fred^{(p+1, q+1)}_{\R}(E \otimes \Delta_{1,1}), &
&A \mapsto A \otimes 1.
\end{align*}
By a direct computation for example, we can verify that this map is bijective, and also a homeomorphism. The iteration of this homeomorphism gives
$$
\Fred^{(p, q)}_{\R}(E) \cong 
\Fred^{(p+4, q+4)}_{\R}(E \otimes \Delta_{1,1}^{\otimes 4}).
$$
In general, if $e_1, \cdots, e_4$ generate $Cl_{4, 0}$, then $e'_1, \cdots, e'_4$ generate $Cl_{0, 4}$, where $e'_i = e_i(e_1 \cdots e_4)$. As a result, we have the equivalences of categories
$$
\Vect_{\R}^{(p + 8, q)} \simeq
\Vect_{\R}^{(p + 4, q + 4)} \simeq
\Vect_{\R}^{(p, q+8)},
$$
and also homeomorphisms
$$
\Fred^{(p+8, q)}_{\R} \cong
\Fred^{(p+4, q + 4)}_{\R} \cong
\Fred^{(p, q+8)}_{\R}.
$$
In the complex case ($k = \C$), we consider $\Delta_{1,1}^{\C} = \Delta_{1,1} \otimes \C$, which is an irreducible $\Z_2$-graded complex module over $Cl_{1,1}$. As in the real case, we have a homeomorphism
\begin{align*}
&\Fred^{(p, q)}_{\C}(E) \to 
\Fred^{(p+1, q+1)}_{\C}(E \otimes \Delta^{\C}_{1,1}), &
&A \mapsto A \otimes 1.
\end{align*}
If $\gamma_j$ acts on a (universal) module $E \in \Vect_{\R}^{(p, q)}$ by $\gamma_i^2 = \pm 1$, then $i\gamma_j$ acts on $E$ by $(i\gamma_j)^2 = \mp 1$. Hence we have equivalence of categories
$$
\Vect_{\C}^{(p + 2, q)} \simeq
\Vect_{\C}^{(p + 1, q + 1)} \simeq
\Vect_{\C}^{(p, q + 2)},
$$
and also homeomorphisms
$$
\Fred^{(p + 2, q)}_{\C} \cong
\Fred^{(p + 1, q + 1)}_{\C} \cong
\Fred^{(p, q + 1)}_{\C},
$$
which completes the proof.
\end{proof}

For further analysis of $\Fred^{(p, q)}_k(E)$, it is useful to express this space in terms of an \textit{ungraded} Clifford module. Lemma \ref{lem:weak_periodicity_most_basic} allows us to set $q = 0$. For a universal $Cl_{p, 0}$-module $E$, we can assume that $E^0 = E^1 = \hat{E}$ is a separable infinite-dimensional Hilbert space. Then the $\Z_2$-grading $\epsilon$ on $E$ and the actions $\gamma_i$ of the Clifford algebra $Cl_{p, 0}$ are expressed as
\begin{align*}
\epsilon 
&=
\left(
\begin{array}{rr}
1 & 0 \\
0 & -1
\end{array}
\right),
&
\gamma_1
&=
\left(
\begin{array}{cc}
0 & -1 \\
1 & 0
\end{array}
\right),
&
\gamma_i
&=
\left(
\begin{array}{cc}
0 & \hat{\gamma}_i \\
\hat{\gamma}_i & 0
\end{array}
\right),
\ \
(i \ge 2)
\end{align*}
where the skew-adjoint maps $\hat{\gamma}_i : \hat{E} \to \hat{E}$, ($2 \le i \le p$) make $\hat{E}$ into an ungraded module over $Cl_{p-1, 0}$ for $p \ge 2$.

\begin{lem} \label{lem:bijection_graded_vs_ungraded}
We have the following bijections 
\begin{align*}
\Fred^{(0,0)}_k(E)
&\cong
\left\{
\hat{A} \in \mathrm{End}(\hat{E}) 
|\
\hat{A}^*\hat{A} - \id, \hat{A}\hat{A}^* - \id \in \mathrm{K}(\hat{E}), \
\lVert \hat{A} \rVert = 1
\right\}, \\
\Fred^{(1,0)}_k(E)
&\cong
\left\{
\hat{A} \in \mathrm{End}(\hat{E}) 
|\
\hat{A}^2 + \id \in \mathrm{K}(\hat{E}), \
\lVert \hat{A} \rVert = 1, \
\hat{A}^* = - \hat{A}
\right\}.
\end{align*}
If $p \ge 2$, then there is the following bijection
$$
\Fred^{(p,0)}_k(E)
\cong
\left\{
\hat{A} \in \mathrm{End}(\hat{E}) 
\bigg|
\begin{array}{l}
\hat{A}^2 + \id \in \mathrm{K}(\hat{E}), \
\lVert \hat{A} \rVert = 1, \
\hat{A}^* = - \hat{A}, \\
\hat{\gamma}_i \hat{A} = - \hat{A}\hat{\gamma}_i \ (i = 2, \ldots, p)
\end{array}
\right\}.
$$
\end{lem}

\begin{proof}
Any skew-adjoint bounded operator $A : E \to E$ of degree $1$ is expressed as
$$
A=
\left(
\begin{array}{cc}
0 & -\hat{A}^* \\
\hat{A} & 0
\end{array}
\right)
$$
by using a bounded operator $\hat{A} : \hat{E} \to \hat{E}$. The assignment $A \mapsto \hat{A}$ induces all the bijections stated in the lemma. 
\end{proof}

Let $\hat{E}$ be as before. We introduce
$$
\overline{\Omega}_0
= 
\left\{
T \in \mathrm{End}(\hat{E}) |\
T - \id \in \mathrm{K}(\hat{E}), \ T^*T = TT^* = \id
\right\}.
$$
For $p \ge 2$, we follow \cite{A-Si} (\S 4) to introduce
$$
\overline{\Omega}_{p-1}
=
\left\{
T \in \mathrm{End}(\hat{E}) \bigg|\
\begin{array}{l}
T - \hat{\gamma}_p \in \mathrm{K}(\hat{E}), \ T^*T = TT^* = \id, \\
T^2 = -\id, \
T\hat{\gamma}_i = - \hat{\gamma}_i T \ (i \le p-1)
\end{array}
\right\}.
$$
For $p \ge 1$, we topologize $\overline{\Omega}_{p-1}$ by the operator norm topology.

\begin{lem} \label{lem:the_key_to_periodicity}
There is a homotopy equivalence $\Fred^{(p, 0)}_k \simeq \overline{\Omega}_{p-1}$ for $p \ge 1$.
\end{lem}

\begin{proof}
The lemma can be shown by adapting the argument to prove Proposition (3.3) and Proposition (4.2) in \cite{A-Si}.

In the case of $p = 1$, we use the bijection in Lemma \ref{lem:bijection_graded_vs_ungraded} to introduce a map
\begin{align*}
\varpi_1 &: \ \Fred^{(1, 0)}_k(\hat{E}) \to 
\overline{\Omega}_0, 
&
\varpi_1(\hat{A}) &= - e^{\pi \hat{A}}.
\end{align*}
In the same way as in the proof of Lemma \ref{lem:key_continuity}, we can show that $\varpi_1$ above is well-defined and continuous. (In the case of $k = \R$, we consider the complexification $\hat{E} \otimes \C$ and its obvious real structure.) In view of the spectral decomposition of unitary operators, $\varpi_1$ is surjective. We would then like to apply Lemma (3.7) in \cite{A-Si} to $\varpi_1$. For this aim, it is enough to check that $\varpi_1 : \varpi_1^{-1}(D(n)) \to D(n)$ is a fiber bundle with contractible fiber, where
$$
D(n) = \{ T \in \overline{\Omega}_0 |\ \mathrm{rank}(\id - T) = n \}
$$
for $n \ge 0$, as defined in \cite{A-Si}. A consideration similar to the proof of Lemma (3.6) in \cite{A-Si} shows that $\varpi_1 : \varpi_1^{-1}(D(n)) \to D(n)$ is the fiber bundle associated to a Hilbert space subbundle $\{ \mathrm{Ker}(\id - T) \}_{T \in D(n)}$ of $D(n) \times \hat{E}$ whose fiber is identified with
$$
\{ \hat{A} \in \mathrm{End}(\hat{E}') |\
\hat{A}^2 = - \id, \lVert \hat{A} \rVert = 1, \hat{A}^* = - \hat{A}
\},
$$
where $\hat{E}' \subset \hat{E}$ is the orthogonal complement of a finite rank subspace of the form $\mathrm{Ker}(\id - T)$ with $T \in D(n)$. The space above is identified with $\Fred^{(1, 0)}_k(E')^\dagger \subset \Fred^{(1, 0)}_k(E')$ under Lemma \ref{lem:bijection_graded_vs_ungraded}, and is contractible by Lemma \ref{lem:contractible_basic_case}, since $E' = \hat{E}' \oplus \hat{E}'$ is a universal $Cl_{1, 0}$-module. As a result, $\varpi_1$ is a homotopy equivalence. 

In the case of $p \ge 2$, we consider
\begin{align*}
\varpi_p &: \ \Fred^{(p, 0)}_k(\hat{E}) \to 
\overline{\Omega}_{p-1}, 
&
\varpi_p(\hat{A}) &= - \hat{\gamma}_pe^{\pi \hat{A} \hat{\gamma}_p}.
\end{align*}
We can see directly that $\varpi_p$ is well-defined. By the spectral decomposition of unitary operators, $\varpi_p$ is surjective. We can also see $\varpi_p$ is continuous as in Lemma \ref{lem:key_continuity}. For $n \ge 0$, let $D(n)$ be
$$
D(n) = \{ T \in \overline{\Omega}_{p-1} |\ 
\mathrm{rank}(\id + \hat{\gamma}_p T) = n \}.
$$
As in the case of $p = 1$, the restriction $\varpi_p : \varpi_p^{-1}(D(n)) \to D(n)$ is a fiber bundle. Its fiber is identified with
$$
\left\{
\hat{A} \in \mathrm{End}(\hat{E}) 
\bigg|
\begin{array}{l}
\hat{A}^2 = - \id, \
\hat{A}^* = - \hat{A}, \\
\hat{\gamma}_i \hat{A} = - \hat{A}\hat{\gamma}_i \ (i = 2, \ldots, p)
\end{array}
\right\}
\cong
\Fred^{(p,0)}_k(E)^\dagger,
$$
which is contractible by Lemma \ref{lem:contractible_basic_case}. As a result, $\varpi_p$ is a homotopy equivalence. 
\end{proof}

\begin{lem} \label{lem:atiyah_singer_map_basic_case}
For $p \ge 1$ and $q \ge 0$, there is a homotopy equivalence 
$$
\mathrm{AS} : \ \Fred^{(p, q)}_k \to \Omega\Fred^{(p-1, q)}_k,
$$
where $\Omega \Fred^{(p-1, q)}_k$ is the space of continuous paths in $\Fred^{(p-1, q)}_k$ from $\gamma_p$ to $-\gamma_p$.
\end{lem}

\begin{proof}
The map $\mathrm{AS}$ is as given in \cite{A-Si}, up to the factor $\gamma_p$:
$$
\mathrm{AS}(A)(t) =
\gamma_p \cos \pi t + A \sin \pi t
= \gamma_p(\cos \pi t + A \gamma_p \sin \pi t)
= \gamma_p e^{\pi t A\gamma_p},
$$
which is continuous in the topology of $\Fred^{(p, q)}_k(E)$. This map is also compatible with the periodicity in Lemma \ref{lem:weak_periodicity_most_basic}. Thus, it suffices to consider the case of $q = 0$. To prove that $\mathrm{AS}$ is a homotopy equivalence, we tentatively define a space $F^p_k(E)$ to be $F^p_k(E) = \Fred^{(p, 0)}_k(E)$ as a set and topologize it by the operator norm topology. We then define a subspace $\mathbb{F}^p_k(E) \subset F^p_k(E)$, which is a model of the classifying space of $K$-theory \cite{A-Si}, as follows: For $A \in F^p_k(E)$, we put 
$$
w(A) = e_1 \cdots e_p A,
$$
and consider the restriction $w(A)|_{E^0}$, where $E^0$ is the degree $0$ part of a universal $Cl_{p, 0}$-module $E = E^0 \oplus E^1$. 
\begin{itemize}
\item
If $k = \C$ and $p = 1 \mod 4$, then $\mathbb{F}^p_k(E)$ consists of self-adjoint Fredholm operators $A$ such that $i^{-1}w(A)|_{E^0}$ are neither  essentially positive nor negative.

\item
If $k = \C$ and $p = 3 \mod 4$, then $\mathbb{F}^p_k(E)$ consists of self-adjoint Fredholm operators $A$ such that $w(A)|_{E^0}$ are neither  essentially positive nor negative.

\item
If $k = \R$ and $p = 3 \mod 4$, then $\mathbb{F}^p_k(E)$ consists of self-adjoint Fredholm operators $A$ such that $w(A)|_{E^0}$ are neither  essentially positive nor negative.

\item
Otherwise, $\mathbb{F}^p_k(E) = F^p_k(E)$.

\end{itemize}
A self-adjoint Fredholm operator is said to be essentially positive (resp.\ negative) if it is positive (resp.\ negative) on some invariant subspace of finite codimension. As shown in \cite{A-Si}(Proposition (3.3), Proposition (4.2)), if $p \ge 1$, then the space $\mathbb{F}^p_k = \mathbb{F}^p_k(E)$ is homotopy equivalent to the space $\overline{\Omega}_{p-1}$ considered in Lemma \ref{lem:the_key_to_periodicity}. The inclusion induces a continuous map $\mathbb{F}^p_k \to \Fred^{(p, 0)}_k$, and makes the following diagram commutative
$$
\begin{CD}
\mathbb{F}^p_k @>>> \Fred^{(p, 0)}_k \\
@VV{\simeq}V @V{\simeq}V{\varpi_p}V \\
\overline{\Omega}_{p-1} @= \overline{\Omega}_{p-1},
\end{CD}
$$
where the left and right vertical maps are the homotopy equivalences in \cite{A-Si}(Proposition (3.3), Proposition (4.2)) and Lemma \ref{lem:the_key_to_periodicity}, respectively. Consequently, the inclusion $\mathbb{F}^p_k \to \Fred^{(p, 0)}_k$ is a homotopy equivalence. Now, we also have the Atiyah-Singer map $\mathrm{AS} : \mathbb{F}^p_k \to \Omega \mathbb{F}^{p-1}_k$, which is a homotopy equivalence \cite{A-Si}(Lemma (2.6), (2.7), Proposition (2.8), Proposition (2.9), Proposition (4.2)). We clearly have the commutative diagram
$$
\begin{CD}
\mathbb{F}^p_k @>{\simeq}>> \Fred^{(p, 0)}_k \\
@V{\mathrm{AS}}V{\simeq}V @VV{\mathrm{AS}}V \\
\Omega\mathbb{F}^{p-1}_k @>{\simeq}>> \Omega\Fred^{(p-1, 0)}_k,
\end{CD}
$$
which implies that $\mathrm{AS} : \Fred^{(p, 0)}_k \to \Omega\Fred^{(p-1, 0)}_k$ is a homotopy equivalence.
\end{proof}

Slightly changing the notation in Definition \ref{dfn:gradation}, we introduce
$$
\Gr^{(p, q)}_k(E) 
= 
\left\{
\eta \in \mathrm{End}(E) 
\bigg|\
\begin{array}{l}
\eta^* = \eta, \ 
\eta^2 = \id, \
\eta - \epsilon \in \mathrm{K}(E) \\
\eta \gamma_i = - \gamma_i \eta \ \ (i = 1, \ldots, p+q)
\end{array}
\right\},
$$
where $E = (E, \epsilon)$ is a $\Z_2$-graded $k$-vector space which is a universal $Cl_{p, q}$-module. The set $\Gr^{(p, q)}_k(E)$ is topologized by the operator norm. Then, as shown in Lemma \ref{lem:key_continuity}, we have the continuous map
\begin{align*}
\vartheta &: \Fred^{(p, q)}_k(E) \to \Gr^{(q, p)}_k(\acute{E}), &
\vartheta(A) &= - e^{\pi A}\epsilon,
\end{align*}
where $\acute{E} = (E, \epsilon)$ as a $\Z_2$-graded $k$-vector space, whereas the $Cl_{q, p}$-action $\acute{\gamma}_i$ is defined as $\acute{\gamma}_i = \gamma_i\epsilon$ by using the original $Cl_{p, q}$-action $\gamma_i$ on $E$.

\begin{lem} \label{lem:Fredholm_vs_Karoubi_basic_case}
For any $p, q \ge 0$, $k = \R, \C$ and a universal $E$, the map 
$$
\vartheta : \Fred^{(p, q)}_k(E) \to \Gr^{(q, p)}_k(\acute{E})
$$
is a homotopy equivalence. 
\end{lem}

\begin{proof}
For $n \ge 0$, we define
\begin{align*}
C(n) &= \{ \eta \in \Gr^{(p, q)}_k(\acute{E}) |\ 
\mathrm{rank}(\id - \eta \epsilon) \le n \}, 
\\
D(n) &= \{ \eta \in \Gr^{(p, q)}_k(\acute{E}) |\ 
\mathrm{rank}(\id - \eta \epsilon) = n \}.
\end{align*}
In the same way as in Lemma \ref{lem:the_key_to_periodicity}, we can see that $\vartheta : \vartheta^{-1}(D(n)) \to D(n)$ is a fiber bundle. The fiber of this bundle is contractible by Lemma \ref{lem:contractible_basic_case}. We would then like to apply Lemma (3.7) in \cite{A-Si}. For its application, we need to see that $C(n-1) \subset C(n)$ has a respectable open neighbourhood $U$. This can be shown in the same way as in \cite{A-Si}, since $\Gr^{(q, p)}_k(\acute{E})$ is topologized by the operator norm. 
\end{proof}


\begin{rem}
Let $F^{p, q}_k$ denote the set $\Fred^{(p, q)}_k$ endowed with the operator norm topology. It is well-known \cite{A-Si} that $F^{p, q}_k$ admits contractible components when $p - q = 1 \mod 2$ for $k = \C$ and $p - q = 3 \mod 4$ for $k = \R$. Though is counter-intuitive from the viewpoint of the operator norm topology, Lemma \ref{lem:contractible_basic_case} implies that the space $\Fred^{(p, q)}_k$ is path connected. Hence we need not care about the ``contractible components'' in $\Fred^{(p, q)}_k$ to realize the classifying spaces of $K$-theories. 
\end{rem}

\subsection{Postponed proof}

We summarize here the proof postponed from the main text in the reduction argument.

\begin{lem} \label{lem:locally_universal_bundle_point_case}
Let $G$ be a compact Lie group. For any homomorphisms $\phi : G \to \Z_2$ and $c : G \to \Z_2$, there exists a (locally) universal $(\phi, c)$-twisted vector bundle on $\pt//G$ with $Cl_{p, q}$-action.
\end{lem}

\begin{proof}
By Lemma \ref{lem:mackey_decomposition_complex_case}, Lemma \ref{lem:mackey_decomposition_Z2_case} and Lemma \ref{lem:mackey_decomposition_D2_case}, the equivalence of categories in Theorem \ref{thm:mackey_decomposition} can be expressed as
$$
{}^\phi \Vect^{c + (p, q)}(\pt//G)
\simeq \prod_{\lambda \in \Lambda} 
\Vect_{k_\lambda}^{(p_{\lambda}, q_{\lambda})}
$$
where $k_\lambda = \R, \C$, and $\Lambda$ is a countable discrete set, since the set of inequivalent irreducible representations of a compact Lie group is at most countable. As in Lemma \ref{lem:universal_clifford_module}, we can realize a (locally) universal bundle $E^{(p_\lambda, q_\lambda)} \in \Vect^{(p_{\lambda}, q_{\lambda})}_{k_\lambda}$. Since $\Lambda$ is countable, the Hilbert space direct sum of $E^{(p_\lambda, q_\lambda)}$ is separable as well. The resulting Hilbert space produces a $(\phi, c)$-twisted universal bundle $E \in {}^\phi\Vect^{c + (p, q)}(\pt//G)$ through the equivalence of categories in Theorem \ref{thm:mackey_decomposition}, because this equivalence preserves the local universality.
\end{proof}

\begin{lem} \label{lem:contractible_point_case}
Let $G$ be a compact Lie group, $\phi : G \to \Z_2$ and $c : G \to \Z_2$ homomorphisms, and $E$ a $(\phi, c)$-twisted locally universal vector bundle on $\pt//G$ with $Cl_{p, q}$-action. Then $\Gamma(\X, \Fred(E)^\dagger)$ is contractible. 
\end{lem}

\begin{proof}
As in the proof of Lemma \ref{lem:locally_universal_bundle_point_case}, we have the equivalence of categories
$$
{}^\phi \Vect^{c + (p, q)}(\pt//G)
\simeq \prod_{\lambda \in \Lambda} 
\Vect_{k_\lambda}^{(p_{\lambda}, q_{\lambda})},
$$
which preserves the (local) universality of (twisted) vector bundles. The equivalence of categories above induces the identification
$$
\Gamma(\pt//G, \Fred(E))
\cong
{\prod}'_{\lambda \in \Lambda} 
\Fred^{(p_\lambda, q_\lambda)}_{k_\lambda}(E^{(p_\lambda, q_\lambda)})
\subset 
\prod_{\lambda \in \Lambda} 
\Fred^{(p_\lambda, q_\lambda)}_{k_\lambda}(E^{(p_\lambda, q_\lambda)}),
$$
where $\prod'_{\lambda \in \Lambda} \Fred^{(p_\lambda, q_\lambda)}_{k_\lambda}(E^{(p_\lambda, q_\lambda)})$ is the subspace consisting of Fredholm operators $(A_\lambda)_{\lambda \in \Lambda}$ such that $\sum_\lambda \dim\mathrm{Ker}A_\lambda < +\infty$. Note that $E^{(p_\lambda, q_\lambda)}$ are universal if $E \in {}^\phi \Vect^{c + (p, q)}(\pt//G)$ is locally universal. The identification restricts to give
$$
\Gamma(\pt//G, \Fred(E))^\dagger
\cong
\prod_{\lambda \in \Lambda} 
\Fred^{(p_\lambda, q_\lambda)}_{k_\lambda}(E^{(p_\lambda, q_\lambda)})^\dagger.
$$
Now the proof is completed by Lemma \ref{lem:contractible_basic_case}.
\end{proof}

\begin{lem} \label{lem:atiyah_singer_map_point_case}
Let $G$ be a compact Lie group, $\phi : G \to \Z_2$ and $c : G \to \Z_2$ homomorphisms, and $E$ a $(\phi, c)$-twisted locally universal vector bundle on $\pt//G$ with $Cl_{p, q}$-action. If $p > 0$, then the Atiyah-Singer map
$$
\mathrm{AS} :
\Gamma(\pt//G, \Fred(E)) \to 
\Gamma([0, 1]//G, \{ 0, 1 \}//G, \Fred(E \times [0, 1])),
$$
which is given by $\mathrm{AS}(A)(t) = \gamma_1 \cos\pi t + A \sin\pi t$, is a homotopy equivalence.
\end{lem}

\begin{proof}
As in the proof of Lemma \ref{lem:contractible_point_case}, we have the identification
$$
\Gamma(\pt//G, \Fred(E))
\cong
{\prod}'_{\lambda \in \Lambda} 
\Fred^{(p_\lambda, q_\lambda)}_{k_\lambda}(E^{(p_\lambda, q_\lambda)}).
$$
Note that $p > 0$ implies $p_\lambda > 0$. We can also identify
$$
\Gamma([0, 1]//G, \{ 0, 1 \}//G, \Fred(E \times [0, 1]))
\cong
{\prod}'_{\lambda \in \Lambda} 
\Gamma([0, 1], \{ 0, 1 \}, \Fred^{(p_\lambda-1, q_\lambda)}_{k_\lambda}).
$$
The Atiyah-Singer map is clearly compatible with these identifications. The space of invertible operators in $\Fred^{(p_\lambda-1, q_\lambda)}_{k_\lambda}$ is contractible, as a result of Lemma \ref{lem:contractible_basic_case}. Hence we get a homotopy equivalence
$$
\Gamma([0, 1], \{ 0, 1 \}, \Fred^{(p_\lambda-1, q_\lambda)}_{k_\lambda})
\simeq
\Omega \Fred^{(p_\lambda, q_\lambda)}_{k_\lambda}.
$$
Now, the lemma follows from Lemma \ref{lem:atiyah_singer_map_basic_case}.
\end{proof}

\begin{lem} \label{lem:Fredholm_vs_Karoubi_point_case}
Let $G$ be a compact Lie group, $\phi : G \to \Z_2$ and $c : G \to \Z_2$ homomorphisms, and $E$ a $(\phi, c)$-twisted locally universal vector bundle on $\pt//G$ with $Cl_{p, q}$-action. For any $p, q \ge 0$, the map
$$
\vartheta : \
\Gamma(\pt//G, \Fred(E)) \to 
\Gamma(\pt//G, \Gr(\acute{E}))
$$
is a homotopy equivalence. 
\end{lem}

\begin{proof}
Notice that $\acute{E}$ is a $(\phi, c)$-twisted locally universal vector bundle on $\pt//G$ with $Cl_{q, p}$-action. As in the proof of Lemma \ref{lem:contractible_point_case}, we have identifications
\begin{align*}
\Gamma(\pt//G, \Fred(E))
&\cong
{\prod}'_{\lambda \in \Lambda} 
\Fred^{(p_\lambda, q_\lambda)}_{k_\lambda}(E^{(p_\lambda, q_\lambda)}), \\
\Gamma(\pt//G, \Gr(\acute{E}))
&\cong
{\prod}'_{\lambda \in \Lambda} 
\Gr^{(q_\lambda, p_\lambda)}_{k_\lambda}(\acute{E}^{(p_\lambda, q_\lambda)}),
\end{align*}
where ${\prod}'_{\lambda \in \Lambda} \Gr^{(q_\lambda, p_\lambda)}_{k_\lambda}(\acute{E}^{(p_\lambda, q_\lambda)}) \subset \prod_{\lambda \in \Lambda} \Gr^{(q_\lambda, p_\lambda)}_{k_\lambda}(\acute{E}^{(p_\lambda, q_\lambda)})$ is the subspace corresponding to $\Gamma(\pt//G, \Gr(\acute{E}))$ under the decomposition of ${}^\phi \Vect^{c + (p, q)}(\pt//G)$. Under these identifications, the map $\vartheta$ in question is decomposed into the homotopy equivalences in Lemma \ref{lem:Fredholm_vs_Karoubi_basic_case}, and hence is a homotopy equivalence as well. 
\end{proof}


\section{Quotient of monoid}
\label{sec:quotient_monoid}

This appendix is about the construction of quotient monoid used in the finite-dimensional formulations of $K$-theories in Definition \ref{dfn:finite_rank_freed_moore_K} and Definition \ref{dfn:finite_rank_karoubi}. The construction may be standard, and can be found for example in \cite{W}.

\smallskip

Let $M = (M, +, 0)$ be an abelian monoid with zero (the additive unit), that is, a set $M$ equipped with a distributive and commutative binary operation $+ : M \times M \to M$ such that $x + 0 = 0 + x = x$ for any $x \in M$. Let $Z \subset M$ be a submonoid of $M$, that is, a subset $Z \subset M$ which is closed under the addition and contains $0 \in M$. Using $Z$, we can introduce an equivalence relation $\sim$ on $M$ by declaring $x \sim x'$ if and only if there are $z, z' \in Z$ such that $x + z = x' + z'$. We write the quotient set as $M/Z = M/\sim$. It is easy to see that $M/Z$ inherits an abelian monoid structure from $M$, in which zero is represented by elements in $Z$.

\begin{lem} \label{appendix:lem_quotient_monoid}
Let $M$ be an abelian monoid, and $Z \subset M$ its submonoid. Suppose that there is a monoid homomorphism $I : M \to M$ such that
\begin{itemize}
\item[(i)]
$I(Z) \subset Z$, 

\item[(ii)]
$I(x) + x \in Z$ for any $x \in M$,
\end{itemize}
Then the quotient monoid $M/Z$ gives rise to an abelian group.
\end{lem}

We remark that $I$ needs not be an involution on $M$.

\begin{proof}
It is enough to verify the existence of inverse elements. We denote by $[x] \in M/Z$ the element represented by $x \in M$. We define the inverse of $[x] \in M/Z$ to be $-[x] = [I(x)]$. This is well-defined. Actually, if $x \sim x'$, then there are $z, z' \in Z$ such that $x + z = x' + z'$, and we have
$$
I(x) + I(z) = I(x + z) = I(x' + z') = I(x') + I(z').
$$
Since $I(z), I(z') \in Z$ by (i), it holds that $[I(x)] = [I(x')]$. Because of (ii), we see that $-[x] + [x] = [I(x) + x] = 0$.
\end{proof}

As an example, we let $N$ be an abelian monoid, and consider the product monoid $M = N \times N$. The diagonal set $Z = \Delta(N) \subset N \times N$ is a submonoid. If we define a homomorphism $I : N \times N \to N \times N$ by $I(x, y) = (y, x)$, then $I$ meets the assumptions in the lemma above. The resulting abelian group $(N \times N)/\Delta(N)$ is exactly the Grothendieck construction of $N$.


\end{document}